\documentclass[11pt,a4paper]{article}
\usepackage[T1]{fontenc}
\usepackage[utf8]{inputenc}
\usepackage{amsmath}
\usepackage{mathtools}
\usepackage{amsfonts}
\usepackage{amsthm}
\usepackage{amssymb, comment}
\usepackage{bm}
\usepackage[colorlinks,unicode]{hyperref}
\usepackage{url}
\usepackage{color, xcolor}
\usepackage{graphicx, color, xcolor}
\usepackage{psfrag, float}  

\usepackage{verbatim}   
\usepackage{booktabs}	
\usepackage{grffile}    

\newtheorem{ansatz}{Ansatz}
\newtheorem{lemma}{Lemma}[section]
\newtheorem{theorem}[lemma]{Theorem}
\newtheorem{corollary}[lemma]{Corollary}
\newtheorem{definition}[lemma]{Definition}

\newtheorem{proposition}[lemma]{Proposition}

\theoremstyle{remark}
\newtheorem{remark}[lemma]{Remark}

\newcommand{\eps}{\varepsilon}
\newcommand{\ga}{\gamma}
\newcommand{\dps}{\displaystyle}
\newcommand{\RR}{\mathbb{R}}
\newcommand{\R}{\mathbb{R}}
\newcommand{\NN}{\mathbb{N}}
\newcommand{\CC}{\mathbb{C}}
\newcommand{\TT}{\mathbb{T}}
\newcommand{\ZZ}{\mathbb{Z}}
\newcommand{\MM}{\mathcal{M}}

\newcommand{\PP}{\mathcal{P}}
\newcommand{\GG}{\mathcal{G}}
\newcommand{\II}{\mathcal{I}}
\newcommand{\BB}{\mathcal{B}}
\newcommand{\LL}{\mathcal{L}}

\newcommand{\KK}{\mathcal{K}}
\newcommand{\SSS}{\mathcal{S}}
\newcommand{\DD}{\mathcal{D}}

\newcommand{\OO}{\mathcal{O}}
\newcommand{\FF}{\mathcal{F}}
\newcommand{\HH}{\mathcal{H}}

\newcommand{\RRR}{\mathcal{R}}

\newcommand{\UU}{\mathcal{U}}
\newcommand{\NNN}{\mathcal{N}}
\newcommand{\EE}{\mathcal{E}}
\newcommand{\JJ}{\mathcal{J}}
\newcommand{\E}{\mathbb{E}}
\newcommand{\AAA}{\mathcal{A}}

\newcommand{\CCC}{\mathcal{C}}
\newcommand{\Id}{\mathrm{Id}}

\newcommand{\ii}{^{-1}}
\newcommand{\de}{\delta}
\newcommand{\pa}{\partial}
\newcommand{\ee}{e_0}
\newcommand{\la}{\lambda}

\newcommand{\inn}{\mathrm{in}}
\newcommand{\out}{\mathrm{out}}

\newcommand{\kk}{\kappa}
\newcommand{\rr}{\rho}

\newcommand{\ol}{\overline}

\newcommand{\al}{\alpha}
\newcommand{\ero}{\zeta}
\newcommand{\om}{\omega}

\renewcommand{\Re}{\mathrm{Re\, }}
\renewcommand{\Im}{\mathrm{Im\,}}
\newcommand{\wt}{\widetilde}
\newcommand{\wh}{\widehat}

\newcommand{\g}{\hat g}
\newcommand{\Lb}{\Lambda}
\newcommand{\lb}{\lambda}
\newcommand{\ccirc}{\mathrm{circ}}
\newcommand{\eell}{\mathrm{ell}}

\newcommand{\SM}{\mathcal{SM}}
\newcommand{\PO}{\mathrm{PO}}
\newcommand{\Prob}{\mathrm{Prob}}
\newcommand{\iinn}{\mathrm{inn}}
\newcommand{\im}{\mathrm{i}}
\newcommand{\tM}{\mathtt{M}}
\newcommand{\bms}{\bm{\sigma}}

\newcommand{\ecc}{\mathbf{e}}
\newcommand{\tJ}{\mathbf{J}}

\newcommand{\fourmap}{\widetilde{\mathcal{P}}}

\newcommand{\Sec}{\Sigma}
\newcommand{\Poinc}{\mathcal{P}}

\newcommand{\splitting}{\theta}

\DeclareMathOperator{\atan}{atan}

\usepackage{hyperref}
\usepackage{subcaption} 

\usepackage{mathtools} 	

\usepackage{gnuplot-lua-tikz} 

\DeclarePairedDelimiter\abs{\lvert}{\rvert}%
\DeclarePairedDelimiter\norm{\lVert}{\rVert}%

\makeatletter
\let\oldabs\abs
\def\abs{\@ifstar{\oldabs}{\oldabs*}}
\let\oldnorm\norm
\def\norm{\@ifstar{\oldnorm}{\oldnorm*}}
\makeatother

\begin{document}
	
\title{Computation of a separatrix map and a normally hyperbolic invariant lamination for the RP3BP}
\author{
M. Guardia\footnote{Universitat de Barcelona \& Centre de Recerca Matem\`atica,
guardia@ub.edu}, \ \ 
V. Kaloshin\footnote{ISTA,
vadim.kaloshin@gmail.com},\ \  
P. Martin \footnote{Universitat Polit\`ecnica de Catalunya  \& Centre de Recerca Matem\`atica, p.martin@upc.edu},
\ \
P. Roldan, \footnote{Universitat Polit\`ecnica de Catalunya, pablo.roldan@upc.edu}
} 

\maketitle

\begin{abstract}
In this paper we discuss the existence of a normally hyperbolic invariant lamination (NHIL)
at the Kirkwood gap $3:1$ for the Restricted Planar Elliptic 3 Body 
Problem. 

This problem models the Sun--Jupiter--Asteroid dynamics. We also show that the induced dynamics on the NHIL is a partially hyperbolic skew-shift which is of the form  
		\[ f:(\omega,I,\theta)\to 
		(\sigma \omega,I+e_0 A_\omega(I)
		\cos (\theta+\psi_\omega)
		+\OO(e^2_0),\theta+\Omega_\omega(I)+\OO(e_0)),\]
		where $ I\in [a,b],\  \theta\in \mathbb T,\ \omega\in\Sigma=\{0,1\}^\mathbb Z$,  the space of sequences of $0,1$’s, $\sigma:\Sigma \to \Sigma$ is the shift in this space, 
		$\Omega_\omega$ is the shear,\ $A_\omega$ is an amplitude, and 
		$e_0$ is the eccentricity of Jupiter, which is taken as a small parameter.

In the companion paper \cite{stochastic23}, relying on these skew-shift, we show the existence of stochastic diffusing behavior for Asteroids belonging to the Kirkwood gap provided the eccentricity of Jupiter is $e_0$ small enough.

Key ingredients to construct the NHIL are the separatrix map associated to homoclinic channels to a normally hyperbolic invariant cylinder and an isolating block construction. Some of the necessary non-degeneracy conditions are verified numerically.
\end{abstract}

\tableofcontents

\section{Introduction}\label{sec:introduction}

The Restricted Planar Elliptic 3 Body Problem (RPE3BP) describes the motion of a body of negligible mass under the gravitational influence of two bodies (the primaries) that move on ellipses of eccentricity $e_0\in (0,1)$. Normalizing the mass of the system, one can assume that the bodies have mass $\mu\in (0,1/2]$ and $1-\mu$. 

If one takes $\mu = 0.95387536 \times 10^{-3}$, the RPE3BP models the Sun-Jupiter-Asteroid dynamics. One region  in the Solar system where asteroids are abundant is the Asteroid belt, which lies between the orbits of Jupiter and Mars. Since the mass of Mars is much smaller than the mass of Jupiter, its influence on the Sun-Jupiter-Asteroid system can be neglected. The distribution of the asteroids in the Asteroid belt is not uniform but rather uneven. In fact, if one considers the distribution of the Asteriods with respect to their semimajor axes, one can see that there are semimajor axes for which Asteroids are barely present (see Figure \ref{fig:distribution}). This are the so-called Kirkwood gaps, which coincide with semimajor axes such that the Asteroids' period is resonant with that of Jupiter (semimajor axis and period are related through the third Kepler law). Those are the so-called mean motion resonances which are, together with the secular resonances \cite{ClarkeFG22, ClarkeFG23}, fundamental in understanding the ``shape'' of our Solar system. Resonances are a major source of instabilities which can arise through different mechanisms. A possible instability mechanism to expalin the presence of the Kirkwood gaps, through slow-fast models and adiabiatic invariants, was proposed by Wisdom \cite{Wisdom82} and Neishtadt \cite{Neishtadt87} (see also \cite{NeishtadtS04}) for
the regime 
\[
\frac{\sqrt{\mu}}{e_0}\ll 1
.\]
	
	\begin{figure} [h!]
		\begin{center}
			\includegraphics[scale=0.64]{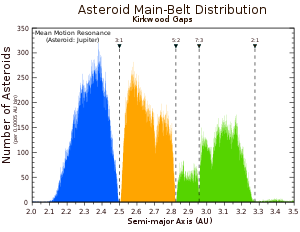}
			\caption{Distribution of asteroids in the Asteroid belt}
			\label{fig:distribution}
		\end{center}
	\end{figure}
	
In \cite{FejozGKR15}, the authors proposed an Arnold diffusion mechanism creating instabilities  at the $3:1$ Kirkwood gaps for $\mu = 0.95387536 \times 10^{-3}$ and $e_0>0$ small enough. This Arnold diffusion mechanism makes the asteroid's eccentricity drift, which implies that the asteroid gets closer to Mars orbit in a such  way that close encounters with Mars may expel the asteroid from the Asteroid belt.

The classical Arnold mechanism relies on a normally hyperbolic invariant cylinder and its stable and unstable invariant manifolds and, unfortunately, it leads to ``few'' diffusing orbits. If one wants a more global understanding of unstable motions one must rely on other hyperbolic objects such as \emph{normally hyperbolic invariant laminations} (NHIL). These are hyperbolic sets which are homeomorphic to a Cantor set times a cylinder.

The purpose of this paper is to build a NHIL along the $3:1$ mean motion resonance for the RPE3BP with $\mu = 0.95387536 \times 10^{-3}$ and $e_0>0$ small enough. This lamination lies in the vicinity of two homoclinic channels to the cylinder analyzed in \cite{FejozGKR15}. A key ingredient to analyze this NHIL is the so-called \emph{separatrix map}, which is an induced  return map from the neighborhood of the homoclinic channels to themselves.

Both the separatrix map and the NHIL are key ingredients in the companion paper \cite{stochastic23} to show that the Arnold diffusion orbits on the NHIL can be approximated by a 1-dimensional stochastic diffusion processes (at time scales $t\sim e_0^{-2}$), where randomness comes from the choice of initial condition  according to certain measure with support at the NHIL.

The  RPE3BP is a $2\frac{1}{2}$ degree of freedom Hamiltonian system given by 
\begin{equation} \label{eq:RPE}
\HH_\mu(y,x, t;e_0) =
\frac{\|y\|^2}{2}- \frac{1-\mu}{\|x + \mu x_0(t;e_0)\|}-
\frac{\mu}{\|x -(1- \mu) x_0(t;e_0)\|},
\end{equation}
where $x, y \in \R^2$ are the position and momentum of the Asteroid, $\|\cdot\|$ is the Euclidean distance and $x_0$ is the normalized position of the primaries (or “fictitious body”)
at time $t$, so that the Sun and Jupiter have respective positions 
$-\mu x_0(t;e_0)$ and $(1-\mu)x_0(t;e_0)$.  Without loss of generality one can assume that $x_0$ has  semi-major axis 1 and period $2\pi$, and then it can be written as 
\[x_0(t;e_0)=r(t;e_0)(\cos f(t;e_0), \sin f(t;e_0))\] 
where $r(t;e_0)$ is the distance between the primaries, given by
\[
r(t;e_0)=\frac{1-e_0^2}{1+e_0\cos f(t;e_0)},
\]
and $f(\cdot; e_0):\TT\to\TT$ is the mean anomaly which satisfies $f(0;e_0)=0$ and
\[
\frac{df}{dt}=\frac{(1+e_0\cos f)^2}{(1-e_0^2)^{3/2}}.
\]

For $e_0 \ge 0$ the RPE3BP has two and a half degrees of 
freedom. When $e_0 = 0$, the primaries describe uniform 
circular motions around their center of mass (indeed $f(t;0)=t$). This system is often 
called {\it the Restricted Planar Circular 3 Body Problem} and 
in what follows is abbreviated RPC3BP. In a frame rotating 
with the primaries, this system becomes autonomous and hence
has only two degrees of freedom. Its energy in the rotating frame 
is a first integral, called {\it the Jacobi integral or the Jacobi 
constant}. It is defined by
\begin{equation} \label{eq:Jacobi}
J(y,x) =\frac{\|y\|^2}{2}- \frac{1-\mu}{\|x + \mu (1,0)\|}-
\frac{\mu}{\|x -(1- \mu)(1,0) \|}-(x_1y_2-x_2y_1).
\end{equation}
If Jupiter performs circular motion, since the system has only two degrees of freedom, KAM invariant tori are $2$-dimensional and separate the $3$-dimensional energy surfaces (see \cite{Arn63, SM95}). This prevents the existence of Arnold diffusion. On the contrary,  the elliptic problem in rotating coordinates, which (identifying $\RR^2$ with $\CC$) becomes 
\begin{equation} \label{eq:RPE:rot}
\HH(y,x, t;e_0) =
\frac{\|y\|^2}{2}- \frac{1-\mu}{\|x + \mu x_0(t;e_0)e^{-it}\|}-
\frac{\mu}{\|x -(1- \mu) x_0(t;e_0)e^{-it}\|}-(x_1y_2-x_2y_1).
\end{equation}
is time dependent and therefore possesses a 5 dimensional phase space. Then, the existing $3$-dimensional KAM tori do not prevent orbits from wandering on a $5$-dimensional phase space and, therefore, Arnold diffusion is possible.

\subsection{Main results}


To state the main results we  assume an ansatz for the RPC3BP. Roughly speaking, the ansatz states that the RPC3BP is ``typical'' in certain sense. More precisely, that it   possesses a family of hyperbolic periodic orbits close to the 3:1 resonance with transverse homoclinic points (as happens typically at the resonances of nearly integrable 2 degrees of freedom Hamiltonian systems). 

\setcounter{ansatz}{-1}
\begin{ansatz}\label{ans:NHIMCircular:bis:intro}
	Consider the Hamiltonian of the RPC3BP~\eqref{eq:Jacobi} with
	 $\mu = 0.95387536 \times 10^{-3}$.  In every energy level $\tJ\in [\tJ_-, \tJ_+]$:
	\begin{itemize}
		\item   There exists a hyperbolic periodic orbit
		$\lb_\tJ$ of period $T_\tJ$ with
		\begin{equation}\label{def:periodestimate:intro}
		9\mu < \left| T_\tJ-2\pi\right|<15 \mu,
		\end{equation}
		such that its osculating semimajor axis satisfies 
		\[
		\left| a_\tJ(t)-3^{-2/3}\right|< 30\mu
		\]
		for all $t\in [0,T_\tJ]$.  The periodic orbit and its period
		depend analytically on $\tJ$.
		
		\item The stable and unstable invariant manifolds of every $\lb_\tJ$, $W^{s}(\lb_\tJ)$
		and $W^{u}(\lb_\tJ)$, intersect transversally at \emph{two primary homoclinic points}. Moreover, these homoclinic points depend analytically on $\tJ$ and their orbits are confined to the interval 
		\[
		\left| a-3^{-2/3}\right|< 56\mu.
		\]
	\end{itemize}
\end{ansatz}
Ansatz \ref{ans:NHIMCircular:bis} rephrases this ansatz in Delaunay coordinates, which are the ones used throughout the paper. Its validity is checked numerically in Appendix~\ref{sec:homocl_channels} for the  Jacobi constant interval $[\tJ_-, \tJ_+]=[-1.581,-1.485]$.  Regarding the estimates in \eqref{def:periodestimate:intro}, the exact lower and upper bounds are not important. What is crucial is that 
\[
0< \left| T_\tJ-2\pi\right|<\frac{\pi}{2},
\]
to avoid certain resonances.

Note that if one considers the RPC3BP in the 5-dimensional extended phase space (adding time as variable), the family of periodic orbits give rise to a 3-dimensional normally hyperbolic invariant cylinder foliated by 2-dimensional invariant tori (since one adds the angle $\dot t=1$). 
Next proposition provides a symplectic system of coordinates which straightens the cylinder, its invariant manifolds and the homoclinic channels for a suitable Poincar\'e map.

\begin{proposition}\label{prop:coordinates}
Fix an interval of Jacobi constant $[\tJ_-,\tJ_+]$ as in Ansatz \ref{ans:NHIMCircular:bis:intro} and $M>0$. Then,  there exists a canonical transformation 
\[
\Phi:\mathcal D=\left\{(I,s,p,q,g): I\in [I_-,I_+], (s,g)\in\TT^2,
|p|,|q|\leq M, |pq|\leq\rr\right\}\to (\mathbb{R}^4\times \TT)\cap \left\{J(x,y)\in[\tJ_-,\tJ_+]\right\},
\]
where $[I_-,I_+]=[-\tJ_+,-\tJ_-]$ such that the Hamiltonian $J\circ\Phi$ induces a  Poincar\'e map $\PP_0$ from $\{g=0\}$ to itself such that
%
\begin{itemize}
    \item $\PP_0$ has a normally hyperbolic invariant cylinder
    \[\wt\Lambda_0=\{(I,s,p,q):\,p=q=0,\, I\in [I_-,I_+],\, s\in\TT\}.\]
    \item The (local) stable invariant manifold of $\wt\Lambda_0$ is  $\{q=0\}$ and the (local) unstable manifold is given by  $\{p=0\}$.
    \item The two homoclinic channels to the cylinder are parameterized as 
 \begin{equation}\label{def:HomoChannels:intro}
\wt\CCC_0^i=\{ (I,s,p,q)=(I,s,0,q^i(I)),\, I\in [I_-,I_+],\, s\in\TT\} ,\qquad i=1,2
\end{equation}
for some smooth functions $q^i:[I_-,I_+]\to\RR$.
    
\end{itemize}

\end{proposition}

This proposition is proven in full detail in \cite{stochastic23}. To make the present paper self contained we include a description of the proof in Section \ref{sec:apriorichaotic}.

Now we define the separatrix maps associated to the Poincar\'e map $\PP_{e_0}$ induced by the Hamiltonian $\HH\circ \Phi$ (see \eqref{eq:RPE:rot}) onto the section $\{g=0\}$. Note that the objects provided by Proposition \ref{prop:coordinates} are persistent for $e_0>$ small enough thanks to normal hyperbolicity and transversality.
%

Consider open neighborhoods $\mathcal U^i\subset\{g=0\}$ of the  homoclinic channels. 
Then, we define the return time functions $N^{ij}:\UU^i\to \ZZ$ as
\begin{equation}\label{def:ReturnTime}
 N^{ij}(z)=\min\left\{n>0: \PP_{e_0}^n(z) \in \UU^j \right\}.
\end{equation}
Note that for a lot of points $N^{ij}$ may not  be defined. On the contrary, in the domains where they are finite, open non-empty sets which we denote by $\UU^{ij}\subset\UU^i,\ i,j\in\{1,2\}$, defined as
\[
\UU^{ij}=\left\{z\in\UU^i: N^{ij}(z)<\infty\right\},
\]
one can conisder the separatrix maps
\begin{equation}\label{def:SMij} \SM_{e_0}^{ij}:\UU^{ij}\to \UU^j\qquad \text{
as }\qquad  \SM_{e_0}^{ij}(z)=\PP_{e_0}^{N^{ij}(z)}(z).
\end{equation}
Note that $N^{ij}$ is locally constant. Thus, in suitable open subsets of $\UU^{ij}$ these maps are smooth.

To derive formulas for the separatrix maps in Theorem \ref{thm:FormulasSMii:intro} below, we restrict them
to strictly smaller domains $\UU_\rho^{ij}\subset\UU^{ij}$ defined as follows.
Fix $\de>0$ small and  $\kk>1$,  we define the subsets
\begin{equation}\label{def:SubDomainUij}
 \UU_\rho^{ij}\subset \left\{
(I,s,p,q)\in  \UU^{ij} :   I\in [I_-+\de,I_+-\de],\ 
s\in\TT,\
\rr^\kk\leq  p\leq \rr,\
|q-q^i(I)|\leq \rr\right\},\,\, i,j=1,2,
\end{equation}
 where $q^i$ are the homoclinic channels introduced in
\eqref{def:HomoChannels:intro}.

Theorem \ref{thm:FormulasSMii:intro} below provides the expansion of the separatrix map of the elliptic problem 
up to order 2 in $e_0$ and we analyze which $s$-harmonics each term has. To this end, we introduce the following notation.

\begin{definition}\label{not:harmonics} For every smooth function $f$ that is
$2\pi$-periodic in $s$, we define $\NNN (f )$ as the set of integers $k\in\ZZ$
such that the $k$-th harmonic of $f$ (possibly depending on other variables) is
non-zero.
\end{definition}

Next theorem summarizes Theorem~\ref{thm:FormulasSMii:circ}, that produces the separatrix map associated to the circular problem, and Theorem~\ref{thm:FormulasSMii}, where the separatrix map of the elliptic problem is written as a perturbation of the one of the circular problem.

\begin{theorem}\label{thm:FormulasSMii:intro}
	Assume Ansatz \ref{ans:NHIMCircular:bis:intro} and fix $\de>0$, $\kk>1$,  $M>0$. 
	There exists $0<\rr\ll 1$ and $e_0^*>0$
	such that for any $e_0\in (0,e_0^ *)$, there exists a system of
	coordinates
	$(\wh I,\wh s,\wh p,\wh
	q)$ on
	\[
	\mathcal D=\left\{(I,s,p,q,g): I\in [I_-+\de,I_+-\de], (s,g)\in\TT^2,
	|p|,|q|\leq M, |pq|\leq\rr\right\},
	\]
	satisfying
	\[
	(\wh I,\wh s,\wh p,\wh q)= (I,s,p,q)+\OO(e_0),
	\]
	such that, on the domain  $\UU_\rho^{ij}$ introduced in \eqref{def:SubDomainUij},
	which satisfies $\UU_\rho^{ij}\subset \mathcal D$,
	%
	%
	the separatrix map $\SM_{e_0}^{ij}:  \UU_\rho^{ij}\to \UU^j$  in
	\eqref{def:SMij},
	is well defined. Moreover, dropping the hats in the coordinates,
	 the separatrix map
	$\SM_{e_0}^{ij}=(\FF^{ij}_{I,e_0},\FF^{ij}_{s,e_0},\FF^{ij}_{p,e_0},\FF^{ij}
	_{q,e_0})$ is as follows.
		%
	\begin{itemize}
		\item The $(I,s)$ components are given by
		\[
		\begin{pmatrix}\FF^{ij}_{I,e_0}\\
			\FF^{ij}_{s,e_0}\end{pmatrix}=\begin{pmatrix}\FF^{ij}_{I,0}+e_0
			\FF^{ij}_{I,1}+e_0^2
			\FF^{ij}_{I,2}+\OO_{C^2}\left(e_0^3\log^3\rr\right)\\
			\FF^{ij}_{s,0}+e_0
			\FF^{ij}_{s,1}+\OO_{C^2}\left(e_0^2\log^3\rr\right)
		\end{pmatrix},
		\]
		where
        \[
\begin{pmatrix}
\FF_{I,0}^{ij}\\\FF_{s,0}^{ij}\end{pmatrix}=\begin{pmatrix}I\\
s+\al_i(I)+(\nu(I)+\pa_I F(I,pq))N^{ij}(I,p,q)+ D^i(I,p,q)\end{pmatrix},
\]
$F$ and $\nu$ are the functions introduced in
Lemma \ref{lemma:MoserNF}, $N^{ij}$ is the return time introduced in 
\eqref{def:ReturnTime}, which is locally constant and satisfies $
 N^{ij}(I,p,q)\sim |\log\rr|$
and $D^i$ is of the form
\[
D^i(I,p,q)=\, a^i_{10}(I)p+
a_{01}^i(I)\left(q-q^i(I)\right)+\OO_2\left(p,q-q^i(I)\right).
 \]
        
        and
        \[
		\FF^{ij}_{z,k}=\MM^{z,k}_i(I,s)+\MM^{z,k}_{i,pq}(I,s,p,q), \qquad z =I,s,\quad k=1,2,
		\]
		where
		the functions $\MM^{z,k}_i$ and $\MM^{z,k}_{i,pq}$ are $C^2$ and satisfy
		\[
		\MM^{I,l}_i=\OO_{C^2}(1),\quad  \MM^{s,1}_i=\OO_{C^2}(\log\rr), \quad
		\MM^{I,l}_{i,pq}=pq\OO_{C^2}(1),\quad  
		\MM^{s,1}_{i,pq}=pq\OO_{C^2}(\log\rr)
		\]
		and
		\[
		\NNN(\MM^{z,1}_{i})=\NNN(\MM^{z,1}_{i,pq})=\{\pm 1\},\,\,\,z=I,s\quad
		\text{and}\quad \NNN(\MM^{I,2}_i)=\NNN(\MM^{I,2}_{i,pq})=\{0,\pm 2\}.
		\]
		\item The $(p,q)$ components are of the form
			%
		\[
		\begin{pmatrix}\FF^{ij}_{p,e_0}\\
			\FF^{ij}_{q,e_0}\end{pmatrix}=\begin{pmatrix}\FF^{ij}_{p,0}+e_0
			\FF^{ij}_{p,1}+\OO_{C^2}\left(e_0^2\log\rr\right)\\
			\FF^{ij}_{q,0}+e_0
			\FF^{ij}_{q,1}+\OO_{C^2}\left(e_0^2\log\rr\right)
		\end{pmatrix},
		\]
		where 
\[
\begin{split}
& \begin{pmatrix}\FF_{p,0}^{ij}\\\FF_{q,0}^{ij}\end{pmatrix}=
\\
& \,\begin{pmatrix}
  e^{-\Theta^i(I,p,q)N^{ij}} &0\\
 0 &e^{\Theta^i(I,p,q)N^{ij}}
 \end{pmatrix}\left[\begin{pmatrix}p_0^i(I)\\0\end{pmatrix}+ \begin{pmatrix}
 b_{10}^i(I) &b_{01}^i(I)\\
 c_{10}^i(I) &c_{01}^i(I)
 \end{pmatrix}
\begin{pmatrix}
  p\\ q-q^i(I)
 \end{pmatrix}+\OO_2\left(p,q-q^i(I)\right)\right],
 \end{split}
\]
with 
\begin{equation*}
 \Theta^i(I,p,q)=\pa_2 F\left(I^i_*(I,p,q), 
p_*^i(I,p,q)q_*^i(I,p,q)\right),\end{equation*}
where $I_*^i, p_*^i, q_*^i$ are the images of the gluing map given in
Lemma~\ref{lemma:GluingOriginal} and $F$ is given by Lemma~\ref{lemma:MoserNF}.
Moreover,
%
the functions $a_*^i$, $b_*^i$, $c_*^i$ and $p_0^i$  satisfy
\begin{equation}\label{eq:GluingSymplectic:intro}
 b_{10}^ic_{01}^i-b_{01}^ic_{10}^i=1,\quad a_{10}^i=c_{10}^i\pa_Ip_0^i,\quad
a_{01}^i=c_{01}^i \pa_Ip_0^i
\end{equation}
and
\begin{equation}\label{eq:GluingTransversality:intro}
 c_{01}^i\neq 0.
\end{equation}
The functions $\FF^{ij}_{p,1}$, $\FF^{ij}_{q,1}$ satisfy
		$\FF^{ij}_{p,1},\FF^{ij}_{q,1}=\OO_{C^2}(\log\rr)$ and
		$\NNN(\FF^{ij}_{p,1})=\NNN(\FF^{ij}_{q,1})=\{\pm 1\}$.
	\end{itemize}
\end{theorem}
Proposition \ref{prop:Melnikov} below rewrites certain terms appearing in the separatrix maps in terms of Melnikov-type integrals. This is necessary to verify certain non-degeneracy conditions (see Ansatz \ref{ansatz:Melnikov:1} below) which are crucial in the companion paper \cite{stochastic23}.

The separatrix map given by Theorem~\ref{thm:FormulasSMii:intro} is used to prove the existence of a normally hyperbolic invariant lamination. That is a (weakly) invariant set which is homeomorphic to 
\[
\Sigma\times[I_-+\delta, I_+-\delta]\times \TT,\qquad \text{where}\quad \Sigma=\{0,1\}^\ZZ,
\]
and is normally hyperbolic (see Definition \ref{def:NHIL}). 

\begin{theorem}\label{thm:NHILElliptic:intro}
	Fix $0<\rr\ll 1$ and consider $\rr$-neighborhoods $B^j_\rr$ of the homoclinic 
	channels given in Proposition \ref{prop:coordinates}.
    Then, for $0\leq e_0\ll 1$, the  separatrix map $\SM_{e_0}$ given  in Theorem \ref{thm:FormulasSMii:intro} has a NHIL denoted by 
	$\LL_{e_0}$ which
    satisfies $\LL_{e_0}\subset B^1_\rr\cup B^2_\rr$. That 
	is,
	\begin{itemize}
		\item
		The set $\LL_{e_0}$ is weakly invariant: there exists an embedding
		\[
		L_{e_0}:\Sigma\times [I_-+\delta, I_+-\delta]\times\TT\to B^1_\rr\cup B^2_\rr\qquad L_{e_0}(\omega, I, s)=(I,s, P_{e_0}(\omega,I,s), Q_{e_0} (\omega,I,s))
		\]
		and functions $\beta$,  $S_{e_0}$ and $J_{e_0}$, which are 
		$C^3$ in $(I,s,e_0)$ and $\vartheta$-H\"older in $\omega$, for some $\vartheta\in (0,1)$ independent of $\rho$,  such that
		\begin{multline*}
			\SM_{e_0}\left(I,s, P_{e_0}(\omega,I,s), Q_{e_0}(\omega,I,s)\right) \\
			=\left(\II_{e_0}(\omega,I,s),\SSS_{e_0}(\omega,I,s), P_{e_0}\circ\FF_{\iinn}(\omega,I,s), Q_{e_0}\circ\FF_{\iinn}(\omega,I,s)\right)
		\end{multline*}
		with
		\begin{equation}\label{def:inner_map_elliptic}
			\FF_{\iinn}(\omega,I,s)=(\sigma\omega,  \II_{e_0}(\omega,I,s), \SSS_{e_0}(\omega,I,s)),
		\end{equation}
		where $\sigma:\Sigma\to\Sigma$ is the Bernouilli shift and
		\[
		\begin{split}
			\II_{e_0}(\omega,I,s)&=I+J_{e_0}(\omega,I,s)\\
			\SSS_{e_0}(\omega,I,s)&=s+\beta(\omega,I)+S_{e_0}(\omega,I,s).
		\end{split}
		\]
		\item The set $\LL_{e_0}$ is normally hyperbolic.
	\end{itemize}
	Moreover,
	there exists $\zeta'>\zeta_\ast>0$ such that
		\[
		\rr^{\zeta'}\lesssim  \inf_{z\in \LL_0, \ z'\in \wt\CCC^i_0,\ i=1,2}\mathrm{ dist}(z,z')\leq  \sup_{z\in \LL_0, \ z'\in \wt\CCC^i_0,\ i=1,2}\mathrm{ dist}(z,z')\lesssim \rr^{\zeta_\ast}.
		\]
\end{theorem}

This theorem is proven in two steps in Section \ref{sec:NHIL}. First, we prove Theorem \ref{thm:NHILCircular}, which provides the existence of the NHIL for the circular problem ($e_0=0$). Note that when $e_0 =0$,  $I$ is a constant of motion. Then, Ansatz \ref{ans:NHIMCircular:bis:intro} and Birkhoff-Smale Theorem leads to the existence of a Smale  horseshoe for each $I$ for the circular problem, that depends regularly on $I$ (see Section~\ref{sec:NHILCircular} for the detailed construction). This gives rise to a Normally Hyperbolic Invariant Lamination in the extended phase space. Then, in Theorem \ref{thm:NHILElliptic}, we prove the existence of  the NHIL for the elliptic problem as a regular perturbation of the NHIL of the circular one. The dynamics of the NHIL is given by the separatrix map provided by Theorem \ref{thm:FormulasSMii:intro}.


\subsection{Literature}
As already mentioned and explained in more detail in  Section \ref{sec:stochastic} below, Theorems \ref{thm:FormulasSMii:intro} and \ref{thm:NHILElliptic:intro} are key steps to show that at certain time scales the Arnold diffusion orbits at mean motion resonances behave as a diffusion process. We refer to the companion paper \cite{stochastic23} for references both on the connection between Arnold diffusion and Ito processes and on the astronomical motivations of this work. We focus here on the previous literature on the analysis of separatrix maps and the use of NHILs in Arnold diffusion.

The analysis of separatrix maps was initiated by Zaslavsky and Filonenko \cite{Zaslavsky} for near-integrable Hamiltonian systems of 
 one-and-a-half degrees of freedom and independently by Shilnikov \cite{Shilnikov} for the unfoldings of  some bifurcations. It was first used in the study of Arnold diffusion by Treschev in \cite{Treschev02a} (see also \cite{Piftankin06, PiftankinT07, GuardiaK15}). In proving Theorem \ref{thm:FormulasSMii:intro}, we follow closely the tools developed in  \cite{Treschev02a}. However, note that \cite{Treschev02a} deals with a priori unstable Hamiltonian systems and therefore the separatrix map can be described rather globally in the stochastic zone. On the contrary, in the present paper we are dealing with an a priori chaotic model and therefore we just analyze the separatrix map in suitable neighborhoods of two homoclinic channels, more in the spirit of the seminal Shilnikov paper \cite{Shilnikov} and as happens in \cite{Piftankin06}. Furthermore, since in the companion paper we want to analyze stochastic behavior at time scales $t\sim e_0^{-2}$, in Theorem \ref{thm:FormulasSMii:intro} we analyze the asymptotic expansion of the separatrix map up to order 2 in $e_0$ (as was done in \cite{GuardiaK15}).

Normally Hyperbolic Invariant Laminations were first constructed by Hirsch, Pugh and Shub in \cite{HirschPS77} and first used in Arnold diffusion phenomena by Gelfreich and Turaev in \cite{Gelfreich:2008}. They have also been used to obtain optimal stability time estimates by Bounemoura and Pennamen \cite{BounemouraPennamen} and more recently to relate Arnold diffusion and diffusive processes for a priori unstable nearly integrable Hamiltonian systems in \cite{KaloshinZ15}. The proof in the present paper relies on the isolating block argument developed in \cite{KaloshinZ15}.

\subsection{Stochastic Arnold diffusion}\label{sec:stochastic}
Theorems \ref{thm:FormulasSMii:intro} and \ref{thm:NHILElliptic:intro} are key steps in proving the  stochastic diffusive behavior of certain Arnold diffusion orbits. Roughly speaking, \cite{stochastic23} shows that Arnold diffusion orbits behave as an Ito diffusion at time scales $t\sim e_0^{-2}$, where randomness comes from the choice of initial condition according to a measure whose support lies on a normally hyperbolic invariant lamination (see Theorems~1.1 and~1.2 in \cite{stochastic23} for the precise statements). 



The proof is achieved through a perturbative analysis. We start with the Hamiltonian of the 2~body problem 
\[
\HH_0(y,x) =\frac{\|y\|^2}{2}-\frac{1}{\|x\|},
\]
and we perform two perturbations:
\begin{itemize}
\item From the 2 body problem Hamiltonian $\HH_0$ to the circular Hamiltonian 
\[
\HH_{\ccirc,\mu}( y,x, t) =
\frac{\|y\|^2}{2}- \frac{1-\mu}{\|x + \mu x_0(t,0)\|}-
\frac{\mu}{\|x -(1- \mu) x_0(t,0)\|}.
\]
\item From the circular Hamiltonian $\HH_{\ccirc,\mu}$ to the elliptic Hamiltonian $\HH_\mu$ (see \eqref{eq:RPE}). 
\end{itemize}

Let us explain the main steps of the proof of the main results in \cite{stochastic23} and what is the contribution of the present paper to them.
\begin{enumerate}
\item 
For the 2 body problem the Kirkwood gap $3:1$ is defined by 
the semimajor axis $a=3^{-2/3}$, where $2a=\HH_0^{-1}$. For each 
eccentricity $e \in [0,1)$, there is a one parameter family of 
periodic orbits, each repesented by an ellipse with semimajor 
axis $a$ and eccentricity $e$. It turns out that, at the Kirkwood 
gap, $e$ is an implicit function of $\tJ$.

\item Consider the Hamiltonian of the circular problem $\HH_{\ccirc,\mu}$,
which is an $O(\mu)$-pertur\-bation of $\HH_0$ (recall that we are taking $\mu = 0.95387536 \times 10^{-3}$). Ansatz \ref{ans:NHIMCircular:bis:intro}, which is verified numerically in Appendix \ref{sec:homocl_channels}, assumes that
\begin{enumerate}
\item  There exists a family of saddle periodic orbits $p_\tJ$ of the RPC3BP  at the $3:1$ mean motion resonance
parametrized by the Jacobi constant $\tJ\in [\tJ_-,\tJ_+]$ and 
\item For values $\tJ\in [\tJ_-,\tJ_+]$,
			the associated stable and unstable invariant manifolds
			$W^s(p_\tJ)$ and $W^u(p_\tJ)$ intersect transversally (within the Jacobi constant level) at two distinct homoclinic orbits. Each transverse intersection 
			gives rise to a homoclinic channel.
\end{enumerate}

\item For each homoclinic channel, we determine a collection of open 
sets $\{\mathcal U_n\}_{n\ge n_0}$ for some $n_0$ large and 
such that the separatrix  map $\SM_{e_0}$ is well defined and 
consists of $n$ iterates of the Poincar\'e map onto de section $\{g=0\}$, when restricted to $\mathcal U_n$. This is shown in Theorem \ref{thm:FormulasSMii:intro}, which also provides its $e_0$-expansion, and proven in Section~\ref{sec:separatrix}. 

\item We prove that the separatrix map $\SM_{e_0}$ restricted to these 
open sets $\{\mathcal U_n\}_n$ is  partially hyperbolic and, using 
an isolating block technique, construct a normally hyperbolic 
lamination $\Lb$. The existence of this lamination is provided by Theorem \ref{thm:NHILElliptic:intro} and proven in Section \ref{sec:NHIL}.

\begin{figure} [h!]
\begin{center}
\includegraphics[scale=0.7]{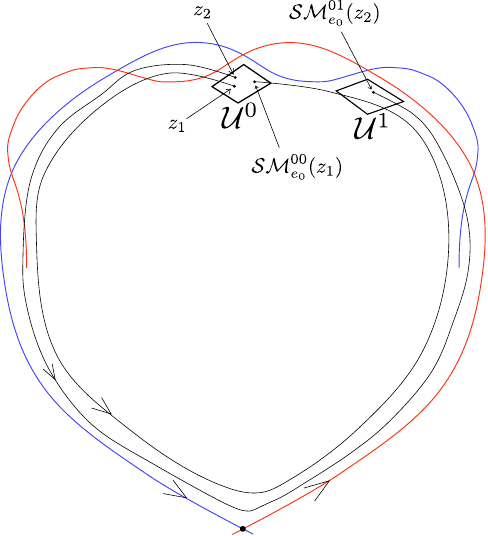}
\caption{The separatrix map, defined on open sets in $\mathcal{U}_0$ and $\mathcal{U}_1$, close to the two homoclinic
channels.}
\label{fig:separatrix-map}
\end{center}
\end{figure}

\item  We derive a normal form for the dynamics of $\SM_{e_0}$ restricted 
to the lamination $\Lb$ such that it is conjugate to 
a skew shift  model. This is done in Lemma 5.1 and  Proposition 5.7 of \cite{stochastic23}. This normal form requires certain non-degeneracy conditions which are provided by Ansatz \ref{ansatz:Melnikov:1} below and verified numerically in Appendix \ref{sec:nondegeneracy}.

\item To prove stochastic diffusive behavior for the skew
		shift model, we analyze first a ``reduced model'': a Lie group (circle) extension of hyperbolic maps, for which one can prove decay of correlations and central limit theorems for certain measures. For some of these measures (called Type 2 in \cite{stochastic23}) these properties were proven in \cite{FieldMT03} and  for others (called Type 1 in \cite{stochastic23})  are proven in Theorem 5.4 of \cite{stochastic23}.
		
		\item Last step is to prove the convergence to the It\^o process through a martingale analysis. Such analysis is provided by Lemma 5.10 in \cite{stochastic23}.
\end{enumerate}


\section{The a priori chaotic framework}\label{sec:apriorichaotic}
\subsection{A good system of coordinates and a time reparameterization}
In this section we reproduce Section 2 of \cite{stochastic23} and we write the RPE3BP Hamiltonian $\HH_\mu$, defined in
\eqref{eq:RPE} in the classical Delaunay coordinates $(L,\ell,G, \g,t)$, which are defined as follows. Let the polar coordinates $r$ and $\varphi$ be defined as
\[
 q=(r\cos \varphi,r\sin\varphi),
\]
we introduce,
\begin{itemize}
 \item $L$ as the square of the semimajor axis of the instantaneous ellipse, which satisfies
 \[
  -\frac{1}{2L^2}=\frac{\|p\|^2}{2}-\frac{1}{\|q\|}.
 \]
 \item $G=q\times p$ is the angular momentum. The eccentricity of the ellipse is defined  through $L$ and $G$ by
 \[
  \ecc=\sqrt{1-\frac{G^2}{L^2}}.
 \]
 \item $\g$ is the argument of the pericenter.
 \item $\ell$ is the mean anomaly defined as follows. The true anomaly $v$ is the angle of the body with respect to the perihelion measured from the focus of the ellipse. That is, $\varphi=v+\g$. From $v$, one can compute the  eccentric anomaly $u$ by
 \[
  \tan\frac{v}{2}=\sqrt{\frac{1+\ecc}{1-\ecc}}\tan\frac{u}{2}
 \]
and from it the mean anomaly $\ell$  by the Kepler equation
\[
 \ell=u-\ecc\sin u.
\]
\end{itemize}
In these coordinates, the Hamiltonian \eqref{eq:RPE} becomes of the form
\begin{equation}\label{def:HamDelaunayNonRot}
\hat H(L,\ell,G, \g-t,t)=
-\frac{1}{2L^2}+\mu\Delta H_\ccirc(L,\ell,G,
  \g-t,\mu)
  + \mu e_0\Delta H_\eell(L,\ell,G, \g-t,t,\mu,e_0),
\end{equation}
where $e_0\in(0,1)$ is the primaries' eccentricity. 
Define the new angle $g=\g -t$ (the argument of the pericenter, measured in
the rotating frame) and a new variable $I$ conjugate to time $t$.
Then we have
\begin{equation}\label{def:HamDelaunayRot}
  H(L,\ell,G, g,I,t)=
  -\frac{1}{2L^2}-G+\mu\Delta H_\ccirc(L,\ell,G, g,\mu)
  +
  \mu e_0 \Delta H_\eell(L,\ell,G, g,t,\mu,e_0)+I.
\end{equation}
Without loss of generality we can restrict our study to $H=0$.

We consider the RPE3BP as a perturbation of the circular one, i.e.  $e_0=0$,
\begin{equation}\label{def:HamDelaunayCirc}
  H_\ccirc(L,\ell,G, g)=-\frac{1}{2L^2}-G+\mu \Delta H_\ccirc(L,\ell,G, g,\mu).
\end{equation}
Now we perform a Poincar\'e-Cartan reduction so that  $g$ becomes the new time. 
To this end, it is enough to consider
a Hamiltonian $K$ such that\footnote{Note that we change the order of the variables so that the new time $g$ is at the end.}
\[
 H(L,\ell,-K(L,\ell,I,t, g), g,I,t)=0.
\]
Note that it can be rewritten as
\begin{equation}\label{def:HamEll:Reparam}
 K(L,\ell,I,t,g)=-\frac{1}{2L^2}-I+\mu\Delta K_\mathrm{circ}(L,\ell, I,g)+\mu
e_0 \Delta K_\eell(L,\ell,I,t,g;e_0).
\end{equation}
We also define
\begin{equation}\label{def:HamCirc:Reparam}
 K_\mathrm{circ}(L,\ell, I,g)=-\frac{1}{2L^2}-I+\mu\Delta
K_\mathrm{circ}(L,\ell, I,g).
\end{equation}
It can be easily checked that the Hamiltonian equations associated to \eqref{def:HamEll:Reparam}, coincide with
\begin{equation}\label{def:Reduced:ODE}
	\begin{array}{rlcrl}
		\frac{d}{ds} \ell=&\dps\frac{\pa_L H}{-1+\mu\pa_G\Delta H_\ccirc +\mu
			e_0\pa_G\Delta H_\eell}&\text{   }&\frac{d}{ds} L=&\dps-\frac{\pa_\ell
			H}{-1+\mu\pa_G\Delta H_\ccirc +\mu e_0\pa_G\Delta H_\eell}\\
		\frac{d}{ds} g=&1&\text{   }&\frac{d}{ds} G=&\dps-\frac{\pa_g
			H}{-1+\mu\pa_G\Delta H_\ccirc +\mu e_0\pa_G\Delta H_\eell}\\
		\frac{d}{ds} t=&\dps\frac{1}{-1+\mu\pa_G\Delta H_\ccirc +\mu e_0\pa_G\Delta
			H_\eell}&\text{   }&\frac{d}{ds} I=&\dps-\frac{\mu e_0\pa_t \Delta
			H_\eell}{-1+\mu\pa_G\Delta H_\ccirc +\mu e_0\pa_G\Delta H_\eell}
	\end{array}
\end{equation}
and, in particular, when $e_0=0$, the equations associated to  $K_\mathrm{circ}$ are given by
\begin{equation}\label{def:Reduced:ODE:circular}
	\begin{array}{rlcrl}
		\frac{d}{ds} \ell=&\dps\frac{\pa_L H_\ccirc}{-1+\mu\pa_G\Delta H_\ccirc}&\text{
		}&\frac{d}{ds} L=&\dps-\frac{\pa_\ell
			H_\ccirc}{-1+\mu\pa_G\Delta H_\ccirc }\\
		\frac{d}{ds} g=&1&\text{   }&\frac{d}{ds} G=&\dps-\frac{\pa_g
			H_\ccirc}{-1+\mu\pa_G\Delta H_\ccirc}\\
		\frac{d}{ds} t=&\dps\frac{1}{-1+\mu\pa_G\Delta H_\ccirc}&\text{
		}&\frac{d}{ds} I=&\dps0,
	\end{array}
\end{equation}
whose right hand side is $t$ independent.


\subsection{The invariant cylinder and the associated homoclinic channels}
We now describe the invariant objects that we consider and also define further changes of coordinates in suitable neighborhoods of them.
Relying on numerical computations in
Appendix~\ref{sec:homocl_channels} 
we assume the following ansatz, which is a rephrasing on Ansatz \ref{ans:NHIMCircular:bis:intro} in Delaunay coordinates.

%

\begin{ansatz}\label{ans:NHIMCircular:bis}
	Consider the Hamiltonian \eqref{def:HamDelaunayCirc} with
	$\mu = 0.95387536 \times 10^{-3}$.  In every energy level $\tJ\in [\tJ_-, \tJ_+]$:
	\begin{itemize}
		\item   There exists a hyperbolic periodic orbit
		$\lb_\tJ=(L_\tJ(t), \ell_\tJ(t), G_\tJ(t), g_\tJ(t))$ of period $T_\tJ$ with
		\begin{equation}\label{def:periodestimate}
		9\mu < \left| T_\tJ-2\pi\right|<15 \mu,
		\end{equation}
		such that
		\[
		\left| L_\tJ(t)-3^{-1/3}\right|< 19\mu
		\]
		for all $t\in [0,T_\tJ]$.  The periodic orbit and its period
		depend analytically on $\tJ$\footnote{The analyticity with 
		respect to $\tJ$ is just a consequence of the analyticity of the Hamiltonian 
\eqref{def:HamDelaunayCirc}.}.
		
		\item The stable and unstable invariant manifolds of every $\lb_\tJ$, $W^{s}(\lb_\tJ)$
		and $W^{u}(\lb_\tJ)$, intersect transversally at \emph{two primary homoclinic points}. Moreover, these homoclinic points depend analytically on $\tJ$ and their orbits are confined in the interval 
		\[
		\left| L-3^{-1/3}\right|< 42\mu.
		\]
	\end{itemize}
\end{ansatz}
\begin{figure} [h!]
	\begin{center}
		\includegraphics[scale=0.59]{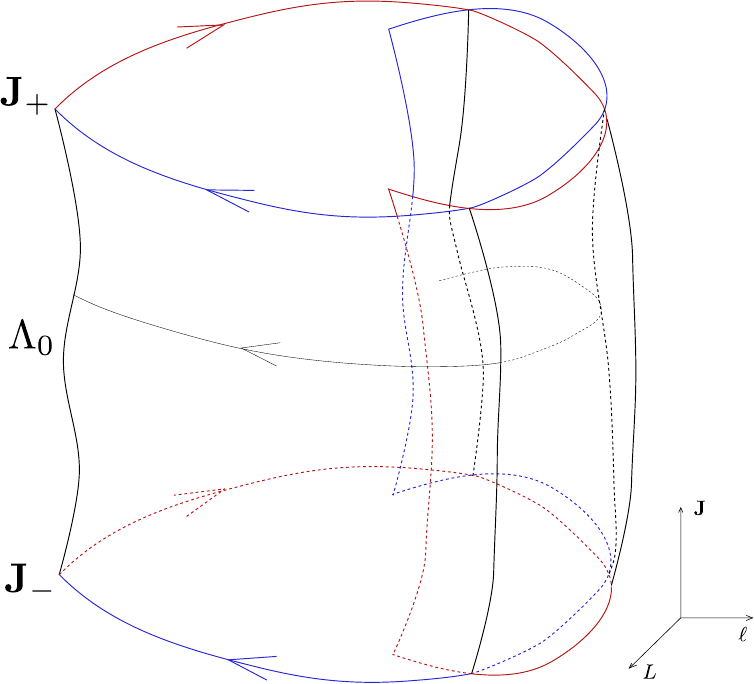}
		\caption{The periodic orbits provided by Ansatz \ref{ans:NHIMCircular:bis} give rise to a normally hyperbolic invariant cylinder, as stated in Corollary \ref{coro:NHIMCircular}. It posseses two transverse homoclinic channels.}\label{fig:NHIMCircular}
	\end{center}
\end{figure}
\medskip
This ansatz is  similar (in fact stronger) to that in \cite{FejozGKR15}. Later we impose an additional Ansatz (see Ansatz \ref{ansatz:Melnikov:1}). In 
Appendix \ref{sec:homocl_channels} 
we check that it is satisfied in the Jacobi constant interval $[\tJ_-, \tJ_+]=[-1.581,-1.485]$.




Ansatz \ref{ans:NHIMCircular:bis} and the the fact that 
the right hand side of the \eqref{def:Reduced:ODE:circular} is $t$ independent,  implies
that there exist analytic functions $(\ell_0, L_0)$ such that the $(\ell, L, g, G)$ components of
	\eqref{def:Reduced:ODE:circular}
have a periodic orbit for each
\[
I_0\in [I_-,I_+]=[-\tJ_+,-\tJ_-].
\]
that  can be parameterized as
\begin{equation}\label{eq:PeriodicOrbitsReparam}
 \begin{split}
 \ell&=\ell_0(I,g)\\
 L&=L_0(I,g).
 \end{split}
\end{equation}

\begin{corollary}\label{coro:NHIMCircular}
  Assume Ansatz~\ref{ans:NHIMCircular:bis} holds. The system
\eqref{def:Reduced:ODE:circular} with $\mu = 0.95387536 \times 10^{-3}$, associated to $K_\mathrm{circ}$,    has an
analytic normally hyperbolic\footnote{See Definition \ref{def:NHIM} 
for the precise definition of Normally Hyperbolic Invariant Manifold.} 
invariant $3$-dimensional cylinder $\Lambda_0$, which is foliated 
by $2$-dimensional invariant tori.

The cylinder $\Lambda_0$ has  stable and
unstable invariant manifolds, denoted $W^{s}(\Lambda_0)$ and
$W^{u}(\Lambda_0)$. In the 
invariant  planes $I=I_0$,  for every $I_0\in [I_-,I_+]$ 
the manifolds $W^{s}(\Lambda_0)$ and $W^{u}(\Lambda_0)$
intersect transversally at two homoclinic channels $\CCC^1_0$
and $\CCC^2_0$.
\end{corollary}

The next two lemmas, proven in \cite{stochastic23}, perform changes of coordinates  to Hamiltonian $K_\mathrm{circ}$. The first of them places the periodic orbits at the origin and the second one is the classical Moser Normal Form \cite{Moser56} in a neighborhood around them.

\begin{lemma}\label{lemma:ChangePO}
There exists an analytic $g$--time dependent symplectic change of coordinates
$(\wh L,\wh\ell, I,\wh t, g)=\Psi(L,\ell,t,I,g)$, of the form
\[
 \begin{split}
 \wh L&=L-L_0(I,g)\\
 \wh \ell&=\ell-\ell_0(I,g)\\
 I&=I\\
 \wh t&=t+T(L,\ell,I,g)
  \end{split}
\]
and periodic in $g$ such that the periodic orbit
\eqref{eq:PeriodicOrbitsReparam} is translated to
$(\wh L,\wh\ell, I)=(0,0,I)$ and the function $T$ satisfies
 \begin{equation}\label{def:0onPO}
  T(L_0(I,g),\ell_0(I,g),I,g)=0.
 \end{equation}
\end{lemma}


\begin{lemma}\label{lemma:MoserNF}
Consider the Hamiltonian $K_\mathrm{circ}$  in
\eqref{def:HamCirc:Reparam} and the change of
coordinates $\Psi$ introduced in Lemma \ref{lemma:ChangePO}.
Then, there exists an  analytic  $g$--time dependent  symplectic change of 
coordinates
$(I,s, p,q, g)=\Upsilon(\wh L,\wh\ell,\wh I,\wh t, g)$ of
the form 
\[
 \begin{split}
 I&=I\\
 s&=\wh t+B(\wh L,\wh \ell,\wh  I)
 \end{split}\qquad \quad\begin{split}
 p&=P(\wh L,\wh \ell,\wh  I,g)\\
 q&=Q(\wh L,\wh \ell,\wh  I,g),
  \end{split}
\]
such that
\[
 B(0,0,\wh  I,g)=0
\]
and
the
Hamiltonian $K_\mathrm{circ}\circ \Psi^{-1}\circ\Upsilon^{-1}$ is of the form
\begin{equation}\label{def:HamCircFinal}
\KK_0(I,p,q)= K_\mathrm{circ}\circ \Psi^{-1}\circ\Upsilon^{-1}(I,p,q)=E(I)+F(I,pq),
\end{equation}
where the frequency
\begin{equation}\label{def:nu}
\nu(I)=\pa_IE(I)
\end{equation}
 satisfies
\begin{equation}\label{def:Twist}
 \frac{9}{2\pi}\mu\leq\left|\nu(I)-1\right|\leq\frac{15}{2\pi}\mu
\end{equation}
and $F$ satisfies $F(I,pq)=\la(I)pq+\OO_2(pq)$, where $\la(I)$ is the positive
eigenvalue of the  periodic orbit
$(p,q)=(0,0)$ given by Ansatz \ref{ans:NHIMCircular:bis}.
\end{lemma}
The next corollary rephrases Corollary \ref{coro:NHIMCircular} in the new system of coordinates.

\begin{corollary}\label{coro:NHIMCircular:MoserCoord}
Assume Ansatz~\ref{ans:NHIMCircular:bis} holds. Then the Hamiltonian \eqref{def:HamCircFinal} with $\mu = 0.95387536 \times 10^{-3}$  has an analytic normally 
hyperbolic invariant $3$-dimensional cylinder
\[
\Lambda_0=\left\{(p,q)=(0,0),\ I\in [I_-,I_+],\ (g,s)\in\TT^2\right\},
\]
 which is foliated by $2$-dimensional invariant tori.

Moreover, for every $I_0\in [I_-,I_+]$, in the  invariant
planes $I=I_0$,  
the invariant manifolds $W^{s}( \Lambda_0)$ and $W^{u}( \Lambda_0)$ 
intersect transversally at two homoclinic channels $\CCC^1_0$ and
$\CCC^2_0$. Their intersection with $\{g=0\}$ are 
parameterized as
\begin{equation}\label{def:HomoChannels}
 (I,s,p,q)=(I,s,0,q^1(I)) \qquad \text{ and }\qquad  (
I,s,p,q)=(
I,s,0,q^2(I))
\end{equation}
for some smooth functions $q^i:[I_-,I_+]\to\RR$ and $s\in\TT$, respectively.
\end{corollary}


To analyze the elliptic  problem for  $e_0>0$ small enough, we start by expressing the elliptic  Hamiltonian
perturbation $\Delta K_\eell$ in \eqref{def:HamEll:Reparam} in the new coordinates. We define
\begin{equation}\label{def:EllipticPerturbMoser}
 \Delta \KK_\eell=\Delta K_\eell\circ \Psi^{-1}\circ\Upsilon^{-1},
\end{equation}
where $\Psi$ and $\Upsilon$ are the symplectic changes of coordinates introduced in Lemmas
\ref{lemma:ChangePO} and 
\ref{lemma:MoserNF} respectively.
We
 define  the transformed elliptic Hamiltonian
\begin{equation}\label{def:EllipticMoser}
\KK=K\circ \Psi^{-1}\circ\Upsilon^{-1}=\KK_0+\mu
e_0 \Delta \KK_\eell
\end{equation}
(see \eqref{def:HamCircFinal}).

Classical perturbation theory ensures the
persistence of the structures provided by Corollary \ref{coro:NHIMCircular:MoserCoord} for $e_0>0$ small enough. In particular, we denote by  $\Lambda_{e_0}$ the perturbed cylinder, which is $e_0$-close to the cylinder
$\Lambda_0$ defined in Corollary \ref{coro:NHIMCircular:MoserCoord}.
Analogously, we can define the cylinder for the Poincar{\'e} map
$\PP_{e_0}$ associated to this Hamiltonian and the section $\{g=0\}$, given by
$\wt\Lambda_{e_0}=\Lambda_{e_0}\cap \{g=0\}$. (Recall that changes performed
in Lemmas \ref{lemma:ChangePO} and \ref{lemma:MoserNF} have not modified $g$ and
thus the section is defined independently of the system of coordinates).


By Corollary \ref{coro:NHIMCircular}, the Poincar\'e map $\PP_0$ of the RPC3BP  has transverse homoclinic channels 
\begin{equation}\label{def:channelsPoincarecirc}
\wt\CCC_0^i=\CCC_0^i\cap\{g=0\}, \qquad\ i=1,2,
\end{equation}
for $I\in [I_-,I_+]$. By a transversality argument, for
any predetermined $\de>0$ and  $e_0>0$ small enough,
the elliptic problem also possesses 
transverse homoclinic channels
$\wt\CCC^i_{\ee},\ i=1,2$ in slightly smaller
domains $I\in [I_-+\de,I_+-\de]$, which are $\OO(e_0)$-close to $\wt\CCC_0^i$. This is summarized in the next theorem.

\begin{theorem}\label{th:Elliptic:NHIM}
Assume Ansatz~\ref{ans:NHIMCircular:bis} holds.
Let $\PP_{e_0}$ be the Poincar{\'e} map associated to 
the Hamiltonian $\KK$ in \eqref{def:EllipticMoser} and the section $\{g=0\}$.
For any $\de>0$, there exists $e_0^\ast>0$ such that for 
$e_0\in (0,e_0^\ast)$ the map $\PP_{e_0}$  has a  normally  hyperbolic  
locally invariant manifold  $\wt\Lambda_{e_0}$, which is $e_0$-close in 
the $C^1$-topology to the unperturbed cylinder  
$\wt\Lambda_0=\Lambda_0\cap \{g=0\}$ of  the circular problem and 
can be written as a graph over it. Namely, there exists a function 
$\GG_{e_0}: [I_-+\de,I_+-\de]\times\TT\rightarrow \RR^3\times\TT$ such that
  \[
  \wt\Lambda_{e_0}=\left\{ \GG_{e_0}(I,s):(I,s)\in[I_-+\de,I_+-\de]\times
\TT\right\}.
  \]

Moreover, in the region $I\in [I_-+\de,I_+-\de]$  the invariant
manifolds $W^{u}(\wt\Lambda_{e_0})$
(resp. $W^{s}(\wt\Lambda_{e_0})$) intersect transversally
along homoclinic channels $\wt\CCC^i_{\ee},\ i=1,2$
which are $\ee$-close to the homoclinic channels
$\wt\CCC^i_{0},\ i=1,2$ respectively.
\end{theorem}


\section{The separatrix map for the circular and   elliptic problems}\label{sec:separatrix}
We devote this section to prove Theorem \ref{thm:FormulasSMii:intro}. We first split it in two. Theorem \ref{thm:FormulasSMii:circ} below provides  formulas for
$\SM_{e_0}^{ij}$ with $e_0 = 0$ (that is for the RPC3BP). Then, in Theorem
\ref{thm:FormulasSMii}, we give formulas for $\SM_{e_0}^{ij}$ for $0<e_0\ll 1$. Finally Proposition \ref{prop:Melnikov} rewrites certain terms in the separatrix maps in terms of Melnikov-type integrals. This is crucial to verify numerically that they satisfy certain non-degeneracy conditions (see Ansatz \ref{ansatz:Melnikov:1} below).

\begin{theorem}\label{thm:FormulasSMii:circ}
Assume  Ansatz \ref{ans:NHIMCircular:bis} and fix $\de>0$, $\kk>1$. There 
exists $0<\rr\ll 1$
such that, on the domain $ \UU_\rho^{ij}$ in \eqref{def:SubDomainUij},
%
%
the separatrix map $\SM_0^{ij}:  \UU_\rho^{ij}\to \UU^j$  in
 \eqref{def:SMij}
is well defined and its components 
$\SM_0^{ij}=(\FF^{ij}_{I,0},\FF^{ij}_{s,0},\FF^{ij}_{p,0},\FF^{ij}_{q,0})$ 
have the following form.
%
\begin{itemize}
\item The $(I,s)$ components are given by
\[
\begin{pmatrix}
\FF_{I,0}^{ij}\\\FF_{s,0}^{ij}\end{pmatrix}=\begin{pmatrix}I\\
s+\al_i(I)+(\nu(I)+\pa_I F(I,pq))N^{ij}(I,p,q)+ D^i(I,p,q)\end{pmatrix},
\]
where  $F$ and $\nu$ are the functions introduced in
Lemma \ref{lemma:MoserNF}, $N^{ij}$ is the return time introduced in 
\eqref{def:ReturnTime}, which is locally constant and satisfies 
\[
 N^{ij}(I,p,q)\sim |\log\rr|
\]
and $D^i$ is of the form
\[
D^i(I,p,q)=\, a^i_{10}(I)p+
a_{01}^i(I)\left(q-q^i(I)\right)+\OO_2\left(p,q-q^i(I)\right).
 \]
\item The $(p,q)$ components are of the form
\[
\begin{split}
& \begin{pmatrix}\FF_{p,0}^{ij}\\\FF_{q,0}^{ij}\end{pmatrix}=
\\
& \,\begin{pmatrix}
  e^{-\Theta^i(I,p,q)N^{ij}} &0\\
 0 &e^{\Theta^i(I,p,q)N^{ij}}
 \end{pmatrix}\left[\begin{pmatrix}p_0^i(I)\\0\end{pmatrix}+ \begin{pmatrix}
 b_{10}^i(I) &b_{01}^i(I)\\
 c_{10}^i(I) &c_{01}^i(I)
 \end{pmatrix}
\begin{pmatrix}
  p\\ q-q^i(I)
 \end{pmatrix}+\OO_2\left(p,q-q^i(I)\right)\right],
 \end{split}
\]
with 
\begin{equation}\label{def:Theta}
 \Theta^i(I,p,q)=\pa_2 F\left(I^i_*(I,p,q), 
p_*^i(I,p,q)q_*^i(I,p,q)\right),\end{equation}
where $I_*^i, p_*^i, q_*^i$ are the images of the gluing map given in
Lemma~\ref{lemma:GluingOriginal} and $F$ is given by Lemma~\ref{lemma:MoserNF}.
\end{itemize}
Moreover,
%
the functions $a_*^i$, $b_*^i$, $c_*^i$ and $p_0^i$  satisfy
\begin{equation}\label{eq:GluingSymplectic}
 b_{10}^ic_{01}^i-b_{01}^ic_{10}^i=1,\quad a_{10}^i=c_{10}^i\pa_Ip_0^i,\quad
a_{01}^i=c_{01}^i \pa_Ip_0^i
\end{equation}
and
\begin{equation}\label{eq:GluingTransversality}
 c_{01}^i\neq 0.
\end{equation}
\end{theorem}

\begin{theorem}\label{thm:FormulasSMii}
Assume Ansatz \ref{ans:NHIMCircular:bis} and fix $\de>0$, $\kk>1$,  $M>0$. 
There exists $0<\rr\ll 1$ and $e_0^*>0$
such that for any $e_0\in (0,e_0^ *)$, there exists a system of
coordinates
$(\wh I,\wh s,\wh p,\wh
q)$ on
\[
 \mathcal D=\left\{(I,s,p,q,g): I\in [I_-+\de,I_+-\de], (s,g)\in\TT^2,
|p|,|q|\leq M, |pq|\leq\rr\right\},
\]
satisfying
\[
(\wh I,\wh s,\wh p,\wh q)= (I,s,p,q)+\OO(e_0),
\]
such that, on the domain  $\UU_\rho^{ij}$ introduced in \eqref{def:SubDomainUij},
which satisfies $\UU_\rho^{ij}\subset \mathcal D$,
%
%
the separatrix map $\SM_{e_0}^{ij}:  \UU_\rho^{ij}\to \UU^j$  in
 \eqref{def:SMij},
is well defined. Moreover, dropping the hats in the coordinates
, the separatrix map
$\SM_{e_0}^{ij}=(\FF^{ij}_{I,e_0},\FF^{ij}_{s,e_0},\FF^{ij}_{p,e_0},\FF^{ij}
_{q,e_0})$ is as follows.
%
\begin{itemize}
\item The $(I,s)$ components are given by\footnote{In the estimates of the errors below, note that the power in the logarithms could be more accurate. Indeed, the $C^0$ norm of the errors is $\OO(e_0^3\log \rr)$ and the $C^1$ norm is $\OO(e_0^3\log^2 \rr)$. Since, we consider $\rr$ fixed and take $e_0$ arbitrarily small, these logarithms do not play role.}
\[
\begin{pmatrix}\FF^{ij}_{I,e_0}\\
\FF^{ij}_{s,e_0}\end{pmatrix}=\begin{pmatrix}\FF^{ij}_{I,0}+e_0
\FF^{ij}_{I,1}+e_0^2
\FF^{ij}_{I,2}+\OO_{C^2}\left(e_0^3\log^3\rr\right)\\
\FF^{ij}_{s,0}+e_0
\FF^{ij}_{s,1}+\OO_{C^2}\left(e_0^2\log^3\rr\right)
\end{pmatrix},
\]
where $\FF^{ij}_{I,0}$, $\FF^{ij}_{s,0}$ are the functions introduced in
Theorem \ref{thm:FormulasSMii:circ} and
\[
\begin{split}
\FF^{ij}_{I,1}&=\MM^{I,1}_i(I,s)+\MM^{I,1}_{i,pq}(I,s,p,q)\\
\FF^{ij}_{I,2}&=\MM^{I,2}_i(I , s)+\MM^{I,2}_{i,pq}(I, s,p,q)\\
\FF^{ij}_{s,1}&=\MM^{s,2}_i(I,s)+\MM^{s,1}_{i,pq}(I,s,p,q),
\end{split}
\]
where
$\MM^{*,*}_i$ are $C^2$ functions which satisfy
\[
\MM^{I,l}_i=\OO_{C^2}(1),\quad  \MM^{s,1}_i=\OO_{C^2}(\log\rr), \quad
\MM^{I,l}_{i,pq}=pq\OO_{C^2}(1),\quad  
\MM^{s,1}_{i,pq}=pq\OO_{C^2}(\log\rr)
\]
and
\[
\NNN(\MM^{z,1}_{i})=\NNN(\MM^{z,1}_{i,pq})=\{\pm 1\},\,\,\,z=I,s\quad
\text{and}\quad \NNN(\MM^{I,2}_i)=\NNN(\MM^{I,2}_{i,pq})=\{0,\pm 2\}.
\]
\item The $(p,q)$ components are of the form
%
\[
\begin{pmatrix}\FF^{ij}_{p,e_0}\\
\FF^{ij}_{q,e_0}\end{pmatrix}=\begin{pmatrix}\FF^{ij}_{p,0}+e_0
\FF^{ij}_{p,1}+\OO_{C^2}\left(e_0^2\log\rr\right)\\
                                           \FF^{ij}_{q,0}+e_0
\FF^{ij}_{q,1}+\OO_{C^2}\left(e_0^2\log\rr\right)
                                          \end{pmatrix},
\]
where $\FF^{ij}_{p,0}$, $\FF^{ij}_{q,0}$ are the functions introduced in
Theorem \ref{thm:FormulasSMii:circ}
and the functions $\FF^{ij}_{p,1}$, $\FF^{ij}_{q,1}$ satisfy
$\FF^{ij}_{p,1},\FF^{ij}_{q,1}=\OO_{C^2}(\log\rr)$ and
$\NNN(\FF^{ij}_{p,1})=\NNN(\FF^{ij}_{q,1})=\{\pm 1\}$.
\end{itemize}
\end{theorem}

These theorems are proven in Section \ref{app:SM}.
The main improvement on \cite{Treschev02a,Piftankin06} is to perform a second order analysis in $e_0$ of the separatrix map. 

For the
analysis in the companion paper \cite{stochastic23}, we need explicit formulas for the functions $\MM_i^{I,1}$ and
$\al_i(I)$ in terms of the
original system
\eqref{def:Reduced:ODE} (see Ansatz \ref{ansatz:Melnikov:1} below). Indeed, they will Melnikov-type integrals defined through the functions
$\Delta H_\ccirc$ and $\Delta H_\eell$ introduced in
\eqref{def:HamDelaunayNonRot}. In  \cite{FejozGKR15} it was checked  that
they satisfy
\begin{equation}\label{def:HarmonicsHams}
 \NNN\left(\Delta H_\ccirc\right)=\{0\}\qquad\text{ and }\qquad\NNN\left(\Delta
H_{\eell}|_{e_0=0}\right)=\{\pm 1\}.
\end{equation}
We need to introduce some more notation. Since the right hand side of \eqref{def:Reduced:ODE:circular} does
not depend on $t$, one can consider it for the coordinates $(\ell, L, g, G)$
 treating $I$ as a constant. Denote by $\Phi^\ccirc_0$ its flow. Ansatz
\ref{ans:NHIMCircular:bis} implies that, for each
value $I\in [I^-, I^+]$,  this flow has a periodic orbit with two transverse
homoclinic orbits. We denote the time-parameterization of this periodic orbit
by $\la_I(s)$ and those of the two
homoclinic orbits by $\ga_I^i(s)$ with $i=1,2$. Moreover, since $\dot t$ in
\eqref{def:Reduced:ODE:circular} only depends on $(\ell, L, g, G, I)$, it can be
integrated on the periodic orbit to obtain
\[
 t=t_0+\wt\la_I(s),\qquad \wt\la_I(s)=\int_0^s\frac{1}{-1+\mu\pa_G\Delta
H_\ccirc(\la_I(\sigma))}d\sigma
\]
and analogously on the homoclinic orbits to obtain analogously  $t=t_0+\wt\ga^i_I(s)$.

\begin{proposition}\label{prop:Melnikov}
The functions $\alpha_i(I)$ and $\MM_i^{I,1}(I,s)$ introduced in  Theorems
\ref{thm:FormulasSMii:circ} and \ref{thm:FormulasSMii} respectively satisfy the following.
\begin{itemize}
\item $\alpha_i(I)$ can be written as $ \alpha_i(I)= \alpha^+_i(I)-
\alpha^-_i(I)$
with
\begin{equation}\label{def:omega}
\alpha^\pm_i(I)
=\mu\lim_{N\to\pm\infty}\left(\int_0^{2\pi N}\frac{\left(\pa_G
\Delta H_\ccirc\right)\circ \ga_I^i(\sigma)}{-1+\mu \left(\pa_G
\Delta H_\ccirc\right)\circ \ga_I^i(\sigma)}d\sigma+2\pi N\nu(I)\right)
\end{equation}
where $\nu$ is the function defined in \eqref{def:nu}.
\item The function  $\MM_i^{I,1}(I,s)$ satisfies
$\NNN(\MM_i^{I,1}(I,s))=\{\pm 1\}$ and can be written as
\[
\MM_i^{I,1}(I,s)=B_i(I)e^{\im s}+\ol{B_i}(I)e^{-\im s}
\]
with $B_i=B_i^\inn+B_i^\out$ where
%
\begin{align}\label{def:Melnikovs}
 B_i^\inn(I)=&\im\mu\frac{1-e^{\im\alpha_i
(I)}}{1-e^{\im 2\pi  \nu(I)}}\int_0^{2\pi}
\frac{\Delta H_\eell^{1,+}\circ\la_I(\sigma)}{-1+\mu \pa_G\Delta
H_\ccirc\circ\la_I(\sigma)}
e^{\im\wt\la_I(\sigma)}d\sigma.\\
 B_i^\out(I)=&- \im\mu\lim_{T\rightarrow+\infty}\int_0^T\left(\frac{\Delta H^{1,+}_\eell\circ\gamma_I^i(\sigma)}{-1+\mu\pa_G\Delta H_\ccirc\circ\gamma_I^i(\sigma)}e^{ \im\wt\gamma_I^i(\sigma)}\right.\nonumber\\
    &\qquad\qquad\qquad\left.-\frac{\Delta
H_\eell^{1,+}\circ\la_I(\sigma)}{-1+\mu\pa_G\Delta
H_\ccirc\circ\la_I(\sigma)}e^{
\im\left(\wt\la_I(\sigma)+\alpha^+_i(I)\right)}\right) d\sigma\\
    & +\im\mu\lim_{T\rightarrow-\infty}\int_0^T\left(\frac{\Delta H_\eell^{1,+}\circ\gamma_I^i(\sigma)}{-1+\mu\pa_G\Delta H_\ccirc\circ\gamma_I^i(\sigma)}e^{ \im\wt\gamma_I^i(\sigma)}\right.\nonumber\\
    &\qquad\qquad\qquad\left.-\frac{\Delta
H_\eell^{1,+}\circ\la_I(\sigma)}{-1+\mu\pa_G\Delta
H_\ccirc\circ\la_I(\sigma)}e^{
\im\left(\wt\la_I(\sigma)+\alpha^-_i(I)\right)}\right)d\sigma,\nonumber
\end{align}
where $\Delta H_\eell^{1,+}$ is defined by $\Delta
H_\eell^{1}(I,s)=\Delta H_\eell^{1,+}(I)e^{\im s}+\Delta H_\eell^{1,-}(I)e^{-\im s}$, with $\Delta
H_\eell^{1}=\Delta
H_\eell|_{e_0=0}$
(see \eqref{def:HarmonicsHams}).
\end{itemize}
\end{proposition}
This proposition is proved in Section \ref{app:Melnikov}.

\subsection{Derivation of the separatrix map: Proof of Theorems \ref{thm:FormulasSMii:circ} and \ref{thm:FormulasSMii}}\label{app:SM}
We devote this section to prove Theorems \ref{thm:FormulasSMii:circ} and   \ref{thm:FormulasSMii}, by  providing formulas for the separatrix maps of the circular and elliptic problems. Note that both theorems are proven at the same time.
To prove both theorems we follow the approach in \cite{Treschev02a, GuardiaK15} adapted to an a priori chaotic setting and to the  features of the RPE3BP (in particular, to the harmonic structure of the expansion in $e_0$ of the Hamiltonian \eqref{def:HamEllMoser}).

The proof is achieved in three steps:
\begin{enumerate}
\item We perform a normal form to the elliptic Hamiltonian \eqref{def:HamEllMoser} in a neighborhood of the normally hyperbolic invariant cylinder and we extend it along its stable and unstable invariant manifolds. Note that the normal form transformation  becomes multivaluated when extended to the homoclinic channels. This is done in Section \ref{sec:NormalFormElliptic}.
\item We analyze the flow of the normal form Hamiltonian (see Section \ref{sec:flowestimates}).
\item We analyze the gluing maps which relate the two extensions of the normal form along the stable and unstable invariant manifolds (see Section \ref{sec:gluing}).
\end{enumerate}	
These three steps lead to the formulas of the separatrix map provided in Theorems \ref{thm:FormulasSMii:circ} and \ref{thm:FormulasSMii}.	

\subsubsection{A normal form for the elliptic
Hamiltonian}\label{sec:NormalFormElliptic}

To perform a normal form procedure, the first step is to expand the elliptic
Hamiltonians $\KK$ in  \eqref{def:EllipticMoser}  in
powers of the parameter $e_0$. Consider the elliptic Hamiltonian perturbation
$\Delta K_\eell$ expressed in
Delaunay coordinates (see  \eqref{def:HamEll:Reparam}) and expand it in powers
of $e_0$,
\[
\Delta K_\eell(L,\ell,g,I,t;e_0)= K_1(L,\ell,g,I,t)+e_0
K_2(L,\ell,g,I,t)+\OO\left(e_0^2\right).
\]
In \cite{FejozGKR15}, it is shown that nontrivial harmonics of $K_1$ and $K_2$ satisfy
the following
\[  \NNN(K_1)=\{\pm 1\},\quad \NNN(K_2)\subset\{0,\pm 1,\pm 2\} \]
(see Definition  \ref{not:harmonics}).
\begin{lemma}\label{lemma:EllipticHamReparamHarmonic}
The Hamiltonian $\KK$ in \eqref{def:EllipticMoser}  is of the form
\begin{equation}\label{def:HamEllMoser}
 \KK(I,s,p,q,g)=\KK_0(I,p,q,g)+e_0\KK_1(I,s,p,q,g)+e_0^2\KK_2(I,s,p,q,g)+\OO(e_0^3),
\end{equation}
where $\KK_0$ is the Hamiltonian introduced in Lemma \ref{lemma:MoserNF} and
$\KK_i=K_i\circ \Psi^{-1}\circ\Upsilon^{-1}$ satisfy
\[
 \NNN(\KK_1)=\{\pm 1\},\quad \NNN(\KK_2)\subset\{0,\pm 1,\pm 2\}.
\]
\end{lemma}
\begin{proof}
 This lemma is a direct consequence of the particular form of the changes
obtained in Lemmas~\ref{lemma:ChangePO} and  \ref{lemma:MoserNF}.
\end{proof}

We perform two steps of normal form to the Hamiltonian  \eqref{def:HamEllMoser}.
This normal form extends the Moser
coordinates to the region where $|qp|$ is small and straightens the invariant
manifolds up to errors $\OO(e_0^3)$. The normal form procedure
follows the same lines as the one in \cite{GuardiaK15}. Nevertheless, it has
some differences since we do not have resonances and therefore it can
be performed globally
%
in a compact domain of the form
\begin{equation}\label{def:DomainNormalForm}
 \DD_{\rho,\de, \tM}=\{(I,s,p,q,g): I\in [I_-+\de,I_+-\de], (s,g)\in\TT^2,
|p|,|q|\leq \tM, |pq|\leq\rr\}.
\end{equation}

\begin{proposition}\label{prop:NormalForm}
Fix $\tM_2>\tM_1$ and $\de>0$. Consider the Hamiltonian $\KK$ given by Lemma
\ref{lemma:EllipticHamReparamHarmonic}. Then, there exists $\rr^*>0$ such that for any $\rr\in (0,\rr^*)$
there exists a symplectic transformation
\[
\Phi :  \DD_{\rho/2,2\de, \tM_1}\rightarrow \DD_{\rho,\de, \tM_2}
\]
such that
\begin{enumerate}
 \item The Hamiltonian $\KK
\circ\Phi$ is of the
 form
\[
 \KK \circ\Phi(\wh I,\wh s,\wh p,\wh q,\wh
g)=E(\wh I\,)
+F_0(\wh I,\wh p \wh q)+e_0^2 F_1(\wh I,\wh p \wh q)+
\OO\left(e_0^{3}\right).
\]
\item The changes of coordinates $\Phi$  is of the form
 \[
 \begin{split}
  I =&\, \wh I+e_0 M_1^I+e_0^2  M_2^I+
\OO\left(e_0^{3}\right)\\
  s =& \, \wh s+e_0 M_1^s+e_0^2 M_2^s
  +\OO\left(e_0^{3}\right)\\
  p = &\, \wh p+e_0 M_1^p
+e_0^2 M_2^p+\OO\left(e_0^{3}\right)\\
  q =& \, \wh q+e_0 M_1^q +e_0^2 M_2^q
  +\OO\left(e_0^{3}\right)
 \end{split}
 \]
were each $M^*_*$ is a $C^2$  function and, for $z=I,p,q, s$, satisfies
 \[
 \begin{split}
  M_1^z&=\OO(1) \quad \text{ and } \quad \NNN(M_1^z)=\{\pm 1\}\\
  M_2^z&=\OO(1)\quad \text{ and } \quad \NNN(M_2^z)\subset\{0,\pm 1,\pm 2\}.
\end{split}
\]
 \end{enumerate}
 \end{proposition}

 The proof of this proposition is deferred to the end of this section
(Section \ref{sec:ProofNormalForm}). It follows the approach in \cite{Treschev02a}, but it presents significant differences. Indeed, there the normal form procedure only removes the first order in $e_0$ and $pq$ whereas in the present paper we remove completely the terms up to errors of order $\OO(e_0^3)$.

\begin{corollary}\label{coro:InverseNormalForm}
The inverse change $\Phi\ii$ is of the same form, that is
 \[
 \begin{split}
  \wh I =&\, I-e_0 M_1^I
  +e_0^2 \wt M_2^I+\OO\left(e_0^{3}\right)\\
  \wh s =&\, s-e_0 M_1^s+e_0^2 \wt
M_2^s+\OO\left(e_0^{3}\right)\\
  \wh p = &\,p-e_0 M_1^p+e_0^2 \wt
M_2^p+\OO\left(e_0^{3}\right)\\
  \wh q =&\, q-e_0 M_1^q+e_0^2 \wt
M_2^q+\OO\left(e_0^{3}\right),
\end{split}
 \]
where the functions $\wt M_2^z$  satisfy the same properties as $M_2^z$.
\end{corollary}

\subsubsection{The flow in the normal form variables}\label{sec:flowestimates}
The equations associated to the Hamiltonian given in Proposition
\ref{prop:NormalForm} are
\[
\begin{split}
\dot I &= \OO\left(e_0^{3}\right)\\
\dot s &= \nu(I)+\pa_I F_0(I,pq)+e_0^2\pa_I F_1(I,pq)+\OO\left(e_0^{3}\right)\\
\dot p &= -\left(\pa_2F_0(I, pq) +e_0^2\pa_2F_1(I, pq) +\OO\left(e_0^3
\right)\right)p+\OO\left(e_0^3\right)\\
 \dot q &= \left(\pa_2F_0(I, pq)+e_0^2\pa_2F_1(I, pq)\right)
q+\OO\left(e_0^3\right)
\end{split}
 \]
 and
$ \frac{d}{dg} (pq) = \OO\left(e_0^{3}\right)$.

Fix $0<\rr\ll 1$ and $\tM>1$ and introduce the domain
\begin{equation}\label{def:DomainForFlow}
U_*=\left\{ (I,s,p,q):\ I\in [I_-+\delta,I_+-\delta],\ \tM\ii<|p|<\tM,\ \tM e_0
\leq |pq|\leq \rr\right\}.
\end{equation}

\begin{lemma}\label{lemma:Flow}
Suppose that for some $(I^*,s^*,p^*,q^*)\in U_*$
and final time 
$\ol g\in\RR$,
\[
\tM\ii\leq |q^*|\,e^{\pa_2 F_0(I^*,p^*q^*)\ol g}\leq \tM.
\]
Then,
\[
\begin{split}
 I(\ol g)&=\ I^* +\OO\left(e_0^{3}\right)\ol g\\
  s(\ol g)&=\ s^* +\left(\nu(I^*)+\pa_I
F_0(I^*,p^*q^*)+e_0^2\pa_I
F_1(I^*,p^*q^*)\right)\ol g+\OO\left(e_0^{3}\right)\ol g^2\\
 p(\ol g)&=\ p^*e^{-(P_0+e_0^2P_1^*)\ol g}
\left(1+\OO\left(e_0^{3}\right)\ol g\right)+\OO\left(e_0^3\right)\\
  q(\ol g)&=\ q^* e^{\,(P_0+e_0^2P_1^*))\ol
g}\left(1+\OO\left(e_0^{3}\right)\ol g\right)\\
p(\ol g)q(\ol g)&=\
p^*q^*+\OO\left(e_0^{3}\right)\ol g,
\end{split}
\]
where $P_0^*=\pa_2 F_0(I^*,p^*q^*)$ and $P_1^*=\pa_2 F_1(I^*,p^*q^*)$.
\end{lemma}


The proof of this lemma can be found in \cite{GuardiaK15}.
Even though there it is only valid for points $\eps$-away from the invariant manifolds,  one can easily check that the proof also applies to the larger domain $U_*$ in \eqref{def:DomainForFlow}.
%
%
%
%
%
%
%

\subsubsection{The gluing maps}\label{sec:gluing}

The Moser normal form coordinates obtained in Lemma \ref{lemma:MoserNF} can be extended  along the stable and unstable invariant manifolds of the cylinder. Note that this extension is multivaluated. Indeed, when propagating the coordinates along the stable and unstable invariant manifolds, we reach domains where the invariant manifolds intersect. In this overlapping domain, we have ``stable'' Moser coordinates (Moser coordinates propagated along the stable manifold) and ``unstable'' Moser coordinates. Following Treschev \cite{Treschev02a}, we call gluing map to the diffeomorphism which expresses the stable Moser coordinates in terms of the unstable ones. 

We devote this section to analyze these gluing maps. We first analyze them in the original Moser normal form coordinates.
We compute them in neighborhoods of the homoclinic channels. Indeed,  since we are in an \emph{a priori}
chaotic setting instead of \emph{a priori} unstable,  the gluing maps can  be
only
 defined in suitable small neighborhoods of  homoclinic points and not in a
whole fundamental domain   (as happened in
\cite{Treschev02a} and \cite{GuardiaK15}). These gluing maps are $e_0$-independent since they are defined by the
Moser normal
form associated to the circular problem given by Lemma \ref{lemma:MoserNF}.

The second step in this section is to  express the gluing maps in the
normal form variables obtained in Proposition \ref{prop:NormalForm}. 
They depend on $e_0$ since the change of coordinates
given
by Proposition \ref{prop:NormalForm} does.
%


%

Since the change to Moser normal form given in Lemma
\ref{lemma:MoserNF} does not modify the variable $g$, we can study the gluing
maps restricted to the section $\{g=0\}$. We define thus $\rr$-neighborhoods of the homoclinic channels $\widetilde\CCC_0^i$ in \eqref{def:channelsPoincarecirc}, 
\[
 \DD^i_{\rr}=\left\{I\in [I_-+\delta,I_+-\delta],\, s\in\TT,\,|p|\leq\rr,\, |q-q^i(I)|\leq
\rr\right\}\quad \text{ for }\quad 0<\rr\ll 1\quad \text{and}\quad i=1,2
\]
(see \eqref{def:HomoChannels}).
We
define the gluing maps $\GG_0^i$ associated to the Moser normal form coordinates in Lemma~
\ref{lemma:MoserNF}  in the domains  $\DD^i_{\rr}$.

\begin{lemma}\label{lemma:GluingOriginal}
Fix $0<\rr\ll 1$ and consider the section $\{g=0\}$. Then, the gluing map
restricted to this section,
 \[
  \GG^i_0:  \DD^i_{\rr}\rightarrow \RR\times\TT\times\RR^2,
 \]
denoted by $(I^*,s^*,p^*,q^*)=\GG_0(I,s,p,q)$ is analytic and of the form
 \[
 \begin{split}
  I^*=&I\\
s^*=&s+\al_i(I)+a_{10}^i(I)p+a_{01}^i(I)(q-q^i(I))+\OO_2(p, q-q^i(I))\\
  p^*=&p_0^i(I)+b_{10}^i(I)p+b_{01}^i(I)(q-q^i(I))+\OO_2(p, q-q^i(I))\\
  q^*=&c_{10}^i(I)p+c_{01}^i(I)(q-q^i(I))+\OO_2(p, q-q^i(I)).
 \end{split}
 \]
Moreover,
the  coefficients $a^i_{kl},\ b^i_{kl},\ c^i_{kl}$
are analytic functions and satisfy formulas \eqref{eq:GluingSymplectic} and
\eqref{eq:GluingTransversality}.
\end{lemma}

\begin{proof}
The particular form and \eqref{eq:GluingSymplectic}
 is just a consequence of Lemma
\ref{lemma:MoserNF},  the
fact that it is symplectic and that $I$ is a first integral. Then, it just
remains to
do a Taylor expansion around the homoclinic channel $(p,q)=(0,q^i(I))$.

Property
\eqref{eq:GluingTransversality} is a consequence of the transversality of the
invariant manifolds at the homoclinic channel $\widetilde\CCC^i_0=\CCC^i_0\cap\{g=0\}$ given by  Corollary \ref{coro:NHIMCircular}.
\end{proof}

%

Next step is to express the gluing maps in the coordinates obtained in
Proposition
\ref{prop:NormalForm}. To simplify notation, we define
the matrices
\begin{equation}\label{def:GluingMatrix}
B^i(I)=\begin{pmatrix}
 b_{10}^i(I) & b_{01}^i(I)\\c_{10}^i(I) & c_{01}^i(I)
\end{pmatrix},\quad i=1,2,
\end{equation}
which, by \eqref{eq:GluingSymplectic}, are invertible and satisfy $\det B^i=1$.

\begin{lemma}\label{lemma:GluingElliptic}
Consider the section $\{g=0\}$, the change of coordinates $\Phi$ obtained in
Proposition
\ref{prop:NormalForm} and the gluing maps $\GG_0^i$, $i=1,2$, obtained in
Lemma \ref{lemma:GluingOriginal}. Then,  the maps
\[
 \wt \GG^i_{e_0}=\Phi\ii\circ \GG_0^i\circ \Phi:\,\,  \DD^i_{\rr}\rightarrow
\RR\times\TT\times\RR^2,
\]
with $\rr\ll 1$, are of the form
\[
 \wt \GG^i_{e_0}=\GG_0^i+e_0\GG_1^i+e_0^2\GG_2^i+\OO\left(e_0^3\right),
\]
where
\begin{itemize}
\item $\GG^i_0$ are the maps given in Lemma \ref{lemma:GluingOriginal}.
\item $\GG^i_1$ satisfy $\NNN\left(\GG^i_1\right)=\{\pm 1\}$
and are of the form
 \[
\GG^i_1(I,s,p,q)=
 \begin{pmatrix}
\MM^{I,1}_i(I,s)+\MM^{I,1}_{i,pq}(I,s,p,q)\\
\wt \MM^{s,1}_i(I,s)+\wt \MM^{s,1}_{i,pq}(I,s,p,q)\\
\MM_i^{p,1}(I ,s)+\MM_{i,pq}^{p,1}(I ,s,p,q)\\
\MM_i^{q,1}(I ,s)+\MM_{i,pq}^{q,1}(I ,s,p,q)
 \end{pmatrix}
 \]
for some  functions $\MM^{1,i}_*$ such that  $\MM^{1,i}_*=\OO_{C^2}(1)$ with
respect to $e_0$ and $\rr$. Moreover
\[
 \begin{split}
 \MM^{I,1}_i(I,s)=&M^I_1( I,  s,0,q^i(I))-M^I_1( I,  s+\al_i(I),p_0^i(I),0)\\
\wt \MM^{s,1}_i(I,s)=&M^s_I( I,  s,0,q^i(I))-M^s_1( I,
s+\al_i(I),p_0^i(I),0)+(\al^i)'(I)
M^{I}_1( I,  s,0,q^i(I))\\
 &+a_{10}^{i}(I)M^{I,i}_p( I, s,0,q^i(I))+a_{01}^{i}(I)M^I_q(I, s,0,q^i(I))\\
\begin{pmatrix} \MM^{p,1}_i(I,s)\\
\MM^{q,1}_i(I,s)\end{pmatrix}=&B^i(I)\begin{pmatrix} M^p_1(I,
s,0,q^i(I))\\  M^q_1( I,
s,0,q^i(I))\end{pmatrix}-\begin{pmatrix} M^p_1(I, s+\al_i(I),p_0^i(I),0)\\
M^q_1(I,
s+\al_i(I),p_0^i(I),0)\end{pmatrix}\\
&+{p_0^i}'(I)\begin{pmatrix} M^I_1(I,
s,0,q^i(I))\\
0\end{pmatrix}.
\end{split}
\]
Finally,
\[
 \MM_{i,pq}^{*,1}(I ,s)=\OO(pq).
\]
\item $\GG^i_2$ satisfy $\NNN\left(\GG^i_2\right)=\{0,\pm 1, \pm 2\}$, are of
the same form as $\GG^i_1$ and the associated functions satisfy the same
estimates.
%
%
%
%
%
\end{itemize}
\end{lemma}
\begin{proof}
The proof of this lemma is a direct consequence of applying the change of
coordinates obtained in Lemma \ref{lemma:NormalFormChangeOfCoordinates} to the
gluing map defined in Lemma \ref{lemma:GluingOriginal}.
\end{proof}

\subsubsection{Formulas for the separatrix maps}\label{sec:FormulasSMii}
 In this section we obtain formulas for the separatrix map $\SM^{ij}$ which go
from suitable domains close to the homoclinic channel $i$ to suitable domains close to the homoclinic channel $j$ and thus we prove Theorem \ref{thm:FormulasSMii}.
To simplify the notation, 
we omit the dependence on the channels given by $i$ and $j$. 

Consider the flow $\Phi_{e_0}^{\ol g}$ analyzed in
Lemma \ref{lemma:Flow} and the gluing
maps $\wt \GG_{e_0}$ analyzed in Lemma \ref{lemma:GluingElliptic}. We compose them and we obtain the following formula for the separatrix map. This lemma completes the proof of Theorem \ref{thm:FormulasSMii}.
%

\begin{lemma}\label{lemma:DefSM}
Consider the section $\{g=0\}$. Then, the map
$\FF=\Phi_{e_0}^{\ol g}\circ \wt \GG_{e_0}: \{g=0\}\rightarrow \{g=0\}$ is
given
by
\[\FF(I,s,p,q)=(\FF_I(I,s,p,q),\FF_s(I,s,p,q),\FF_p(I,s,p,q),\FF_q(I,s,p,q))\]
as follows.
\begin{itemize}
 \item For the $(I,s)$ components,
\[
\begin{split}
\FF_I=& \,I+
e_0\left(\MM^{I,1}(I,s)+\MM^{I,1}_{pq}(I,s,p,q)\right)+e_0^2\left(\MM^{I,2}(I , s)+\MM^{I,2}_{pq}(I , s,p,q)\right)+\OO\left(e_0^3\right)\ol g\\
\FF_s=& \,s+\al(I)+(\nu(I)+\pa_I F(I,pq))\ol g+ D(I,p,q)+e_0\left(\MM^{s,1}(I,s)+\MM^{s,1}_{pq}(I,s,p,q)\right)+\OO\left(e_0^2\right)\ol g
%
%
%
\end{split}
\]
where the functions $\MM^{I,*}$ have been introduced in Lemma \ref{lemma:GluingElliptic},
\[
\begin{split}
D(I,p,q)=&\, a_{10}(I)p+ a_{01}(I)(q-q(I))+\OO\left(p,q-q(I)\right),\\ \MM^{s,1}(I,s)=&\,\wt\MM^{s,1}(I,s)+\ol g\nu'(I)\MM^{I,1}(I,s),
\end{split}
\]
and  $\MM^{s,1}_{pq}$ satisfies $\MM^{s,1}_{pq}=\OO(\ol g)$.
\item For the $(p,q)$ components,
\[
\begin{pmatrix}\FF_p\\
\FF_q\end{pmatrix}=\begin{pmatrix}\FF_p^0+e_0 \FF_p^1+\OO\left(e_0^2\right)\ol g\\
                                           \FF_q^0+e_0 \FF_q^1+\OO\left(e_0^2\right)\ol g
                                          \end{pmatrix}
\]
with
\[
\begin{split}
\begin{pmatrix}\FF^0_p\\\FF^0_q\end{pmatrix}=&\,\begin{pmatrix}
  e^{-\Theta(I,p,q)\ol g} &0\\
 0 &e^{\Theta(I,p,q)\ol g}
 \end{pmatrix}\left[\begin{pmatrix}p_0(I)\\0\end{pmatrix}+ \begin{pmatrix}
 b_{10}(I) &b_{01}(I)\\
 c_{10}(I) &c_{01}(I)
 \end{pmatrix}
\begin{pmatrix}
  p\\ q-1
 \end{pmatrix}+\OO_2(p,q-q(I))\right]\\
 \end{split}
\]
with
\begin{equation}\label{def:Theta:app}
 \Theta(I,p,q)=\pa_2 F\left(I^*(I,p,q), p^*(I,p,q)q^*(I,p,q)\right),\end{equation}
where $I^*, p^*, q^*$ are the images of the gluing map given in
Lemma~\ref{lemma:GluingOriginal} and $F$ is given by
Proposition~\ref{prop:NormalForm}.

Moreover, the functions $\FF_p^1$, $\FF_q^1$ satisfy $\FF_p^1,\FF_q^1=\OO_{C^2}(\ol g)$ and $\NNN(\FF_p^1)=\NNN(\FF_q^1)=\{\pm 1\}$ and
\end{itemize}

%
%
\end{lemma}
Note that in this lemma $\ol g$ is not a free parameter but is related to the
distance of the original point to the stable invariant manifold. This relation is more explicitly stated in Section \ref{sec:NHIL}.
The proof of this lemma is a direct consequence of the composition of the maps
obtained in Lemmas \ref{lemma:Flow} and \ref{lemma:GluingElliptic}.

\subsubsection{The normal form: Proof of Proposition
\ref{prop:NormalForm}}\label{sec:ProofNormalForm}

We perform the change of coordinates by the Lie Method.
We consider the expansion
of the Hamiltonian $\KK$ given in Lemma \ref{lemma:EllipticHamReparamHarmonic}
and we perform successive changes to remove successive orders (in $e_0$ and
$qp$).

We consider a Hamiltonian  $e_0 W_1$ and call
$\Phi_1$ the time-one map associated to the flow of
$e_0 W_1$. This change is symplectic. Moreover,
\[
\begin{split}
 \KK\circ \Phi_1=&\KK_0+e_0\left(\KK_1+
\{\KK_0,W_1\}\right)\\
&+e_0^2\left(\{\KK_1,W_1\}+\{\{\KK_0,W_1\},W_1\}
+\KK_2\right) +\OO\left(e_0^3\right).
\end{split}
 \]
We look for suitable $W_1$ such that
\begin{equation}\label{eq:Homological}
 \KK_1+ \{\KK_0,W_1\}=0.
\end{equation}
To  compute $W_1$, we write the
Hamiltonian
$\KK_1$  defined in Lemma~\ref{lemma:EllipticHamReparamHarmonic}, in the
following way
\[
  \KK_1(I,s,p,q,g)=\sum_{j\geq 0}(pq)^j\left(  \KK_1^{j}(I,s,g)+  \KK_1^{j,p}(I,s,p,g)+
\KK_1^{j,q}(I,s,q,g)\right)
\]
where
\begin{equation}\label{def:HamiltonianSplitting}
\begin{split}
\KK_1^{j}(I,s,g)&=\frac{1}{(j!)^2}\pa_p^j\pa_q^j\KK_1(I,s,0,0,g)\\
\KK_1^{j,p}(I,s,p,g)&=\frac{1}{(j!)^2}\left(\pa_p^j\pa_q^j\KK_1(I,s,p,0,
g)-\pa_p^j\pa_q^j\KK_1(I,s,0,0,
g)\right)\\
\KK_1^{j,q}(I,s,q,g)&=\frac{1}{(j!)^2}\left(\pa_p^j\pa_q^j\KK_1(I,s,0,q,
g)-\pa_p^j\pa_q^j\KK_1(I,s,0,0,
g)\right).
\end{split}
\end{equation}
The analyticity of $\KK_1$ implies that this series is convergent  in
$\DD_{\rho,\de, \tM}$  in \eqref{def:DomainNormalForm} for any given $\tM>0, \de>0$
and taking $\rr>0$ small enough (depending on $M$) and that in $\DD_{\rho,\de,
\tM}$,
\begin{equation}\label{def:EstimateKj}
|\KK_1^{j,\ast}|\leq C,\qquad \ast=\emptyset, q,p,
\end{equation}
for some $C>0$ independent of $\rr$ and $j$.
%

The next lemma contains the first step of the normal form. In the next two
lemmas we denote by $\{\cdot,\cdot\}_{(q,p)}$, the Poisson bracket with
respect to the conjugate variables $(q,p)$.

\begin{lemma}
\label{lemma:NormalForm:FirstOrderCohomological}
Fix $\tM'>\tM$ and $\de'<\de$. Then, for $\rr>0$ small enough, there exists a change
of coordinates $\Phi_1: \DD_{3\rho/4,\de, \tM}\to \DD_{\rho,\de', \tM'}$ such that
\[
\| \Phi_1-\mathrm{Id}\|_{C^0}\leq \OO(e_0)
\]
and the Hamiltonian $\KK
\circ\Phi_1$ is of the
 form
\[
 \KK \circ\Phi_1 (\wh I,\wh s,\wh p,\wh q,\wh g)=A(\wh I\,)
+F(\wh I,\wh q \wh p)+e_0^2\wt \KK_2+ \OO\left(e_0^{3}\right),
\]
and $\wt \KK_2$ satisfies $\NNN(\wt \KK_2)\subset\{0,\pm 1,\pm 2\}$.
\end{lemma}

\begin{proof}
We build a function
$W_1(I,s,p,q,g)$ which is a solution of 
equation \eqref{eq:Homological}, which can be rewritten as
\begin{equation}\label{eq:Homological:2}
 \left(\nu(I)+\pa_I F(I,pq)\right)\pa_s W_1+\pa_g W_1+
\{F(I,pq), W_1\}_{(p,q)}+\KK_1
=0.
\end{equation}
We take
\begin{equation}\label{def:SeriesW1}
W_1(I,s,p,q,g)=\sum_{j\geq 0}(pq)^j\left(
W_1^{j}(I,s,g)+  W_1^{j,p}(I,s,p,g)+
W_1^{j,q}(I,s,q,g)\right).
\end{equation}
Then, it can be easily
checked that solving equation \eqref{eq:Homological:2} is equivalent to solving  the
equations
\[
 \left(\nu(I)+\pa_I  F(I,pq)\right)\pa_s W_1^{j,\ast}+\pa_g W_1^{j,\ast}+\{
F(I,pq),
W_1^{j,\ast}\}_{(q,p)}+\KK_1^{(j,\ast)}=0,\qquad \ast=\emptyset, q, p, 
\]
(note in particular that $\{F(I,pq),(pq)^j W\}_{(p,q)}=(pq)^j\{F(I,pq),
W\}_{(p,q)}$).

Each equation is solved as follows. For $W_1^{j}$, we just have
\begin{equation}\label{def:homo0}
 \left(\nu(I)+\pa_I  F(I,pq)\right)\pa_s W_1^{j}+\pa_g
W_1^{j}+\KK_1^{j}=0.
\end{equation}
Thus, we invert the operator
\begin{equation}\label{def:OperatorPartialTilde}
\wt\pa:=(\nu(I)+\pa_I  F(I,pq))\pa_s
+\pa_g
\end{equation}
by using the Fourier expansion and inverting for each Fourier
coefficient. Note that this operator can be inverted globally since thanks to
Lemma \ref{lemma:EllipticHamReparamHarmonic}, the Hamiltonians $\KK_1^{(j)}$  do
not have strong resonances. Indeed, Lemma
\ref{lemma:EllipticHamReparamHarmonic} implies that these functions only have
harmonics $(k_1,k_2)\in\ZZ^2$ with $k_1=\pm 1$. Then, by \eqref{def:Twist}, taking
$k_1=\pm 1$, $k_2\in\ZZ$ and  $\rr>0$ small enough,
\[
 \left|k_1 (\nu(I)+\pa_I  F(I,pq))+k_2\right|\geq 8\mu.
\]
Therefore, using also \eqref{def:EstimateKj}, the functions  $W_1^{j}$ satisfy that, in
$\DD_{\rho,\de', \tM'}$,
\[
|W_1^{j}|=|\wt\pa\ii \KK_1^{(1)}|\leq C 
\]
for some $C>0$ independent of $j$ and $\rr$.

For the others, we use the characteristics method to obtain
\[
\begin{split}
W_1^{j,p}&=-\int^{+\infty}_0 \KK_1^{j,p}(I,s+
 (\nu(I)+\pa_I F(I,pq))g', pe^{-\pa_2 F(I,pq) g'},g+g')\,dg',\\
 W_1^{j,q}&=-\int_{-\infty}^0 \KK_1^{j,q}(I,s+
(\nu(I)+\pa_I F(I,pq))g', qe^{\pa_2 F(I,pq) g'}, g+g')\,dg'
\end{split}
\]
which also satisfy
\[
\|W_1^{j,p}\|_{C^2},\|W_1^{j,q}\|_{C^2}\leq C 
\]
in $\DD_{\rho,\de', \tM'}$, for some $C>0$ independent of $j$ and $\rr$.

Then, the series \eqref{def:SeriesW1} is convergent for $\rr$ small enough and
we can define the change of coordinates $\Phi_1$ as the time 1 map associated
to the flow given by the Hamiltonian $e_0W_1$. Then, one can easily see that it
satisfies the statements of the lemma and that $\NNN(W_1)=\{\pm 1\}$.

Finally, one can easily see that
\[
\wt \KK_2= \{\KK_1,W_1\}+\{\{\KK_0,W_1\},W_1\}
+\KK_2
\]
where $\KK_2$ is Hamiltonian defined in Lemma
\ref{lemma:EllipticHamReparamHarmonic}. Then, $\NNN(\wt \KK_2)=\{0,\pm 1,\pm 2\}$.

\end{proof}

Now we eliminate the $e_0^2$-terms proceeeding as
in Lemma \ref{lemma:NormalForm:FirstOrderCohomological}.

\begin{lemma}
\label{lemma:NormalForm:SecondOrderCohomological}
Fix $\tM''>\tM'$ and $\de''<\de'$. Then, for $\rr>0$ small enough, there exists a
change
of coordinates $\Phi_2: \DD_{\rho/2,\de', \tM'}\to \DD_{3\rho/4,\de'', \tM''}$ such
that
\[
\| \Phi_2-\mathrm{Id}\|_{C^2}\leq \OO(e_0^2)
\]
and the Hamiltonian $\KK
\circ\Phi_1$ is of the
 form
\[
 \KK \circ\Phi_1\circ\Phi_2 (\wh I,\wh s,\wh p,\wh q,\wh g)=A(\wh
I\,)
+F(\wh I,\wh q \wh p)+ e_0^2F_2(\wh I,\wh q \wh p)+\OO\left(e_0^{3}\right).
\]
\end{lemma}

\begin{proof}
We construct the change $\Phi_2$ as the time 1 map associated to the flow
of a Hamiltonian $e_0^2 W_2$. We consider the Hamiltonian $ \KK
\circ\Phi_1$ obtained in Lemma \ref{lemma:NormalForm:FirstOrderCohomological}
%
and we proceed as in the proof of Lemma
\ref{lemma:NormalForm:FirstOrderCohomological} by looking for
a function $W_2$ satisfying
\[
\wt \KK_2+ \{\KK_0,W_2\}=0.
\]
This equation can be solved as in Lemma
\ref{lemma:NormalForm:FirstOrderCohomological}. Indeed, by Lemma \ref{lemma:NormalForm:FirstOrderCohomological},  $\wt \KK_2$ has
only harmonics
$\{0, \pm1,\pm 2\}$. Then, by
\eqref{def:Twist},
and $k_1=0,\pm 1,\pm 2$ and $k_2\in\ZZ$ with $k_1k_2\neq 0$,
\[
 \left|k_1 (\nu(I)+\pa_I  F(I,pq))+k_2\right|\geq 18\mu.
\]
Thus, proceeding as in Lemma \ref{lemma:NormalForm:FirstOrderCohomological}
one can deduce the statements of Lemma
\ref{lemma:NormalForm:SecondOrderCohomological}.
\end{proof}
To complete the proof of
 Proposition
 \ref{prop:NormalForm}, we need
more precise information of the change of coordinates. Define
\[
\Phi=\Phi_1\circ\Phi_2: (\wh I,\wh
s,\wh p,\wh q,\wh g)\
\longmapsto \ (I,s,p,q,g).
\]

\begin{lemma}\label{lemma:NormalFormChangeOfCoordinates}
The change $\Phi$ is of the form
 \[
 \begin{split}
  I =&\, \wh I+e_0 M_1^I+e_0^2  M_2^I+
\OO\left(e_0^{3}\right)\\
  s =& \, \wh s+e_0 M_1^s+e_0^2 M_2^s
  +\OO\left(e_0^{3}\right)\\
  p = &\, \wh p+e_0 M_1^p
+e_0^2 M_2^p+\OO\left(e_0^{3}\right)\\
q =& \, \wh q+e_0 M_1^q +e_0^2 M_2^q
  +\OO\left(e_0^{3}\right)
 \end{split}
 \]
 where
 \[
  M_1^I=\pa_s W_1,\quad   M_1^s=-\pa_I W_1,\quad  M_1^p=-\pa_q W_1\quad  \text{and}\quad  M_1^q=\pa_p W_1,
 \]
and $W_1$ is the function defined in Lemma
\ref{lemma:NormalForm:FirstOrderCohomological}. Moreover, 
 \begin{equation}\label{def:propsM1M2}
\NNN(M_1^z)=\{\pm1\},\quad  M_2^z=\OO(1)\quad \text{ and }\quad \NNN(M_2^z)=\{0,\pm1,\pm
2\}\quad\,\,z=I,p,q,s.
\end{equation}
The inverse change is of the same form, that is
 \[
 \begin{split}
  \wh I =& I-e_0 M_1^I
  +e_0^2 \wt M_2^I+\OO\left(e_0^{3}\right)\\
  \wh s =& s-e_0 M_1^s+e_0^2 \wt
M_2^s+\OO\left(e_0^{3}\right)\\
  \wh p = &p-e_0 M_1^p+e_0^2 \wt
M_2^p+\OO\left(e_0^{3}\right)\\
  \wh q =& q-e_0 M_1^q+e_0^2 \wt
M_2^q+\OO\left(e_0^{3}\right)
 \end{split}
 \]
 and
 \[
 \wh q\wh p=qp-e_0 M_1^r+e_0^2\wt  M_2^r+\OO(e_0^3).
\]
The terms $\wt M_2^z$  satisfy the same properties as   $M_2^z$ given in \eqref{def:propsM1M2}.
\end{lemma}
The proof of this lemma is analogous to the proof of Lemma 4.4 in \cite{GuardiaK15}
%
%
%

\subsection{The first order of the separatrix map: Proof of Proposition
\ref{prop:Melnikov}}\label{app:Melnikov}
The derivation of the formulas for $\al_i(I)$ follows the same lines as for
the analogous terms in \cite{Treschev02a}. Indeed, the homoclinic orbits $\ga_I^i(\sigma)$ considered in Proposition \ref{prop:Melnikov} are asymptotic to the periodic orbit $\la_I(\sigma)$. Note that, if we choose parameterizations that such that both satisfy $g=g_0+\sigma$, one has that $\|\ga_I^i(\sigma)-\la_I(\sigma)\|\to 0$ as $\sigma \to\pm \infty$ (recall that $\la_I$ and $\ga_I^i$ only contain the components $(\ell, L, g, I)$). This convergence  does not take place in the $t$ component since there may be a time shift. Indeed, the parameterization for the $t$ component is of the form
\[
 t=t_0+\la_I(\sigma), \quad  t=t_0+\ga^i_I(\sigma)\qquad \text{ with }\quad \la_I(0)=\ga^i_I(0).
\]
Then, one can easily see (see \cite{FejozGKR15} for all the details) that the functions $\al^i_\pm (I)$ introduced in \eqref{def:omega} satisfy
\begin{equation}\label{def:timeshift}
 \al_i^\pm (I)=\lim_{\sigma\to\pm\infty}\left(\ga^i_I(\sigma)-\la_I(\sigma)\right).
\end{equation}
Now we prove that these functions are those appearing in the formulas for the gluing map introduced in Lemma \ref{lemma:GluingOriginal}. The first step is to point out that the  shifts  $\al^i_\pm(I)$ in the $t$ variable as defined in \eqref{def:timeshift} do not change after applying the changes of coordinates from Lemmas \ref{lemma:ChangePO} and \ref{lemma:MoserNF}. Indeed, the functions $B$ and $T$ involved in those changes vanish on the periodic orbits and therefore they tend to zero when evaluated at the homoclinic orbits with $\sigma \to\pm\infty$.

Now we analyze the Moser normal form Hamiltonian \eqref{def:HamCircFinal}. The periodic orbit in Moser coordinates is just
\[
I=I_0,\quad s^{\PO}=s_0^{\PO}+\nu(I)\sigma,\quad p^{\PO}(\sigma)=0,\quad q^{\PO}(\sigma)=0,\quad  g=g_0+\sigma,
\]
and the
homoclinic orbit $\ga_I^i$ with initial condition at the fixed homoclinic point  has become the following two orbits, depending whether we extend the Moser coordinates along the stable or the unstable invariant manifolds,
\begin{align*}
I&=I, & s^+&=s_0^++\nu(I)\sigma,&  p^+(\sigma)&=0, &q^+(\sigma)&=q^i(I)e^{\la(I)\sigma},& g&=g_0+\sigma\\
I&=I, & s^-&=s_0^-+\nu(I)\sigma,& p^-(\sigma)&=p_0^i(I)e^{-\la(I)\sigma}, & q^-(\sigma)&=0,& g&=g_0+\sigma.
\end{align*}
By definition, we have that $s_0^\pm-s_0^\PO=\al_i^\pm(I)$. The gluing map $\GG_0^i$ maps one trajectory to the other. Thus, if we call $\GG_{0,s}^i$ to the $s$-component of the gluing map, we have
\[
 \GG_{0,s}^i(I,s,0,q^i(I))=s+\al_i (I).
\]
The proof of the first statement of Proposition \ref{prop:Melnikov} follows by applying fundamental theorem of calculus.

Now we derive the formulas for $B^i$. To simplify notation we omit the dependence on  $i$ in all functions below.  We rely on the definition of the function $\MM^{I,1}$ in Lemma \ref{lemma:GluingElliptic}, that is
\[
 \MM^{I,1}(I,s)= M_1^{I}(I,s,0,q(I),0)-M_1^{I}(I,s+\al(I),p_0(I),0,0),
\]
where $M_1^I$ is the function introduced in Proposition \ref{prop:NormalForm}. 


We look for formulas for this function. To this end, we need to go into the
proof of Proposition \ref{prop:NormalForm} in Section \ref{sec:ProofNormalForm}. As stated in
Lemma \ref{lemma:NormalFormChangeOfCoordinates}, the function $M_1^I$ is
defined as $M_1^I=\pa_s W_1$ where $W_1$ is the function introduced in the
proof of Lemma \ref{lemma:NormalForm:FirstOrderCohomological}
as the sum 
\[
W_1=W_1^0+W_1^{0,p}+W_1^{0,q}+\OO(pq),
\]
(see \eqref{def:SeriesW1}). Following this sum, we can split
$\MM^{I,1}$ as a sum and analyze the first three terms. 

We start with the first one. As explained in the proof of Lemma
\ref{lemma:NormalForm:FirstOrderCohomological}, $W_1^0=W_1^0(I,s,g)$. Then, the
first term is given by
\[
 \MM^{I,1}(I,s)= \pa_s W_1^{0}(I,s,0)-\pa_s W_1^{0}(I,s+\al(I),0)
\]
To look for formulas for this term, we use that, as explained in the proof of
Lemma \ref{lemma:NormalForm:FirstOrderCohomological},  $W_1^0$
satisfies
\[
 \nu(I)\pa_s W_1^{0}(I,s,g)+\pa_g  W_1^{0}(I,s,g)+\KK_1^{(0)}=0.
\]
Note that this equation is just \eqref{def:homo0} with $pq=0$.

Since $\NNN(\KK_1^{(0)})=\NNN(\KK_1)$, Lemma
\ref{lemma:EllipticHamReparamHarmonic} implies $\NNN(\KK_1^{(0)})=\{\pm 1\}$
and the same happens for $W_1^{0}$. Writing $W_1^{0}$ as
\[
 W_1^{0}(I,s,g)=A(I,g) e^{\im s}+\ol{A}(I,g)e^{-\im s}
\]
we know that $A$ satisfies
\[
\im\nu(I) A(I,g)+\pa_g  A(I,g)+\KK_1^{(0)}(I,g)=0,
\]
which has a unique solution periodic in $g$, which is given by
\[
 A(I,g)=A_0(I) e^{-\im\nu(I)g}-e^{-\im\nu(I)g}\int_0^g
\KK_1^{(0)}(I,\sigma)e^{\im\nu(I)\sigma}d\sigma
\]
with
\[
 A_0(I)=\frac{1}{1-e^{2\pi \im \nu(I)}}\int_0^{2\pi}
\KK_1^{(0)}(I,\sigma)e^{\im\nu(I)\sigma}d\sigma.
\]
Therefore
\[
 \MM^{I,1}(I,s)=  \MM^{I,1,+}(I)e^{\im s}+\ol{ \MM^{I,1,+}}(I)e^{-\im s}
\]
with
\[
 \MM^{I,1,+}(I)=\im A_0(I) \left(1-e^{\im\al (I)}\right)=\frac{1-e^{\im\al
(I)}}{1-e^{2\pi \im \nu(I)}}\int_0^{2\pi}
\KK_1^{(0)}(I,\sigma)e^{\im\nu(I)\sigma}d\sigma.
\]
Now it only remains to write the integral in terms of the original coordinates.
Note that we are integrating the Hamiltonian $\KK$ at $p=q=0$, that is, along
the periodic orbit. The Moser normal form change of coordinates (Lemma
\ref{lemma:MoserNF}) is the identity on the periodic orbit. Thus, undoing the change of coordinates done at Lemma \ref{lemma:ChangePO},
 one can see that
\[
 \MM^{I,1,+}(I)=\im\frac{1-e^{\im\al
(I)}}{1-e^{2\pi \im \nu(I)}}\int_0^{2\pi}
\frac{\Delta H_\eell^{1,+}\circ\la_I(\sigma)}{-1+\mu \pa_G\Delta
H_\ccirc\circ\la_I(\sigma)}
e^{\im\wt\la_I(\sigma)}d\sigma,
\]
where $\la_I(\sigma)$,
$\wt\la_I(\sigma)$ and $\Delta H_\eell^{1,+}$  are the functions used in
Proposition \ref{prop:Melnikov}. This is just the function $B_\inn$ introduced in Proposition \ref{prop:Melnikov}. The function $B_\out$ can be obtained by analyzing the functions $W_1^{0,p}$ and $W_1^{0,q}$ and interpreting them as Menikov functions. This derivation follows exactly the same lines as the analogous derivation in \cite{Treschev02a}.

\section{The normally hyperbolic invariant lamination}\label{sec:NHIL}
\subsection{Existence of the normally hyperbolic invariant lamination}
In this section we construct a Normally Hyperbolic Invariant Lamination for RPE3BP. We do it in two steps. First we analyze the Poincar\'e map associated to the circular problem \eqref{def:HamDelaunayCirc}. This system is autonomous and hence $I$ is a constant of motion. Therefore building the lamination is reduced to proving the existence of a horseshoe at each level $I=\text{constant}$ and showing its regularity with respect to $I$. We prove its existence by an isolating block technique.
As a second step, we show the existence of the NHIL for  (the Poincar\'e map of) the elliptic problem \eqref{def:HamDelaunayRot} with $0<e_0\ll 1$. This is a regular perturbation problem and therefore the existence of the lamination is a consequence of normal hyperbolicity.

A normally hyperbolic invariant lamination generalizes the concepts of hyperbolic sets and normally hyperbolic invariant manifolds. Here we will consider a less general definition than that of~\cite{HirschPS77}, which will be enough for our purposes.
It follows the ideas in~\cite{KaloshinZ15}. 

First, we define normally hyperbolic invariant manifold.

\begin{definition}\label{def:NHIM}
Given a $C^1$ map $F$ on a manifold $M$, we say that  that a submanifold 
$N\subset M$ is normally hyperbolic for the map $F$ if it is $F$--invariant and 
one can split $T_NM$ into three invariant subbundles,
\[
 T_xM=T_xN\oplus E_x^s\oplus E_x^u,\qquad \text{for all }\quad x\in N
\]
which satisfy
\[
 \begin{split}
  \|DF^n(x)|_{E^s}\|\leq C \la^n\qquad \text{ for }n\geq 0\\
  \|DF^n(x)|_{E^u}\|\leq C \la^{|n|}\qquad \text{ for }n\leq 0\\
  \|DF^n(x)|_{TN}\|\leq C \eta^{|n|}\qquad \text{ for }n\in\NN
 \end{split}
\]
for some $C>0$, $\la<1$ and $\eta\geq 1$ such that $\eta\la<1$.
\end{definition}

The concept of normally hyperbolic invariant lamination is a generalization of 
this definition. To this end, we define the following.

\begin{definition}
\label{def:Shift}
Let $\Sigma = \{0,1\}^{\ZZ}$. We define the Bernoulli  
shift $\sigma:\Sigma \to \Sigma$ as $(\sigma\omega)_k=\omega_{k+1}$.
\end{definition}

\begin{definition}\label{def:NHIL}
Consider a $C^1$ map $F$ on a manifold $M$. We say that a closed set $\Xi$ 
is a normally hyperbolic lamination if 
\[
 \Xi=\bigcup_{\omega\in\Sigma}\Xi_\omega
\]
where $\Xi_\omega$ are $C^1$ manifolds which satisfy 
$F(\Xi_\omega)=\Xi_{\sigma\omega}$ and is normally hyperbolic in the following 
sense. There exist constants  $C>0$, $\la<1$ and $\eta\geq 1$ satisfying 
$\eta\la<1$ and subbundles $E^u$ and $E^u$ such that for any $\omega\in\Sigma$ 
and $x\in\Xi_\omega$,
\[
 T_xM=T_x\Xi_\omega\oplus E_x^s\oplus E_x^u
\]
which are invariant and satisfy 
\[
 \begin{split}
  \|DF^n(x)|_{E^s}\|\leq C \la^n\qquad \text{ for }n\geq 0\\
  \|DF^n(x)|_{E^u}\|\leq C \la^{|n|}\qquad \text{ for }n\leq 0\\
  \|DF^n(x)|_{T\Xi_\omega}\|\leq C \eta^{|n|}\qquad \text{ for }n\in\NN.
 \end{split}
\]
Finally, we say that $\Xi$ is a  normally hyperbolic lamination weakly invariant under $F$ if there exist a neighborhood $U$ of $\Xi$ such that, if there exist $x\in\Xi$ such that $F(x)\not\in \Xi$ implies $F(x)\not\in U$.
\end{definition}

The next two theorems prove the existence of such objects. First, Theorem 
\ref{thm:NHILCircular} for the RPC3BP (Hamiltonian 
\eqref{def:HamDelaunayCirc}) and then Theorem \ref{thm:NHILElliptic} for the 
RPE3BP with $0<e_0\ll 1$ (Hamiltonian \eqref{def:HamDelaunayRot}).

\begin{theorem}\label{thm:NHILCircular}
Fix $0<\rr\ll 1$ and consider $\rr$-neighborhoods $B^i_\rr$ of the homoclinic 
channels $\{\wt\CCC_0^i\}_{I\in \II}$ (see \eqref{def:channelsPoincarecirc}) of  the Poincar\'e map of the circular problem 
\eqref{def:HamDelaunayCirc}. Then, the  separatrix maps $\SM_0$ associated to this 
problem given in  Theorem \ref{thm:FormulasSMii} (with $e_0=0)$ have a NHIL 
denoted by $\LL_0$ satisfying $\LL_0\subset B^1_\rr\cup B^2_\rr$. That is,
\begin{itemize}
\item
The set $\LL_0$ is invariant: there exists an embedding
\[
 L_0:\Sigma\times\TT\times \II\to B^1_\rr\cup B^2_\rr\qquad L_0(\omega, I, 
s)=(P_0(\omega,I), Q_0 (\omega,I),I,s)
\]
and a function $\beta (\omega,I)$, which are $C^3$ in $I$ and $\vartheta$-H\"older in 
$\omega$, for some $\vartheta\in (0,1)$ independent of $\rho$, such that
\[
 \SM_0\left(I,s,P_0(\omega,I), Q_0(\omega,I)\right)=\left(I,s+\beta(\omega,I), P_0({\sigma\omega},I), Q_0 ({\sigma\omega},I)\right).
\]
\item The set $\LL_0$ is normally hyperbolic.
\end{itemize}
Moreover, there exists $\zeta'>\zeta_\ast>0$ such that
\[
\rr^{\zeta'}\lesssim  \inf_{z\in \LL_0, \ z'\in \wt\CCC^i_0,\ i=1,2}\mathrm{ dist}(z,z')\leq  \sup_{z\in \LL_0, \ z'\in \wt\CCC^i_0,\ i=1,2}\mathrm{ dist}(z,z')\lesssim \rr^{\zeta_\ast}.
\]
\end{theorem}
This theorem is proven in Section \ref{sec:NHILCircular}.

For $e_0>0$ small enough the NHIL persists for the Hamiltonian system associated to  \eqref{def:HamDelaunayRot} and it is smooth with respect to $e_0$.
\begin{theorem}\label{thm:NHILElliptic}
 Fix $0<\rr\ll 1$ and consider $\rr$-neighborhoods $B^j_\rr$ of the homoclinic 
channels $\{\wt\CCC^i_{e_0}\}_{I\in \II}$ (see Theorem \ref{thm:NHILElliptic}) of  the Poincar\'e map of the elliptic problem 
\eqref{def:HamDelaunayRot}. Then, the  separatrix maps $\SM_{e_0}$ associated to this 
problem given in  Theorem \ref{thm:FormulasSMii} have a NHIL denoted by 
$\LL_{e_0}$ which is $e_0$-close to the NHIL $\LL_0$ obtained in Theorem 
\ref{thm:NHILCircular} and  satisfies $\LL_{e_0}\subset B^1_\rr\cup B^2_\rr$. That 
is,
\begin{itemize}
\item
The set $\LL_{e_0}$ is invariant: there exists an embedding
\[
 L_{e_0}:\Sigma\times\TT\times \II\to B^1_\rr\cup B^2_\rr\qquad L_{e_0}(\omega, I, s)=(I,s, P_{e_0}(\omega,I,s), Q_{e_0} (\omega,I,s))
\]
and functions $S_{e_0}(\omega,I,s)$ and $J_{e_0}(\omega,I,s)$, which are 
$C^3$ in $(I,s,e_0)$ and $\vartheta$-H\"older in $\omega$, for some $\vartheta\in (0,1)$ independent of $\rho$,  such that
\begin{multline}
 \SM_{e_0}\left(I,s, P_{e_0}(\omega,I,s), Q_{e_0}(\omega,I,s)\right) \\
 \label{eq:inveq_lamination_elliptic}
 =\left(\II_{e_0}(\omega,I,s),\SSS_{e_0}(\omega,I,s), P_{e_0}\circ\FF_{\iinn}(\omega,I,s), Q_{e_0}\circ\FF_{\iinn}(\omega,I,s)\right)
\end{multline}
with
\begin{equation*}
\FF_{\iinn}(\omega,I,s)=(\sigma\omega,  \II_{e_0}(\omega,I,s), \SSS_{e_0}(\omega,I,s))
\end{equation*}
and
\[
\begin{split}
 \II_{e_0}(\omega,I,s)&=I+J_{e_0}(\omega,I,s)\\
\SSS_{e_0}(\omega,I,s)&=s+\beta(\omega,I)+S_{e_0}(\omega,I,s).
\end{split}
\]
\item The set $\LL_{e_0}$ is normally hyperbolic.
\end{itemize}
Moreover,
\begin{itemize}
\item There exists $\zeta'>\zeta_\ast>0$ such that
\[
\rr^{\zeta'}\lesssim  \inf_{z\in \LL_0, \ z'\in \wt\CCC^i_0,\ i=1,2}\mathrm{ dist}(z,z')\leq  \sup_{z\in \LL_0, \ z'\in \wt\CCC^i_0,\ i=1,2}\mathrm{ dist}(z,z')\lesssim \rr^{\zeta_\ast}.
\]
\item The functions $P_{e_0}$, $Q_{e_0}$ are of the form
\[
(P_{e_0}, Q_{e_0})=(P_0, Q_0)+e_0(P_1, Q_1)+\OO(e_0^2)\quad \text{with}\quad \NNN(P_1)=\NNN(Q_1)=\{\pm 1\},
\]
where $(P_0, Q_0)$ are the functions obtained in Theorem \ref{thm:NHILCircular}.
\item The functions $S_{e_0}$ and $J_{e_0}$ are of the form
\[
\begin{split}
J_{e_0}=&\,e_0J_1+e_0^2J_2+\OO(e_0^3)\qquad \text{with}\qquad \NNN(J_1)=\{\pm 1\},\,\, \NNN(J_2)=\{0,\pm 2\},\\
S_{e_0}=&\, e_0S_1+\OO(e_0^2)\qquad \text{with}\qquad \NNN(S_1)=\{\pm 1\}
\end{split}
\]
and\footnote{Note that in the formula below we are abusing notation. The leading term in the order $e_0$ in the separatrix map, $\MM^{I,1}_i$, only depends on $(I,s)$ and the homoclinic channel, which now is labeled by $\omega_0$. On the contrary, $\MM^{I,1}_{i,pq}(I,s, p,q)$ introduced in Theorem \ref{thm:FormulasSMii} depends also on $p$ and $q$. When pull backed to the invariant lamination, this implies that the corresponding term depends on the full $\omega$. Accordingly, we denote this term by $\MM^{I,1}_{\omega,pq}(I,s)$.}
\[
J_{1}(\omega,I,s) = \MM^{I,1}_{\omega_0}(I,s)+\MM^{I,1}_{\omega,pq}(I,s),
\]
where the right hand side functions are defined in Theorem \ref{thm:FormulasSMii}
\end{itemize}
\end{theorem}

\subsection{Dynamics on the normally hyperbolic invariant lamination}\label{sec:RandomCylinder}
Theorem \ref{thm:NHILElliptic} induces a map on the  invariant lamination constructed, which can be further analyzed by relying on the formulas of the separatrix map  obtained in 
Theorem \ref{thm:FormulasSMii}. Let us define the  induce a map as
\begin{equation}\label{def:laminationmap}
\begin{split}
 \FF:\,\, \Sigma\times\TT\times [I_-,I_+] &\longrightarrow  \Sigma\times\TT\times \RR\\
(\omega, s,I)&\mapsto \left(\sigma\omega, \FF_\omega (s,I)\right)
 \end{split}
 \end{equation}

By Theorem \ref{thm:NHILElliptic}, the leaves of the NHIL, which  are cylinders with
boundaries,  are  localized  $\rr^{\ero_*}$-close to the homoclinic channels.
Since both $e_0$ and $\rr$ are taken small, $e_0\ll\rr$, we have ``good''
expansions on the separatrix maps restricted into the weakly invariant lamination.
%
\begin{lemma}\label{lemma:laminationmap}
The separatrix map obtained in Theorem \ref{thm:FormulasSMii} induces a map 
$\FF$ of the form \eqref{def:laminationmap}
on the lamination obtained in Theorem \ref{thm:NHILElliptic} with
\begin{equation}\label{def:stochasticmap}
 \FF_\omega: \begin{pmatrix} I\\ s\end{pmatrix}\to \begin{pmatrix} I+
e_0\AAA_\omega(I,s)+e_0^2\BB_\omega(I,s)+\OO_\omega\left(e_0^{2+a}\right)\\
s+\beta_\omega(I)+e_0\DD_\omega(I,s)+\OO_\omega\left(e_0^{1+a}\right)
\end{pmatrix}
\end{equation}
where $\beta_\omega$ satisfies
\begin{equation}\label{def:Omega}
\beta_\omega(I)=\beta^0_{\omega_0}(I)+\eta_{\omega}(I),\qquad
\beta^0_{\omega_0}(I)=\nu(I)\ol g+\alpha_{\omega_0}(I)
\end{equation}
and the other functions satisfy $\NNN( \AAA_\omega)=\NNN(\DD_\omega)=\{\pm 1\}$ and $\NNN(\BB_\omega)=\{0,\pm 2\}$,
\begin{equation}\label{def:ExpansionCylinderMap}
 \begin{split}
  \AAA_\omega(I,s)&=\MM^{I,1}_{\omega_0}(I,s)+\RRR^{I,1}_{\omega}(I,s)\\
  \BB_\omega(I,s)&=\MM^{I,2}_{\omega_0}(I,s)+\RRR^{I,2}_{\omega}(I,s)\\
  \DD_\omega(I,s)&=\wt\MM^{s}_{\omega_0}(I,s,\ol
g)+\RRR^{s}_{\omega}(I,s).
 \end{split}
\end{equation}
which satisfy
\[
 \NNN(\AAA_\omega)=\NNN(\DD_\omega)=\{\pm 1\}\quad \text{ and }\quad \NNN(\BB_\omega)=\{\pm 0,\pm 2\}.
\]
The functions $\MM^\ast_{\omega_0}$ and $\alpha_{\omega_0}$ are those
introduced in Theorem \ref{thm:FormulasSMii} (see also Proposition
\ref{prop:Melnikov}). $\nu$ is defined in \eqref{def:nu} and the time $\ol g$ is independent of $e_0$ and satisfies $\ol g\sim |\log \rr|$.  Moreover, the remainders in formulas \eqref{def:Omega} and
\eqref{def:ExpansionCylinderMap} satisfy the following estimates.
\[
\begin{split}
\left\|\RRR^{I,j}_{\omega}\right\|_{\Sigma\times\TT\times[I_-,I_+]}
&\lesssim\rr^{\ero_*}\\
\left\|\eta_{\omega}\right\|_{\Sigma\times\TT\times[I_-,I_+]},
\left\|\RRR^{s}_{\omega}\right\|_{\Sigma\times\TT\times[I_-,I_+]}
&\lesssim\rr^{\ero_*}|\log\rr|.
\end{split}
\]
Moreover, there exists $\vartheta\in (0,1)$, such that the functions $\RRR^{I,j}_{\omega}$, $\RRR^{s}_{\omega}$, $\eta_{\omega}$ are $\vartheta$-H\"older with respect to $\omega$.
\end{lemma}


Since $\rr\ll 1$, the formula for $\beta_{\omega}$ in \eqref{def:Omega}
implies that each of the maps encounter many resonances, which are
 $|\log \rr|\ii$-close.
Nevertheless, they will not play any role in
the random map analysis since, thanks to Ansatz \ref{ansatz:Melnikov:1} below, the  maps have
different twist
condition and therefore they are not simultaneously resonant. 

 Ansatz \ref{ansatz:Melnikov:1} is crucial in the analysis of the stochastic behavior of the map $\FF_\omega$, done  in the companion paper \cite{stochastic23}. We include it here, even though it is not used in the present paper,  since we verify it numerically in  Appendix \ref{sec:nondegeneracy}.


\begin{ansatz}\label{ansatz:Melnikov:1}
The following is satisfied.
\begin{itemize}
\item The functions $\alpha_i, \ i=1,2$ defined in Proposition \ref{prop:Melnikov}, satisfy
\[
 |\alpha_i(I)|\leq 100\mu\quad \textup{ for all}\,\,I\in [I_-,I_+].
\]
\item For all $I \in
[I_-,I_+]$, we have 
\[
 \alpha_1(I)-\alpha_2(I)\ {\bf \neq }\ 0.
\]
\item 
We define the first order of a certain variance, given by
		\begin{equation}\label{def:DriftVarFirstOrder}
	\bms_0^2(I)=2\E_\omega\left|\MM^{I,1}_{\omega_0}(J)-\E_\omega\MM^{I,1}_{\omega_0}(I)\frac{1-e^{
			\im\beta^0_{\omega_0}}(I)}{1-\E_\omega\left(e^{
			\im\beta^0_{\omega_0}(I)}\right)}\right|^2,
\end{equation}
where $\E_\omega$ is the expectation with respect to the Bernoulli measure on $\Sigma$.
We assume that it satisfies
\[
\bms_0^2(I)\  {\bf \neq }\ 0\quad \textup{for all}\, I\in [I_-,I_+].
\]
\end{itemize}
\end{ansatz}

We devote the rest of the section to prove Theorems \ref{thm:NHILCircular} and \ref{thm:NHILElliptic}. They are proven in Sections \ref{sec:NHILCircular} and \ref{sec:NHIL:Elliptic} respectively.
\subsection{The NHIL for the circular problem: Proof of Theorem~\ref{thm:NHILCircular}}\label{sec:NHILCircular}

We first analyze the existence of laminations for the circular problem. Its existence will be a consequence of \cite{BernardKZ16} (Appendix A, see also Theorem~A.4 in~\cite{KaloshinZ15})  which uses the isolating block method. In this section we check that our system satisfies the hypotheses of that theorem, namely, we obtain suitable \emph{blocks} by finding their \emph{centers} and then check the isolating block conditions, that is, existence of contracting and expanding directions and cone conditions.

\subsubsection{Linearization of the separatrix map of the circular
	problem}\label{sec:LinearizationSMCircular}

Theorem~\ref{thm:FormulasSMii:circ} provides formulas
for the four separatrix maps we are considering, $\mathcal{SM}_{0}^{ij}$, $i,j=1,2$.
To simplify notation, we denote the return time \eqref{def:ReturnTime} as
\[
\ol g=2\pi N^{ij}(I,s,p,q),
\]
We consider the $(p,q)$ expansion of  the separatrix map of the
circular problem $\FF^{ij}_0=\mathcal{SM}_{0}^{ij}$. We recall that
\begin{equation}\label{def:SMFirstOrder}
	\FF^{ij}_0=\left(\FF^{ij}_{I,0}, \FF^{ij}_{s,0}, \FF^{ij}_{p,0}, \FF^{ij}_{q,0}\right)
\end{equation}
with
\[
\begin{split}
	\FF^{ij}_{I,0}= &I\\
	\FF^{ij}_{s,0}=&s+\alpha_i(I)+(\nu(I)+\pa_I F(I,pq))\ol g+ a^i_{10}(I)p+
	a^i_{01}(I)(q-q^i(I))+\OO_2(p,q-q^i(I))
\end{split}
\]
and
\[
\begin{split}
	\begin{pmatrix}\FF^{ij}_{p,0}\\\FF^{ij}_{q,0}\end{pmatrix}=&\begin{pmatrix}e^{-\Theta^{i}(I,s,p,
			q)\ol
			g}p^i_0(I)\\0 \end{pmatrix}\\
	&+\begin{pmatrix}
		e^{-\Theta^{i}(I,s,p,q)\ol g}b^i_{10}(I) &e^{-\Theta^{i}(I,s,p,q)\ol g}b^i_{01}(I)\\
		e^{\Theta^{i}(I,s,p,q)\ol g}c^i_{10}(I) &e^{\Theta^{i}(I,s,p,q)\ol g}c^i_{01}(I)
	\end{pmatrix}
	\begin{pmatrix}
		p+ \OO_2(p,q-q^i(I)) \\ q-q^i(I)+ \OO_2(p,q-q^i(I))
	\end{pmatrix}
\end{split}
\]
where, using the definition of $\Theta^{i}$ in~\eqref{def:Theta} and the gluing map
given by Lemma \ref{lemma:GluingOriginal}, $\Theta^{i}$ is independent of $s$ and
satisfies
\begin{equation*}
	\Theta^{i}(I,s,p, q)= \Theta^{i}(I,p,
	q)=\la(I)+\la_2(I)p^i_0(I)
	\left(c^i_{10}(I)p+c^i_{01}(I)(q-q^i(I))\right)+\OO_2(p,q-q^i(I)),
\end{equation*}
for some smooth function $\lambda_2$.
We remark that $\FF^{ij}_{I,0}$, $\FF^{ij}_{s,0}-s$, $\FF^{ij}_{p,0}$ and $\FF^{ij}_{q,0}$  do
not depend on $s$.

\begin{lemma}\label{lemma:DifferentialSMCircular}
	The differential of $\FF_0^{ij}$ satisfies
	\begin{scriptsize}
		\[
		D\FF^{ij}_0=
		\begin{pmatrix}
			1&0&0&0\\
			\dot \alpha_i - a_{01}^i \dot q^i +\dot \nu \ol g+\OO_1\ol g &1 & \dot \la \ol
			g +a^i_{10}+\OO_1\ol
			g& a^i_{01}+\OO_1\ol g \\
			e^{-\Theta^{i}\ol g}(A^i_1+\OO_1 \ol g + \OO_1)&0&  e^{-\Theta^{i}\ol g}(-\la_2 (p_0^i)^2 c^i_{10} \ol g +b^i_{10}+\OO_1\ol g)
			&e^{-\Theta^{i}\ol g}(-\la_2 (p^i_0)^2 c^i_{01} \ol g+ b^i_{01}+\OO_1 \ol g)\\
			e^{\Theta^{i}\ol g} (-c^i_{01}\dot q^i +\OO_1 \ol g )&0&  e^{\Theta^{i}\ol
				g}(c^i_{10}+\OO_1\ol g) &e^{\Theta^{i}\ol g}(c^i_{01}+\OO_1 \ol g)
		\end{pmatrix},
		\]
	\end{scriptsize}
	where $\OO_1=\OO_1(p,q-q^i(I))$, $\OO_2=\OO_2(p,q-q^i(I))$,
	\[
	A^i_1=\dot p^i_0 -b^i_{01} \dot q^i -\pa_I\Theta^{i} p^i_0 \ol g
	\]
	and $\dot a^i_{mn}$ denotes the derivative of $a^i_{mn}$ with respect to $I$.
\end{lemma}

\begin{proof}
	To prove the lemma, it
	is enough to take derivatives in the formulas of the separatrix map in~\eqref{def:SMFirstOrder}.
\end{proof}

Now we can analyze the eigenvectors and eigenvalues of the matrix $
D\FF_0^{ij}$. Recall that, by Theorem~\ref{thm:NHILCircular}, the
function $c^i_{01}$ satisfies $c^i_{01}\neq 0$. We assume
\[
c^i_{01}> 0
\]
One can proceed analogously if one assumes  $c^i_{01}<0$.

The proofs of the following two lemmas are straightforward.
\begin{lemma}\label{lemma:differential0hyperbolic}
	Take $\ol g$ big enough. The matrix
	\[
	\mathcal B^{ij}_0 =\begin{pmatrix}
		e^{-\Theta^{i}\ol g}(-\la_2 (p^i_0)^2 c^i_{10} \ol g +b^i_{10}+\OO_1 \ol g)
		&e^{-\Theta^{i}\ol g}(-\la_2 (p^i_0)^2 c^i_{01} \ol g+b^i_{01}+\OO_1 \ol g)\\
		e^{\Theta^{i}\ol
			g}(c^i_{10}+\OO_1 \ol g) &e^{\Theta^{i}\ol g}(c^i_{01}+\OO_1 \ol g)
	\end{pmatrix}
	\]
	has eigenvalues $\mu^{ij}, (\mu^{ij})^{-1}\in\RR$ with
	\[
	\mu^{ij} =e^{\Theta^{i}\ol
		g}\left(c^i_{01}+\OO_1  \ol g \right)+\OO\left(e^{-\Theta^{i} \ol g}\right)\gg 1.
	\]
	Let $w^{ij}_1, w^{ij}_2$ be two eigenvectors associated to these eigenvalues.
	They can be chosen to satisfy
	\begin{equation}
		\label{eq:w1andw2}
		w^{ij}_1=\begin{pmatrix}
			e^{-2\Theta^{i}\ol
				g}\frac{-\lambda_2 (p^i_0)^2 c^i_{01} \ol g +b^i_{01}}{c^i_{01}} \left(1+\OO_1 + \OO\left( e^{-2\Theta^{i}\ol
				g}\right)\right)\\ 1
		\end{pmatrix}\qquad
		w^{ij}_2=\begin{pmatrix}
			1\\ -\frac{c^i_{10}}{c^i_{01}}+\OO_1+\OO\left( e^{-\Theta^{i}\ol
				g}\right)
		\end{pmatrix}.
	\end{equation}
	
\end{lemma}

\begin{lemma}\label{lemma:differential0}
	Take $\ol g$ big enough. Then, the matrix $D\FF^{ij}_0$
	has as eigenvalues
	\[
	\la^{ij}_1=\mu^{ij}=e^{\Theta^{i}\ol
		g}\left(c^i_{01}+\OO_1 \ol g \right)+\OO\left(e^{-\Theta^{i}\ol
		g}\right)\gg 1,
	\quad \la^{ij}_2=(\la^{ij}_1)^{-1}.
	\]
	The associated eigenvectors can be chosen to be
	\[
	v^{ij}_1=\begin{pmatrix}0\\
		\frac{e^{-\Theta^{i}\ol
				g}}{c^i_{01}}\left(a^i_{01}+\OO_1 \ol g + \OO\left( e^{-\Theta^{i} \ol g}\right)\right)
		\\w^{ij}_{11}\\w^{ij}_{21}
	\end{pmatrix}, \quad
	v^{ij}_2=\begin{pmatrix}0\\ -\dot \la\ol g-a^i_{10} + \frac{c_{10}^i}{c^{i}_{01}}a_{01} + \OO_1 \ol g + \OO\left(e^{-\Theta^{i}
			g} \ol g \right)
		\\w^{ij}_{12}\\w^{ij}_{22}\end{pmatrix},
	\]
	where $w^{ij}_1$ and $w^{ij}_2$ are the eigenvectors obtained in Lemma~\ref{lemma:differential0hyperbolic}.
\end{lemma}

\begin{remark}
	The matrix $D\FF^{ij}_0$ has also an eigenvalue equal to 1 with an associated
	Jordan Block and eigenvector $v_2=(0,1,0,0)^T$. Nevertheless, it is not used
	in the isolating block construction.
\end{remark}

\begin{remark}
	Since $\FF^{ij}_0-(0,s,0,0)$ does not depend on $s$, neither $D\FF^{ij}_0$ nor
	its eigenvectors, $v^{ij}_1$ and $v^{ij}_2$, depend on $s$.
\end{remark}

%

\subsubsection{Center of the blocks}\label{sec:CenterBlocksCircular} We compute
the center of
the isolating blocks for the separatrix map of the circular problem,
given in \eqref{def:SMFirstOrder}. We have certain freedom to choose
the distance of the center from the homoclinic points. We will use
this freedom later.

Recall that for the circular problem, the function
$\Theta^{i}$ satisfies
\begin{equation}\label{def:lambdaIenTheta}
\Theta^{i}(I,p,q)=\la(I)+\OO(p,q-q^i(I)),
\end{equation}
where $\la(I)>0$ has been defined in Lemma \ref{lemma:MoserNF}. 
We fix $I^*\in [I_-,I_+]$ and we define $\la^*=\la(I^*)$ for some $I^*\in\II$.
We define the function
\begin{equation}\label{def:alpha}
	\ero(I)=\frac{\la(I)}{\la^*}.
\end{equation}

We start by computing the center of the blocks  for
$\FF^{11}_{0}$ and $\FF^{22}_{0}$.
\begin{lemma}\label{lemma:CenterBlocks}
	Fix $(I,s)\in [I_-,I_+]\times\TT$ and $0<\rr\ll1$ such that
	\[
	\frac{|\log\rr|}{\la^*}\in\ZZ.
	\]
	Consider the separatrix map in~\eqref{def:SMFirstOrder}. Then, the
	equation
	\[
	\begin{split}
		p&=\FF^{ii}_{p,0}(I,s,p,q)\\
		q&=\FF^{ii}_{q,0}(I,s,p,q)
	\end{split}
	\]
	has a solution of the form $(p_*^{ii}(I),
	q_*^{ii}(I))=(p^{ii}_\ccirc(I)\rr^{\ero(I)},q^i(I)+
	q^{ii}_\ccirc(I)\rr^{\ero(I)})$ with
	\[
	\begin{split}
		p^{ii}_\ccirc(I)&=p^{i}_0(I)+\OO\left(\rr^{\ero(I)}\log\rr\right)\\
		q^{ii}_\ccirc(I)&=\frac{1}{c^{i}_{01}(I)}\left(q^i(I)-c^{i}_{10}(I)p^{i}
		_0(I)\right)
		+\OO\left(\rr^ { \ero(I)}\log\rr\right).
	\end{split}
	\]
	Moreover, the time $\ol g$ needed to achieve this transition is
	given by
	\[\ol g= \frac{|\log\rr|}{\la^*}.\]
\end{lemma}

	The proof is straightforward setting up a fixed point argument.

The curve $(p_*^{ii}(I), q_*^{ii}(I))$ defines the parameterization of the center
of the
blocks for all $s\in\TT$. That is, the centers do not depend on $s$. This
is due to the particular form of the circular problem. 
Note
also that when $I$ varies, the distance of the center of the block with respect
to the homoclinic point changes. This is due to the dependence on $I$ of the
eigenvalue $\lambda(I)$.

The same happens for the center of the blocks of
$\FF^{12}_0$ and $\FF^{21}_0$.
\begin{lemma}\label{lemma:CenterBlocks:ij}
	Fix $i,j=1,2$ with $i\neq j$,  $(I,s)\in [I_-,I_+]\times\TT$ and $0<\rr\ll1$ such
	that
	\[
	\frac{|\log\rr|}{\la^*}\in\ZZ.
	\]
	Consider the separatrix map in~\eqref{def:SMFirstOrder}. Then, the
	equation
	\[
	\begin{split}
		p&=\FF^{ji}_{p,0}\circ \FF^{ij}_0(I,s,p,q)\\
		q&=\FF^{ji}_{q,0}\circ \FF^{ij}_0(I,s,p,q)
	\end{split}
	\]
	has a solution of the form $(p_*^{ij}(I),
	q_*^{ij}(I))=(p^{ij}_\ccirc(I)\rr^{\ero(I)},q^i(I)+
	q^{ij}_\ccirc(I)\rr^{\ero(I)})$ with
	\[
	\begin{split}
		p^{ij}_\ccirc(I)&=p^{j}_0(I)+\OO\left(\rr^{\ero(I)}\log\rr\right)\\
		q^{ij}_\ccirc(I)&=\frac{1}{c^{i}_{01}(I)}(q^j(I)-c_{10}^j (I)p_0^j(I))
		+\OO\left(\rr^ { \ero(I)}\log\rr\right).
	\end{split}
	\]
	Moreover, the time $\ol g$ needed to to achieve this transition is
	given by
	\[\ol g=2 \frac{|\log\rr|}{\la^*}.\]
\end{lemma}
The proof is a direct consequence of applying a fix point argument.

\subsubsection{Isolating block conditions}
\label{sec:Isolating_block_conditions}
To prove the existence of the lamination,
 we need to
define the blocks and prove that they satisfy the cone conditions given in \cite{BernardKZ16,KaloshinZ15} (explained below).

In Lemma \ref{lemma:differential0hyperbolic}, we have computed the eigenvalues
of the $D\FF^{ij}_0$. Using the notation in \eqref{def:alpha} and
Lemma \ref{lemma:CenterBlocks} and taking into account that, from now on,
$e^{\Theta^i \ol g} = \rho^{\ero(I)}$, we have that the hyperbolic  eigenvalues  are $\la^{ij}_1=\mu^{ij}$ and
$\la^{ij}_2=(\mu^{ij})\ii$ with
\[
\mu^{ij}=\frac{c^i_{01}}{\rr^{\ero(I)}}\left(1+\OO\left(\rr^{\ero(I)}\right)\right).
\]
We also consider the eigenvectors of $D\FF^{ij}_0$ given by Lemma
\ref{lemma:differential0} based at the center of the blocks computed in Lemma
\ref{lemma:CenterBlocks}. Thus, we define
\[
v_k^{ij}(I)=v_k(I, p_*^{ij}(I), q_*^{ij}(I))\quad k=1,2.
\]
Using \eqref{eq:GluingSymplectic}, \eqref{def:alpha} and Lemma~\ref{lemma:differential0}, we have
\begin{equation}
	\label{eq:eigenvectors_circular_case}
	v_1(I)=\begin{pmatrix}0\\\OO(\rr^{\ero(I)} )\\
		\OO(\rr^{2\ero(I)}\log \rho)\\1     \end{pmatrix}, \qquad
	v_2(I)=\begin{pmatrix}0\\
		-\dot \la\ol g-a^i_{10} + \frac{c_{10}^i}{c^{i}_{01}}a_{01}  +\OO(\rr^{\ero(I)}\log \rr)
		\\1\\-\frac{c_{10}}{c_{01}}+\OO(\rr^{\ero(I)})\end{pmatrix}.
\end{equation}
We fix $\de>0$, which will be defined in terms of $\rr$ later. Then,
we consider the blocks
\begin{equation}\label{def:stableblocks}
	\begin{split}
		\Pi_{ij}^{s,\kk}=\Big\{(I,s,p,q)=&(I,s, p_*^{ij}(I), q_*^{ij}(I))+c
		v_1^{ij}(I)+ d v_2^{ij}(I):\\
		&I\in [I_-,I_+], s\in\TT, |c|\leq \kk\de^2, |d|\leq
		\kk\de\Big\}.
	\end{split}
\end{equation}
These blocks are parallelograms centered at the points computed in
Lemma \ref{lemma:CenterBlocks} with sides in the direction of the
eigenvectors obtained in Lemma \ref{lemma:differential0}. The
superscript $s$ refers to the fact that this blocks are stretched
along the stable eigenvector $v_2^{ij}$. We introduce
\begin{equation}
	\label{def:borderstableblock}
	\partial^u \Pi_{ij}^{s,\kk} = \{(I,s,p,q) \in \Pi_{ij}^{s,\kk}: |c|
	= \kk \de^2\}.
\end{equation}
Analogously, we can define the unstable blocks
\begin{equation}\label{def:unstableblocks}
	\begin{split}
		\Pi_{ij}^{u,\kk}=\Big\{(I,s,p,q)=&(I,s, p_*^{ij}(I), q_*^{ij}(I))+c
		v_1^{ij}(I)+ d v_2^{ij}(I):\\
		&I\in [I_-,I_+], s\in\TT, |c|\leq \kk\de, |d|\leq
		\kk\de^2\Big\},
	\end{split}
\end{equation}
which are stretched along the unstable eigenvector $v_1$, and
\begin{equation}
	\label{def:borderunstableblock}
	\partial^s \Pi_{ij}^{u,\kk} = \{(I,s,p,q) \in \Pi_{ij}^{s,\kk}: |d|
	= \kk \de^2\}.
\end{equation}
One can define a system of coordinates adapted to these blocks.
Define the map
\begin{equation}\label{def:BlockCoordinates}
	\Psi^{ij}:(I,s,c,d) \mapsto (I,s, p_*^{ij}(I), q_*^{ij}(I))+c
	v_1^{ij}(I)^\top+ d v_2^{ij}(I)^\top.
\end{equation}
Since, in view of Lemma~\ref{lemma:differential0}, the vectors
$(1,0,0,0)^{\top}$, $(0,1,0,0)^{\top}$, $v_1^{ij}$ and $v_2^{ij}$
are linearly independent, we have that $\Psi^{ij}$ is a
diffeomorphism from $\Pi_{ij}^{\sigma,\kk}$ onto its image for~$c$
and~$d$ small enough (that is, if $\delta$ is small enough). As a matter of fact,
since $\Psi^{ij}_I(I,s,c,d) = I$, it is a diffeomorphism for all $(c,d)$. We
introduce its inverse
\begin{equation}\label{def:BlockCoordinatesInverse}
	\Phi^{ij}=\left(\Psi^{ij}\right)^{-1}:
	(I,s,p,q)\mapsto(I_0(I,s,p,q),s_0(I,s,p,q),c^{ij}(I,s,p,q),d^{ij}(I,s,p,q)).
\end{equation}

We express the separatrix map and the blocks in this new system of
coordinates. That is,
\begin{equation}\label{def:separatrixmapNewCoord}
	F_0=\Phi^{ij}\circ \FF_0\circ (\Phi^{ij})^{-1}.
\end{equation}
We also define
\begin{equation}\label{def:blocksNewCoord}
	P_{ij}^{\sigma,\kk}=\Phi^{ij}\left(\Pi_{ij}^{\sigma,\kk}\right),\qquad
	\sigma=u,s
\end{equation}
and
\begin{equation}
	\label{def:borderstableunstableblocksNewCoord} \partial^u
	P_{ij}^{s,\kk}= \Phi^{ij}\left(\partial^u \Pi_{ij}^{s,\kk}\right)
	\quad \text{and} \quad \partial^s P_{ij}^{u,\kk}=
	\Phi^{ij}\left(\partial^s \Pi_{ij}^{u,\kk}\right).
\end{equation}

Let $\pi_j(z)$ denote the projection onto the $j$-component of $z$.

\begin{lemma}
	\label{lem:stableblocksexpansion}
	Let $\delta = \rho^{\ero}$ be small enough. There exist $\kappa,
	\kk_1,\kk_2 >0$ such that,  if $z \in \partial^u P_{ij}^{s,\kk}$,
	\begin{equation}
		\label{ineq:blockcontraction_s} | \pi_3 \circ F_0 (z) | \ge
		\kk_1 \rho^{\ero}
	\end{equation}
	and, if $z\in P_{ij}^{s,\kk}$,
	\begin{equation}
		\label{ineq:blockexpansion_s} | \pi_4 \circ F_0 (z) | \le  \kk_2
		\rho^{2\ero}.
	\end{equation}
	Moreover, $F_0 (\partial^u P_{ij}^{s,\kk})$ is homotopically
	equivalent to $\partial^u P_{ij}^{s,\kk}$.
\end{lemma}

\begin{proof}
	In what follows, we  assume $\kappa < 1$. Given $(I,s,c,d) \in  P_{ij}^{s,\kk}$, 
	we  write $(I,\tilde s, \tilde
	p, \tilde q) = (\Phi^{ij})^{-1}(I,s,c,d)$. Let
	\[
	z^{ij}(I) = (I,s,p_*^{ij}(I),q_*^{ij}(I)) = (\Phi^{ij})^{-1}(I,s,0,0)
	\]
	be the center of the corresponding block. Along the proof, we 
	denote $F_0(\wt z) =\wt  z^+$ and $\FF_0(z) = z^+$. By
	Lemma~\ref{lemma:CenterBlocks}, $F_0(I,s,0,0) = (I^+,s^+,0,0) =
	(I,s^+,0,0)$ that is, $\pi_{3,4} \circ F_0(I,s,0,0) = (0,0)$.
	Introducing
	\begin{equation}
		\label{DeltaIscd}
		\Delta(I,s,c,d)
		=
		\int_0^1 \left[D \FF_0 (z^{ij} +u (c v_1^{ij} +d
		v_2^{ij}))-D \FF_0 (z^{ij})\right]_{\mid(I,s)}\, du
	\end{equation}
	we have  that
	\[
	\begin{aligned}
		\FF_0  \circ (\Phi^{ij})^{-1}(I, &
		s, c, d)   \\
		= &  \FF_0  \circ (\Phi^{ij})^{-1}(I, s, 0, 0) + \FF_0  \circ (\Phi^{ij})^{-1}(I, s, c, d) - \FF_0  \circ (\Phi^{ij})^{-1}(I, s, 0, 0) \\
		= & \FF_0  \circ (\Phi^{ij})^{-1}(I, s, 0, 0) + \int_0^1 D \FF_0 (z^{ij} +u (c v_1^{ij} +d v_2^{ij}))_{\mid(I,s)}\,
		du (c v_1^{ij} +d v_2^{ij})_{\mid(I,s)}  \\
		= & \FF_0  \circ (\Phi^{ij})^{-1}(I, s, 0, 0) + D \FF_0 (z^{ij})(c v_1^{ij} +d v_2^{ij})_{\mid(I,s)}
		+ \Delta(I,s,c,d) (c v_1^{ij} +d
		v_2^{ij})_{\mid(I,s)} \\
		= & z^{ij}(I^+,s^+) + c  \lambda^{ij}_1 v_1^{ij}(I^+,s^+) + d \lambda^{ij}_2 v_2^{ij}(I^+,s^+) + \EE^{ij} (I,s)
	\end{aligned}
	\]
	where
	\begin{multline}
		\label{errorij}
		\EE^{ij} (I,s) = \FF_0  \circ (\Phi^{ij})^{-1}(I, s, 0, 0) - z^{ij}(I^+,s^+)
		+ c \lambda^{ij}_1 (v_1^{ij}(I^+,s^+) -v_1^{ij}(I,s)) \\
		+ d \lambda^{ij}_2( v_2^{ij}(I^+,s^+) -v_2^{ij}(I,s))
		+ \Delta(I,s,c,d)(c v_1^{ij} +d
		v_2^{ij})_{\mid(I,s)}.
	\end{multline}
	We claim that
	\begin{equation}
		\label{boundEEij}
		\begin{aligned}
			|\EE^{ij}_I | & = 0, \\
			|\EE^{ij}_s | & \le \OO (\rho^{2\ero(I)}\log \rho), \\
			|\EE^{ij}_p | & \le \OO(\rho^{3\ero(I)}\log \rho), \\
			|\EE^{ij}_q | & \le  \kk^2\OO( \rho^{\ero(I)}).
		\end{aligned}
	\end{equation}
	Indeed, introducing
	\[
	\begin{aligned}
		\EE_1 & = \FF_0  \circ (\Phi^{ij})^{-1}(I, s, 0, 0) - z^{ij}(I^+,s^+), \\
		\EE_2 & = c \lambda^{ij}_1 (v_1^{ij}(I^+,s^+) -v_1^{ij}(I,s)) + d  \lambda^{ij}_2( v_2^{ij}(I^+,s^+) -v_2^{ij}(I,s)), \\
		\EE_3 & = \Delta(I,s,c,d)(c v_1^{ij} +d v_2^{ij})_{\mid(I,s)},
	\end{aligned}
	\]
	we have that $\EE^{ij} = \EE_1+\EE_2+\EE_3$ and
	\begin{enumerate}
		\item by Lemmas~\ref{lemma:CenterBlocks} and~\ref{lemma:CenterBlocks:ij} and taking into account that $I^+ = I$,
		$\EE_1 =0$,
		\item since $I^+ = I$ and $v_1^{ij}$, $v_2^{ij}$ only depend on $I$,  $\EE_2 = 0$,
		\item since $c v_1^{ij} +d v_2^{ij} = \kk (0, \OO( \rho^{\ero(I)} \log \rho),
		\OO( \rho^{\ero(I)} ), \OO( \rho^{\ero(I)} ))$ and taking into account the expression of $\FF_0$ given by Theorem~\ref{thm:FormulasSMii:circ},
		\[
		\Delta(I,s,c,d)  = \kk
		\begin{pmatrix}
			0 & 0 & 0 & 0 \\
			\OO(\rho^{\ero(I)}\log \rho) & 0 & \OO(\rho^{\ero(I)}\log \rho)  &
			\OO(\rho^{\ero(I)}\log \rho) \\
			\OO(\rho^{2\ero(I)}\log \rho) & 0 & \OO(\rho^{2\ero(I)}\log \rho)  &
			\OO(\rho^{2\ero(I)}\log \rho)  \\
			\OO(\log \rho)  & 0 & \tilde \AAA & \tilde \BB
		\end{pmatrix},
		\]
		where, using that $\partial_p \Theta = \lambda_2 p_0 c_{10} +  \OO_1$,
		$\partial_q \Theta = \lambda_2 p_0 c_{01} + \OO_1$ and the expressions for $v_1^{ij}$ and $v_2^{ij}$ given by Lemma~\ref{lemma:differential0},
		\[
		\tilde \AAA, \tilde \BB =  \OO(1),
		\]
		which implies that
		\[
		\Delta(I,s,c,d)(c v_1^{ij} +d v_2^{ij})_{\mid(I,s)}
		= \kk^2\left(0,
		\OO(\rho^{2\ero(I)}\log \rho),\OO(\rho^{3\ero(I)}\log \rho),\OO(\rho^{\ero(I)})\right).
		\]
	\end{enumerate}
	Combining 1), 2) and 3) we obtain the bound~\eqref{boundEEij}.

	Hence,
	\begin{equation}
		\label{separatrix_in_adapted_coordinates}
		\begin{aligned}
			F_0(I,s,c,d) & = \Phi^{ij} \circ \FF_0 \circ \Psi^{ij}(I,s,c,d)  \\
			& = \Phi^{ij} (z^{ij}(I^+,s^+) + c \lambda^{ij}_1 v_1^{ij}(I^+,s^+) + d \lambda^{ij}_2 v_2^{ij}(I^+,s^+) + \EE^{ij} (I,s)) \\
			& = (I^+,s^+,c \lambda^{ij}_1,d \lambda^{ij}_2) + \widetilde \EE^{ij},
		\end{aligned}
	\end{equation}
	where
	\begin{equation}
		\label{error_in_separatrix_in_adapted_coordinates}
		\begin{aligned}
			\widetilde \EE^{ij} = &  \Phi^{ij} (z^{ij}(I^+,s^+) + c \lambda^{ij}_1 v_1^{ij}(I^+,s^+) + d \lambda^{ij}_2 v_2^{ij}(I^+,s^+) + \EE^{ij} (I,s)) \\
			& - \Phi^{ij} (z^{ij}(I^+,s^+) + c \lambda^{ij}_1 v_1^{ij}(I^+,s^+) + d \lambda^{ij}_2 v_2^{ij}(I^+,s^+)) \\
			= & \int_0^1 D \Phi^{ij} (z^{ij}(I^+,s^+) + c \lambda^{ij}_1 v_1^{ij}(I^+,s^+) + d \lambda^{ij}_2 v_2^{ij}(I^+,s^+) + u \EE^{ij} (I,s))\, du \,\EE^{ij} (I,s).
		\end{aligned}
	\end{equation}
	Since $\Phi^{ij}$ is invertible,  $I^+ = I$, $z^{ij}$, $v^{ij}_k$, $\lambda^{ij}_k$ depend only on $I$
	and $\pi_I v^{ij}_k = 0$, we have that, for $u\in [0,1]$,
	\begin{multline*}
		z^{ij}(I^+,s^+) + c \lambda^{ij}_1 v_1^{ij}(I^+,s^+) + d \lambda^{ij}_2 v_2^{ij}(I^+,s^+) + u \EE^{ij} (I,s) \\ =
		z^{ij}(I) + c \lambda^{ij}_1 v_1^{ij}(I) + d \lambda^{ij}_2 v_2^{ij}(I) + u \EE^{ij} (I,s)
		=
		z^{ij}(  I) + \tilde c  v_1^{ij}(I) + \tilde d  v_2^{ij}( I),
	\end{multline*}
	for some $(\tilde c, \tilde d)$. But, using the bound~\eqref{boundEEij},  we deduce that
	\[
	\tilde c - c  \lambda^{ij}_1 = \OO(\rho^{ \ero(I)}), \qquad \tilde d - d  \lambda^{ij}_2 =
	\OO(\rho^{3 \ero(I)}\log \rho).
	\]
	This implies, by its definition in~\eqref{def:BlockCoordinates} and the formulas for the center of the blocks $z^{ij}$ in Lemma~\ref{lemma:CenterBlocks:ij}
		\begin{equation}
			\label{boundPsi}
			D {\Psi^{ij}}_{\mid(I, s^+,\tilde c, \tilde d)}
			=
			\begin{pmatrix}
				1 & 0 & 0 & 0 \\
				\OO(\rho^{2\ero(I)}\log \rho) & 1 & \OO(\rho^{\ero(I)}) & \OO(\log \rho) \\
				\OO(\rho^{2\ero(I)}\log \rho) &  0 & \OO( \rho^{\ero(I)}) & 1 \\
				\OO( \rho^{\ero(I)}\log \rho) &  0 & 1 & -c_{10}/c_{01}+\OO( \rho^{\ero(I))}
			\end{pmatrix}.
		\end{equation}
	Hence,
		\begin{equation}
			\label{boundPsiminus1}
			D {\Phi^{ij}}_{\mid \Psi^{ij} (I,  s^+,\tilde c, \tilde d)} \\=
			\begin{pmatrix}
				1 & 0 & 0 & 0 \\
				\OO(\rho^{2\ero(I)}\log^2 \rho) & 1 & \OO(\log \rho) & \OO(\rho^{\ero(I)}\log \rho) \\
				\OO( \rho^{\ero(I)}\log \rho) &  0 & c_{10}/c_{01}+\OO( \rho^{\ero(I)}) & 1+\OO( \rho^{\ero(I)}) \\
				\OO(\rho^{2\ero(I)}\log \rho) &  0 & 1 +\OO(\rho^{\ero(I)})& \OO( \rho^{\ero(I)}).
			\end{pmatrix}.
		\end{equation}
	Using these bounds on $D {\Phi^{ij}}_{\mid \Psi^{ij} (\tilde I, \tilde s^+,\tilde c, \tilde d)}$, \eqref{error_in_separatrix_in_adapted_coordinates} and~\eqref{boundEEij}, we have that
	\begin{equation}
		\label{boundtildeEEij}
		|\tilde \EE^{ij}_p |  \le  \kk^2\OO(  \rho^{\ero(I)}), \qquad
		|\tilde \EE^{ij}_q |  \le \kk^2 \OO(\rho^{2\ero(I)}).
	\end{equation}
	Now, if $z \in \partial^u P_{ij}^{s,\kk}$, that is, $|c| =  \kk \delta^2 =
	\kk \rho^{2\ero(I)}$ and $|d| < \kappa \delta = \kk
	\rho^{\ero(I)}$,
	\[
	\begin{aligned}
		|\pi_3 \circ F_0(I,s,c,d)|  & \ge  |c \lambda^{ij}_1| - \kk^2\OO(\rho^{\ero(I)}) \\
		& \ge \kk c^i_{01}  \rho^{\ero(I)} (1+  \kk\OO(1)).
	\end{aligned}
	\]
	Hence, choosing $\kk$ such that $1+  \kk\OO(1) > 0$,  \eqref{ineq:blockcontraction_s}
	follows. Also, if $z\in P_{ij}^{s,\kk}$, that is,
	$|c| = \kk \delta^2 =
	\kk  \rho^{2\ero(I)}$ and $|d| = \kk \delta = \kk \rho^{\ero(I)}$,
	\[
	|\pi_4 \circ F_0(I,s,c,d)|
	\le  d \lambda^{ij}_2 + \kk^2 \OO(\rho^{2\ero(I)})
	\le \frac{\kk}{c^i_{01}} \rho^{2\ero(I)}(1+\kappa \OO(1)),
	\]
	which proves~\eqref{ineq:blockexpansion_s}.
	
	As for $F_0(\partial^u P_{ij}^{s,\kk})$ being homotopically
	equivalent to $\partial^u P_{ij}^{s,\kk}$, we can use the same arguments as
	in~\cite{KaloshinZ15}, considering the lift of $\FF_0$ to the covering space 
	where
	$P_{ij}^{s,\kk}$ becomes simply connected. The image of its boundary, which is a
	hyperplane, is a slightly deformed hyperplane.
\end{proof}

We also have the analogous claim for the unstable blocks. The proof follows the same lines as the proof of Lemma~\ref{lem:stableblocksexpansion}.
\begin{lemma}
	\label{lem:unstableblocksexpansion}
	Let $\delta = \rho^{\ero}$ be small enough. There exist $\kappa,
	\kk_1,\kk_2 >0$ such that,  if $F_0(z) \in \partial^s P_{ij}^{u,\kk}$,
	\begin{equation}
		\label{ineq:blockcontraction}
		| \pi_4 \circ F_0^{-1} (F_0(z)) |   \ge
		\kk_1 \rho^{\ero(I)}
	\end{equation}
	and, if $F_0(z)\in P_{ij}^{u,\kk}$,
	\begin{equation}
		\label{ineq:blockexpansion} | \pi_3 \circ F_0^{-1} (F(z)) |  \le  \kk_2
		\rho^{2\ero(I)}.
	\end{equation}
	Moreover, $F_0^{-1} (\partial^s P_{ij}^{u,\kk})$ is homotopically
	equivalent to $\partial^s P_{ij}^{u,\kk}$.
\end{lemma}

From Lemmas~\ref{lem:stableblocksexpansion} and~\ref{lem:unstableblocksexpansion} and the isolating block argument in \cite{KaloshinZ15}, we obtain an
invariant set in the intersections $\Pi^{u,\kappa}_{ij} \cap
\Pi^{s,\kappa}_{lk}$, $i,j,l,k = 0,1$. Now we check that the cone
conditions are also satisfied.

Our invariant set is a subset of $\{g = 0\} \subset \RR\times
\TT\times \RR^2$. On its tangent space we consider the vector fields
\begin{equation}
	\label{def:centralstableandunstabledirections}
	\begin{aligned}
		e_1^c &= (1,0,0,\dot q^i)^{\top}, \\
		e_2^c &= (0,1,0,0)^{\top}, \\
		e^u &= v^{ij}_1, \\
		e^s &= v^{ij}_2,
	\end{aligned}
\end{equation}
which, for each $z\in \{g = 0\}$, span $T_z \{g = 0\}$.
Given $v\in T_z \{g = 0\}$, we shall identify $v$ with its coordinates, $(v^1,
v^2, v^3, v^4)$.

For $z \in
\Pi^{s,\kappa}_{ij}$, we consider the the unstable cones
\[
C^u_{ij}(z) = \{ v\in T_z\{g = 0\}: |v^1| < |v^3|, \; |v^2| <  |v^3|,\;
|v^4| <  |v^3|\}, \qquad i,j = 0,1,
\]
and, for $z \in \Pi^{u,\kappa}_{ij}$, we consider the the stable
cones
\[
C^s_{ij}(z) = \{ v\in T_z\{g = 0\}: |v^1| < |v^4|, \; |v^2| < |v^4|,\;
|v^3| <  |v^4|\},  \qquad i,j = 0,1.
\]
We will use the standard flat norm, that is,  $\|v\|^2 =
|v^1|^2+|v^2|^2+|v^3|^2+|v^4|^2$.

%

\begin{lemma}
	\label{lem:conecondition}
	There exists $C>0$ such that for all $z \in \Pi^{s,\kappa}_{ij}$ and
	$v \in C^u_{ij}(z)$,
	\begin{equation}
		\label{unstable_cones} D\FF^{ij}_0(z) v \in C^u_{ij}(\FF^{ij}_0(z)) \quad
		\text{and} \quad \|D\FF^{ij}_0(z) v \| \ge \frac{C}{\rho^{\ero(I)}} \|v\|.
	\end{equation}
	Assume $\inf |\dot\la|>0$ (see also \eqref{def:lambdaIenTheta}). Then, for all $z^+ \in  \Pi^{u,\kappa}_{ij}$ and $v \in
	C^s_{ij}(z^+)$,
	\begin{equation}
		\label{stable_cones} D(\FF^{ij}_0)^{-1}(z^+) v \in C^s_{ij}(z) \quad
		\text{and} \quad \|D(\FF^{ij}_0)^{-1}(z^+) v \| \ge \frac{C}{\rho^{\ero(I)}} \|v\|.
	\end{equation}
	If $\inf |\dot\la|= 0$, then the stable cones are invariant and
	\[
	\|D(\FF^{ij}_0)^{-1}(z^+) v \| \ge \frac{C  }{\rho^{\ero(I)}|\log \rho|} \|v\|.
	\]
\end{lemma}

\begin{proof}
	Let $v \in C^u_{ij}(z)$. Let $z^{ij}$ be the center of the block.
	In view of Lemma~\ref{lemma:DifferentialSMCircular},
	\[
	D\FF^{ij}_0 (z) -  D\FF^{ij}_0 (z^{ij}) =
	\begin{pmatrix}
		0 & 0 & 0 & 0 \\
		\OO(\rho^{\ero(I)} \log \rho) & 0 & \OO(\rho^{\ero(I)} \log \rho) & \OO(\rho^{\ero(I)} \log \rho) \\
		\OO(\rho^{\ero(I)}\log \rho) & 0 & \OO(\rho^{\ero(I)}\log \rho) & \OO(\rho^{\ero(I)}\log \rho) \\
		\OO(\log \rho) & 0 &  \OO(\log \rho) & \OO(\log \rho)
	\end{pmatrix}.
	\]
	In the basis $e_1^c$, $e_2^c$, $e^u$, $e^s$, $D\FF^{ij}_0 (z) -  D\FF^{ij}_0 (z^{ij})$ becomes
	\begin{equation}
		\label{def:DFF0errormatrix}
		\begin{pmatrix}
			0 & 0 & 0 & 0 \\
			\OO(\log^2 \rho) & 0 & \OO(\log^2 \rho) & \OO(\log^2 \rho) \\
			\OO(\log \rho) & 0 & \OO(\log \rho) & \OO(\log \rho) \\
			\OO(\log \rho) & 0 & \OO(\log \rho) & \OO(\log \rho)
		\end{pmatrix}.
	\end{equation}
	Then,
	\begin{equation}
		\label{eq:cones}
		\begin{aligned}
			D\FF^{ij}_0 (z) v  = & D\FF^{ij}_0 (z^{ij}) v  +(D\FF^{ij}_0 (z)   - D\FF^{ij}_0 (z^{ij})) v \\
			= & v^1  D\FF^{ij}_0 (z^{ij}) e_1^c + v^2 D\FF^{ij}_0 (z^{ij}) e_2^c + v^3 D\FF^{ij}_0 (z^{ij}) e^u + v^4 D\FF^{ij}_0 (z^{ij}) e^s \\
			& +(D\FF^{ij}_0 (z)   - D\FF^{ij}_0 (z^{ij})) v\\
			= & v^1 e_1^c + (v^2+ v^1 \OO(\log \rho)) e_2^c + (\lambda_1^{ij} v^3+v^1 \OO(\log \rho)) e^u +
			(\lambda_2^{ij} v^4+v^1 \OO(\log \rho)) e^s \\
			& + \OO(\log^2 \rho) (|v^2|+|v^3|+|v^4|) e_2^c  + \OO(\log \rho) (|v^2|+|v^3|+|v^4|) e^u\\
			& +\OO(\log \rho) (|v^2|+|v^3|+|v^4|) e^s.
		\end{aligned}
	\end{equation}
	Then, it is clear that, for $\rho^\ero(I)$ small enough,
	\begin{equation}
		\label{expansionunstablecone}
		\left|(D\FF^{ij}_0 (z) v)^3 \right|  \ge
		\frac{c_{01}^i}{\rho^{\ero(I)}}(|v^3|-\OO(\rho^{\ero(I)} \log \rho)(|v^1|+|v^2|+|v|^3+|v^4|)) \ge
		\frac{c_{01}^i}{2\rho^{\ero(I)}}|v^3|,
	\end{equation}
	which
	implies
	\[
	\begin{aligned}
		|(D\FF^{ij}_0 (z) v)^1| & =  |v^1| \le |v^3| \le |(D\FF^{ij}_0 (z) v)^3|, \\
		|(D\FF^{ij}_0 (z) v)^2| & \le \OO(\log^2 \rho) |v^3| \le |(D\FF^{ij}_0 (z) v)^3|, \\
		|(D\FF^{ij}_0 (z) v)^4| & \le  \OO(\log \rho) |v^3| \le  |(D\FF^{ij}_0
		(z) v)^3|.
	\end{aligned}
	\]
	We have proved that the unstable cone field is invariant.
	As for the expansion, it follows easily from~\eqref{expansionunstablecone},
	since
	\[
	\|D\FF^{ij}_0(z) v \|^2 \ge |(D\FF^{ij}_0(z) v)^3 |^2 \ge \frac{(c_{01}^i)^2}{2\rho^{2\ero(I)}}|v^3|^2
	\\
	\ge \frac{(c_{01}^i)^2}{8\rho^{2\ero(I)}} \|v\|^2.
	\]
	Now we study the field of stable cones.
	A simple computation yields
	\begin{equation}
		\label{eq:inverseDF0} (D\FF^{ij}_0)^{-1} = 
		\begin{pmatrix}
			1 & 0 & 0 & 0 \\
			\ol g^2 \OO_0  & 1 & e^{\Theta^i \ol g}
			\ol g \OO_0
			&
			e^{-\Theta^i \ol g} \ol g \OO_0 \\
			\ol g^2 \OO_0  & 0 & e^{\Theta^i \ol g} \OO_0 & - e^{-\Theta^i \ol g}
			\ol g \OO_0 \\
			\ol g \OO_0  & 0 & - e^{\Theta^i \ol g} \OO_0 &  e^{-\Theta \ol g}
			\ol g \OO_0
		\end{pmatrix}.
	\end{equation}
	Let $v \in
	C^s_{ij}(z^+)$. Let $z^{ij}_+$ the center of the block.
	In view of~\eqref{eq:inverseDF0},
	\[
	D(\FF^{ij}_0)^{-1}(z^+)- D(\FF^{ij}_0)^{-1}(z^{ij}_+) =
	\begin{pmatrix}
		0 & 0 & 0 & 0 \\
		\OO(\rho^{\ero(I)}\log^2 \rho) & 0 & \OO( \log^2 \rho) & \OO(\rho^{\ero(I)} \log \rho^2) \\
		\OO(\rho^{\ero(I)}\log^2 \rho) & 0 & \OO( \log \rho) & \OO(\rho^{\ero(I)}\log^2 \rho) \\
		\OO(\rho^{\ero(I)}\log \rho) & 0 &  \OO(\log \rho) & \OO(\rho^{\ero(I)}\log^2 \rho)
	\end{pmatrix}.
	\]
	In the basis $e_1^c$, $e_2^c$, $e^u$, $e^s$, $D(\FF^{ij}_0)^{-1}(z^+)- D(\FF^{ij}_0)^{-1}(z^{ij}_+)$ becomes
	\[
	\begin{pmatrix}
		0 & 0 & 0 & 0 \\
		\OO(\log^2 \rho) & 0 & \OO(\log^2 \rho) & \OO(\log^2 \rho) \\
		\OO(\log \rho) & 0 & \OO(\log \rho) & \OO(\log \rho) \\
		\OO(\log \rho) & 0 & \OO(\log \rho) & \OO(\log \rho)
	\end{pmatrix}.
	\]
	By~\eqref{eq:inverseDF0},
	\[
	\begin{aligned}
		D(\FF^{ij}_0)^{-1}(z^+) v  = & D(\FF^{ij}_0)^{-1} (z_+^{ij}) v  +(D(\FF^{ij}_0)^{-1}(z^+)   - D(\FF^{ij}_0)^{-1} (z_+^{ij})) v \\
		= & v^1  D(\FF^{ij}_0)^{-1} (z_+^{ij}) e_1^c + v^2 D(\FF^{ij}_0)^{-1} (z_+^{ij}) e_2^c + v^3 D(\FF^{ij}_0)^{-1} (z_+^{ij}) e^u \\
		& + v^4 D(\FF^{ij}_0)^{-1} (z_+^{ij}) e^s  +(D(\FF^{ij}_0)^{-1}(z^+)   - D(\FF^{ij}_0)^{-1} (z_+^{ij})) v\\
		= & v^1 e_1^c + (v^2+ v^1 \OO(\log^2 \rho)) e_2^c + ((\lambda_1^{ij})^{-1} v^3+v^1 \OO(\log^2 \rho)) e^u \\
		& +
		((\lambda_2^{ij})^{-1} v^4+v^1 \OO(\log^2 \rho)) e^s
		+ \OO(\log^2 \rho) (|v^2|+|v^3|+|v^4|) e_2^c  \\
		& + \OO(\log^2 \rho) (|v^2|+|v^3|+|v^4|) e^u
		+\OO(\log^2 \rho) (|v^2|+|v^3|+|v^4|) e^s.
	\end{aligned}
	\]
	Then, it is clear that, for $\rho^{\ero(I)}$ small enough,
	\begin{equation}
		\label{expansionunstablecone_u}
		\left|(D(\FF^{ij}_0)^{-1}(z^+) v)^4 \right|  \ge
		\frac{c_{01}^i}{\rho^{\ero(I)}}(|v^4|-\OO(\rho^{\ero(I)} \log^2 \rho)(|v^1|+|v^2|+|v|^3+|v^4|)) \ge
		\frac{c_{01}^i}{2\rho^{\ero(I)}}|v^4|,
	\end{equation}
	which
	implies
	\[
	\begin{aligned}
		|(D(\FF^{ij}_0)^{-1}(z^+) v)^1| & =  |v^1| \le |v^4| \le |(D(\FF^{ij}_0)^{-1}(z^+) v)^4|, \\
		|(D(\FF^{ij}_0)^{-1}(z^+) v)^2| & \le \OO(\log^2 \rho) |v^4| \le |(D(\FF^{ij}_0)^{-1}(z^+) v)^4|, \\
		|(D(\FF^{ij}_0)^{-1}(z^+) v)^3| & \le  \OO(\log^2 \rho) |v^4| \le  |(D(\FF^{ij}_0)^{-1}(z^+) v)^4|.
	\end{aligned}
	\]
	We have proved that the stable cone field is invariant.
	As for the expansion, it follows easily from~\eqref{expansionunstablecone_u}
	since
	\[
	\|D(\FF^{ij}_0)^{-1}(z^+) v \|^2 \ge |(D(\FF^{ij}_0)^{-1}(z^+) v)^4 |^2 \ge \frac{(c_{01}^i)^2}{2\rho^{2\ero(I)}}|v^4|^2
	\\
	\ge \frac{(c_{01}^i)^2}{8\rho^{2\ero(I)}} \|v\|^2.
	\]
\end{proof}

Lemmas~\ref{lem:stableblocksexpansion}, \ref{lem:unstableblocksexpansion} and~\ref{lem:conecondition} imply the existence
of two collections of Lipschitz graphs, $W_{ij}^{uc}$ and $W_{ij}^{sc}$. The normally hyperbolic lamination $\LL_0$
is then $W_{ij}^c = W_{ij}^{uc} \pitchfork W_{ij}^{sc}$. Now, for each $x \in W_{ij}^c$, the sequence $\{(\FF^{ij}_0)^k(x)\}_{k\in \ZZ}$ defines a unique sequence $\wt \omega = (\wt\omega_k)$, $k\in \ZZ$, $\wt\omega_k \in \{00,01,10,11\}$, by the relation
$(\FF^{ij}_0)^k)(x) \in \Pi_{\wt\omega_k}^{s,\kk}$. Notice that not all sequences are possible: if $\wt\omega_k = (\wt\omega_k^1,\wt\omega_k^2)$, then $\wt\omega_{k+1}^1 = \wt\omega_k^2$. Let $\wt \Sigma \subset \Sigma_{\{00,01,10,11\}}$ be the subset of possible sequences in the space of sequences of four symbols. Let $\Sigma_{\{0,1\}}$ the space os sequences of two symbols. We endow $\Sigma_{\{00,01,10,11\}}$ and $\Sigma_{\{0,1\}}$ with the product topology. It is clear that
the map $h: \wt \Sigma \to \Sigma_{\{0,1\}}$ defined by
\[
h( \dots,\omega_{-1}^1,\omega_{-1}^2, \omega_0^1,\omega_0^2, \omega_1^1,\omega_1^2,\dots ) = (\dots, 
\omega_{-1}^1, \omega_0^1, \omega_1^1,\dots )
\]
is a homeomorphism with the given topology. That is, we can use sequences of two symbols to label each leave of the lamination. By construction, the action of $\FF^{ij}_0$ on the lamination commutes with the action of the Bernoulli shift
on $\Sigma_{\{0,1\}}$.

We claim that each leave of the lamination, $L_0(\omega,\cdot): \II \times \TT \to B_{\rho}^1 \cup B_{\rho}^2$ is in
fact $C^r$, for any $r$, if $\rho$ is small enough. To prove the claim, we consider the  splitting $E^c \oplus E^u \oplus E^s$ of $T_{B_{\rho}^1 \cup B_{\rho}^2} (\II \times \TT \times \RR^2)$, where $E^c_x = \langle e_1^c(x), e_2^c(x)\rangle$,
$E^u_x = \langle e^u(x)\rangle$ and $E^s_x = \langle e^s(x)\rangle$, where $\langle v_1, \dots, v_n\rangle$ denotes the subspace spanned by $v_1$, \dots, $v_n$ and the vector fields $e_1^c$, $e_2^c$, $e^u$, $e^s$ were introduced in~\eqref{def:centralstableandunstabledirections}, with the Riemannian metric associated to the standard norm.
Let $\pi^*$ denote the projection onto $E^*$. Proceeding in the
same way as in~\eqref{eq:cones} and taking into account~\eqref{def:DFF0errormatrix}, there exists a constant $C>0$ such that for all $v^{sc} \in E^{sc}$, $v^u \in E^u$, $v^{uc} \in E^{uc}$ and $v^s \in E^s$
\[
\begin{aligned}
	\| \pi^{sc} D\FF^{ij}_0(x) v^{sc} \| & \le C |\log \rho |^2 \|v^{sc}\|, & \|\pi^u D\FF^{ij}_0(x) v^u\| & \ge \frac{C}{\rho^{\ero(I)}} \|v^u\|, \\
	\| \pi^{u} D\FF^{ij}_0(x) v^{sc} \| & \le C |\log \rho |^2\|v^{sc}\|, & \|\pi^{sc} D\FF^{ij}_0(x) v^u\| & \le C |\log \rho|^2 \|v^u\|,
	\\
	\| \pi^{uc} (D\FF^{ij}_0)^{-1}(x) v^{uc} \| & \le C |\log \rho |^2\|v^{sc}\|, & \|\pi^s (D\FF^{ij}_0)^{-1} v^s\| & \ge \frac{C}{\rho^{\ero(I)}} \|v^s\|, \\
	\| \pi^{s} D\FF^{ij}_0(x) v^{uc} \| & \le C |\log \rho |^2 \|v^{uc}\|, & \|\pi^{uc} D\FF^{ij}_0(x) v^s\| & \le C |\log \rho|^2 \|v^s\|.
\end{aligned}
\]
Then the standard theory of normally hyperbolic laminations ensures that the leaves are $C^r$ if, for all $0 \le k \le r$,
\[
(C |\log \rho|^2)^k \frac{\rho^{\ero(I)}}{C^2 |\log \rho|^2} < 1.
\]
Hence, for any fixed $r$, taking $\rho$ small enough, follows the claim.

We only need to check the H\"older dependence of the leaves with respect to $\omega$. 
To do so, for a fixed $a>1$, we consider in the space of sequences $\Sigma_{\{0,1\}}$ the distance
\[
d_a(\omega,\omega') = \sum_{k\in \ZZ} \frac{|\omega_k-\omega_k'|}{a^{|k+1|}}.
\]
Then, applying directly the arguments in~\cite{KaloshinZ15}, we have that
\[
\|L_0(\omega,\cdot)-L_0(\omega',\cdot)\|_{C^r} \le C_r d_a(\omega,\omega')^{\prod_{j=1}^r \varphi_j},
\]
with
\[
\varphi_j = \frac{\log (C/\rho^{\ero(I)}) - \log (C|\log \rho|^2)}{\log b_j - \log (C|\log \rho|^2)}
\]
and $b_j = C_j \|\FF^{ij}_0\|_{C^{j+3}}^{j+3}$. Since $\| \FF^{ij}_0\|_{C^{j+3}}^{j+3} = \OO(\rho^{-(j+3)\ero(I)})$,
$\varphi_j \approx 1/(j+3)$.

\subsection{The NHIL for the elliptic problem: Proof of Theorem \ref{thm:NHILElliptic}}\label{sec:NHIL:Elliptic}

Normal Hyperbolicity implies that the NHIL given by Theorem~\ref{thm:NHILCircular} for the circular problem persists for $e_0>0$ small enough and depends smoothly
on~$e_0$ \cite{HirschPS77}. Moreover, it can be parameterized by an embedding $L_{e_0}$, which is regular with respect to $(e_0,I,s)$, is H\"older with respect to $\omega$ and is $e_0$-close to the embedding $L_0$ given in Theorem~\ref{thm:NHILCircular}. The embedding $L_{e_0}$ satisfies the invariance equation~\eqref{eq:inveq_lamination_elliptic}. It only remains to prove the claim on
the functions  $J_{e_0}$, $S_{e_0}$,   $P_{e_0}$ and $Q_{e_0}$. To do so, we write
\[
\begin{aligned}
	\SM_{e_0} & = \SM_{0}+ e_0 SM_1 + e_0^2 SM_2+\OO(e_0^3), \\
	L_{e_0} & = L_0 + e_0 \tilde L_1 + e_0^2 \tilde L_2 + \OO(e_0^3), \\
	\FF_{\iinn} & = \FF_0 + e_0 F_1 + e_0^2 F_2 + \OO(e_0^3),
\end{aligned}
\]
where $\SM_{e_0}$ and $\SM_{0}$ denote the separatrix map for the elliptic and circular problem, respectively, $L_{e_0}$
and $L_0$ are their corresponding laminations
and $\FF_{\iinn}$ and $\FF_0$ are their inner maps. The invariance equation~\eqref{eq:inveq_lamination_elliptic} implies that the functions $\tilde L_1$, $\tilde L_2$, $F_1$ and $F_2$ have to satisfy
\begin{align}
	\label{eq:L1F1}
	D \SM_0 \circ L_0 \, \tilde L_1 - \tilde L_1 \circ \FF_0  = & DL_0 \circ \FF_0 \, F_1 - SM_1\circ L_0 \\
	\nonumber
	D \SM_0 \circ L_0 \, \tilde L_2 - \tilde L_2 \circ \FF_0  = &
	DL_0 \circ \FF_0 \, F_2 + \frac{1}{2} D^2 L_0 \circ \FF_0 \, F_1^2 + D \tilde L_1 \circ \FF_0 \, F_1 \\
	\label{eq:L2F2}
	& - \frac{1}{2} D^2 \SM_0 \circ L_0 \, \tilde L_1^2 - DSM_1 \circ L_0 \, \tilde L_1 - SM_2 \circ L_0.
\end{align}
Since $L_{e_0}$ is a graph over $(I,s)$, $\pi_{I,s} \circ \tilde L_i = 0$, $i=1,2$, where $\pi_{I,s}$ denotes the projection onto the $(I,s)$ variables and $\pi_{I,s} \circ DL_0 \circ \FF_0 = \Id$. Moreover, since $I$ is invariant for $\SM_0$, we have that $\pi_I D \SM_0 \circ L_0 \, \tilde L_1 = 0$. Hence, by the expression of $SM_1$ given by Theorem~\ref{thm:FormulasSMii}, the projection onto the $I$ variable of equation~\eqref{eq:L1F1} gives
\[
J_{1} = \pi_I F_1 = \MM^{I,1}_{i}+\MM^{I,1}_{i,pq},
\]
which, since $\NNN(\MM^{I,1}_{i}),\NNN(\MM^{I,1}_{i,pq}) = \{\pm 1\}$, proves the claim for $J_1$.

By Theorem~\ref{thm:NHILCircular}, the projection onto the $(p,q)$ variables of $L_0$ does not depend on $s$ and $\FF_0(\omega,I,s) = (I,s+\beta(\omega,I))$. Hence, the projection onto the $(p,q)$ variables of~\eqref{eq:L1F1} reads
\[
\begin{pmatrix}
	\partial_p \FF_{p,0} & \partial_q \FF_{p,0 }\\
	\partial_p \FF_{q,0} & \partial_q \FF_{q,0}
\end{pmatrix}_{\mid L_0(I,s)}
\begin{pmatrix}
	P_1 \\ Q_1
\end{pmatrix}_{\mid(I,s)} -
\begin{pmatrix}
	P_1 \\ Q_1
\end{pmatrix}_{\mid(I,s+\beta(\omega,I))}=
\begin{pmatrix}
	\partial_I P_0 J_1 \\ \partial_I Q_0 J_1
\end{pmatrix}_{\mid(I,s)}
-\begin{pmatrix}
	F_{p,1} \\ F_{q,2}
\end{pmatrix}_{\mid(I,s)}.
\]
The right hand side of the above equation has only harmonics $\pm 1$. Since $\pi_{p,q} \circ D \SM_0 \circ L_0 \circ \pi_{p,q}$ is hyperbolic and does not depend on $s$, it follows that $\NNN(P_1) = \NNN(Q_1) = \{\pm 1\}$.

Since the projection onto the $s$-variable of~\eqref{eq:L1F1} is
\[
\partial_p \FF_{s,0} \circ L_0 \, P_1 + \partial_q \FF_{s,0} \circ L_0 \, Q_1 = S_1- F_{s,1} \circ L_0,
\]
the function $\FF_{s,0}$ does not depend on $s$ and $\NNN(F_{s,1}) = \{\pm 1\}$, the same claim holds for $S_1$.

Finally, regarding the $e_0^2$-term of the  $I$ component, it is enough to take into account~\eqref{eq:L2F2}. We have that $\pi_I DL_0 \circ \FF_0 \, F_2 = J_2$ and, then, using that $\pi_I \tilde L_2 = \pi_I D \SM_0 \circ L_0 \, \tilde L_2 = 0$,
\[
J_2 = -\pi_I \left( \frac{1}{2} D^2 L_0 \circ \FF_0 \, F_1^2 + D L_1 \circ \FF_0 \, F_1
- \frac{1}{2} D^2 \SM_0 \circ L_0 \, \tilde L_1^2 - DSM_1 \circ L_0 \, \tilde L_1 - SM_2 \circ L_0 \right).
\]
The claim follows from the fact that neither $D^2 L_0$ nor $D^2 \SM_0$ depend on $s$ and
$\NNN(F_1) = \NNN(\tilde L_1) =\NNN(DSM_1) = \NNN(SM_2) = \{\pm 1\}$.

\section*{Declarations}
\addcontentsline{toc}{section}{\protect\numberline{}Declarations}

\begin{itemize}
\item \textbf{Funding}: 
This work was partially supported by the grant PID-2021-122954NB-100 and PID2021-123968NB-100  funded by MCIN/AEI/10.13039/501100011033 and “ERDF A way of making Europe.

M.Guardia is also supported by the Catalan Institution for Research and
Advanced Studies via  ICREA Academia Prizes 2019 and 2023.

V. Kaloshin is also partially supported by the ERC Advanced Grant SPERIG (\# 885707). 

P. Mart\'in is also partially supported by PID2024-158570NB-I00.

This work was also supported by the Spanish State Research Agency through the Severo Ochoa and Mar\'ia de Maeztu Program for Centers and Units of Excellence in R\&D(CEX2020-001084-M).

\item \textbf{Competing interests}: The authors have no competing interests to declare that are relevant to the content of this article. All authors certify that they have no affiliations with or involvement in any organization or entity with
any financial interest or non-financial interest in the subject matter or materials discussed in this manuscript.
The authors have no financial or proprietary interests in any material discussed in this article. 
\item \textbf{Data availability} : The authors declare that all the data supporting the findings of this study are available within the paper.
\end{itemize}

\appendix

\section{Numerical verification of Ansatz \ref{ans:NHIMCircular:bis}}\label{sec:homocl_channels}

\subsection{Equations of motion and Poincar\'e map} \label{sec:ComputationPoincareMap}

Consider the restricted planar circular three body problem (RPC3BP) in rotating
Cartesian coordinates
\begin{equation}\label{eq:RTBP:RotCartesian}
  \tJ(x,y,p_x,p_y)=\frac{1}{2}(p_x^2+p_y^2)+yp_x-xp_y-\frac{\mu}{r_1}-\frac{1-\mu}{r_2},
\end{equation}
where
\begin{align*}
r_1^2 &= (x-(1-\mu))^2 + y^2, \\
r_2^2 &= (x+\mu)^2 + y^2.
\end{align*}
We follow the convention to place the large mass (Sun) to the left of the origin, and
the small mass (Jupiter) to the right.
The system has a time-reversal symmetry across the $y=0$ plane: 
\begin{equation}\label{eq:symmetry}
	R(x(t), y(t))= (x(-t), -y(-t)).
	\end{equation} 
Recall that the energy $\tJ$ is a first integral. From now on, let the mass parameter be fixed
to $\mu=0.95387536\times10^{-3}$.

In the paper~\cite{FejozGKR15} (Section C.3) we 
established numerically that, in every energy level $\tJ\in[\tJ_-,\tJ_+]=[-1.731,-1.359]$, there
exists a hyperbolic periodic orbit $\lambda_\tJ$ that is close to $3:1$-resonant and satisfies \eqref{def:periodestimate}.
The periodic orbit and its period depend smoothly on $\tJ$. 
This supports the  statements in Item 1 of Ansatz \ref{ans:NHIMCircular:bis}.

Figure~\ref{fig:trtbp_porbits_3to1} shows a representative sample of periodic orbits in this family.
\begin{figure}
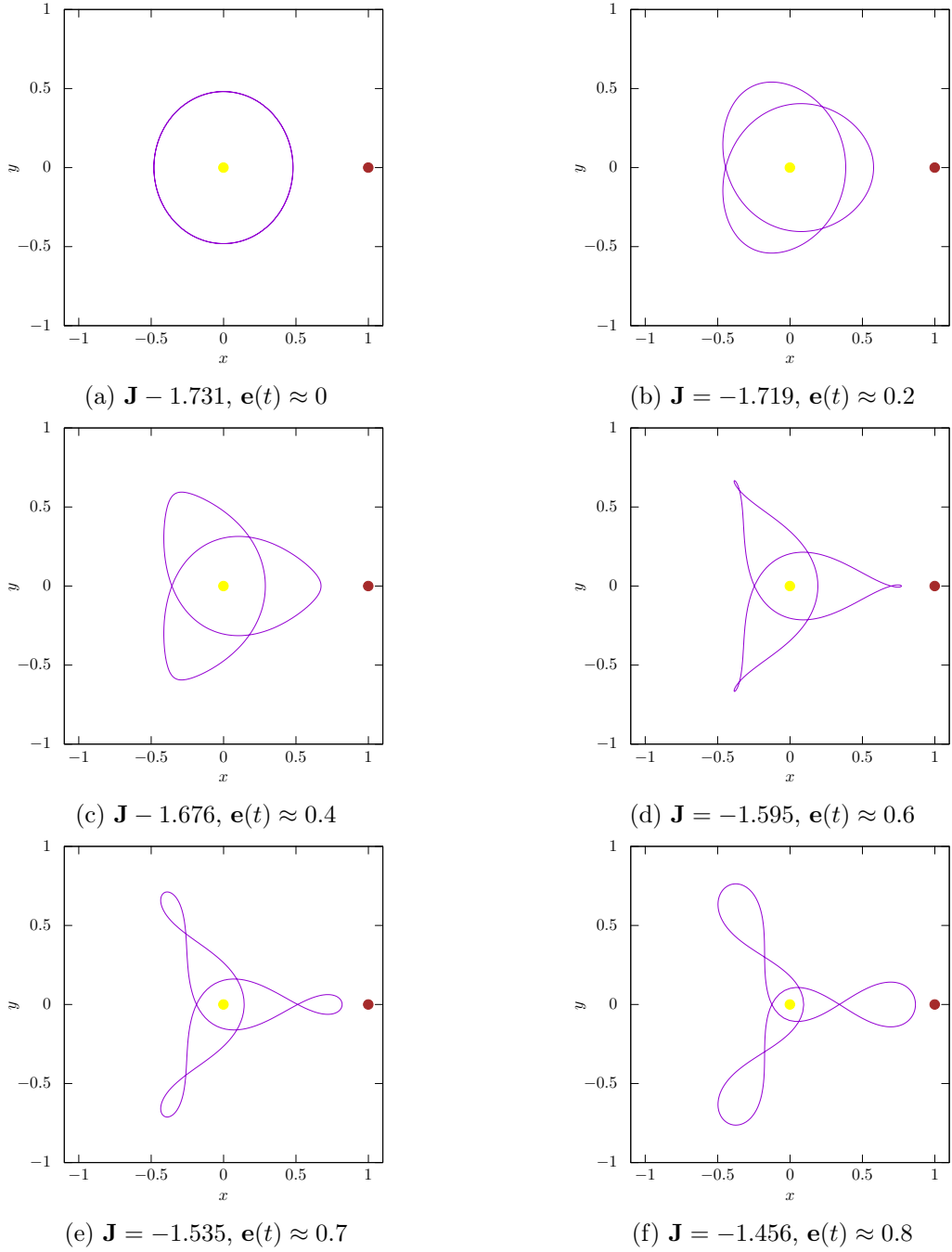

	\centering
	\begin{subfigure}{0.45\textwidth}
	\centering
	    \resizebox{\textwidth}{!}{\input{figs/trtbp_3to1_ell00}}
		\caption{$\tJ-1.731$, $\ecc(t)\approx 0$}
	\end{subfigure}
	~
	\begin{subfigure}{0.45\textwidth}
	    \resizebox{\textwidth}{!}{\input{figs/trtbp_3to1_ell02}}
		\caption{$\tJ=-1.719$, $\ecc(t)\approx 0.2$}
	\end{subfigure}
	~
\begin{subfigure}{0.45\textwidth}
	    \resizebox{\textwidth}{!}{\input{figs/trtbp_3to1_ell04}}
	\caption{$\tJ-1.676$, $\ecc(t)\approx 0.4$}
\end{subfigure}
	~
\begin{subfigure}{0.45\textwidth}
	    \resizebox{\textwidth}{!}{\input{figs/trtbp_3to1_ell06}}
	\caption{$\tJ=-1.595$, $\ecc(t)\approx 0.6$}
\end{subfigure}
	~
\begin{subfigure}{0.45\textwidth}
	    \resizebox{\textwidth}{!}{\input{figs/trtbp_3to1_ell07}}
	\caption{$\tJ=-1.535$, $\ecc(t)\approx 0.7$}
\end{subfigure}
	~
\begin{subfigure}{0.45\textwidth}
	    \resizebox{\textwidth}{!}{\input{figs/trtbp_3to1_ell08}}
	\caption{$\tJ=-1.456$, $\ecc(t)\approx 0.8$}
\end{subfigure}
	\caption{Some (near-)resonant periodic orbits $\lambda_\tJ$ for different energy values $\tJ$, or  instantaneous eccentricity values $\ecc(t)$ ($xy$ projection).  The location of the Sun and Jupiter are marked with a yellow and brown dot, respectively.}
	\label{fig:trtbp_porbits_3to1}
\end{figure}
Consider the surface of section 
\[\Sec \coloneqq \{y=0\}\] 
and its associated Poincar\'e map 
\[\Poinc:\Sec\to\Sec.\] 
Notice that low-energy periodic orbits such as $\tJ=-1.719$ intersect the section $\Sec$ at 4 points. However, if we raise the energy above a certain threshold, periodic orbits such as $\tJ=-1.535$ have three ``loops'' at the apohelion, so they intersect $\Sec$ at 6 points. 

\begin{remark}~\label{rem:loops}
The appearance of these loops in RPC3BP coordinates further complicates the study of their associated stable/unstable invariant manifolds using a surface of section, due to tangencies of the manifolds with $\Sec$.
This fact is illustrated in Figure~\ref{fig:tangent_trajectory}, which shows
a resonant periodic orbit and a trajectory in its associated unstable
invariant manifold. As the trajectory is integrated
forward along the invariant manifold, it precedes clockwise, so that one loop ceases to intersect the surface of section $\Sec$, and later another loop starts to intersect it.
\end{remark}

\begin{figure}
	\begin{center}
	    \resizebox{!}{0.7\height}{\input{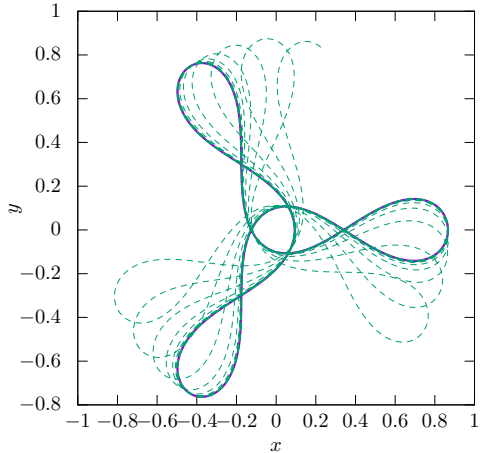}}
	\end{center}
	\caption{Solid line: resonant periodic orbit $\lambda_\tJ$ with energy $\tJ=-1.456$. Dashed line: trajectory in its associated unstable invariant manifold.
		As the unstable trajectory precedes clockwise, the horizontal loop of the trajectory ceases to intersect with the section $\{y=0\}$. The loop at an angle $-2\pi/3$ will eventually start intersecting the section.}
	\label{fig:tangent_trajectory}
\end{figure}

\emph{Numerically}, it is more convenient to integrate the flow of
the RPC3BP in Cartesian than in Delaunay variables. 
The reason is that, to evaluate the vector field in Delaunay, one needs to
find the eccentric anomaly $u$ by solving Kepler's equation
$u - \ecc\sin u =\ell$, 
where $\ecc$ is the eccentricity and $\ell$ is the mean anomaly.
However, Kepler's equation cannot be solved algebraically for $u$.
Of course, one can still solve Kepler's equation numerically,
but for extreme eccentricities, Newton's method may take
long to converge and can lead to numerical errors.
Evaluating the variational equations in Delaunay is even more problematic.
Thus, from now on we will always integrate the flow of the RPC3BP in Cartesian coordinates~\eqref{eq:RTBP:RotCartesian}. Eventually in Appendix~\ref{sec:nondegeneracy}, we will need to map some phase space points to Delaunay, but this poses no difficulty numerically.

\begin{remark}
    For all numerical integrations, we use a variable-order Taylor method with
    local error tolerance $10^{-16}$. 
\end{remark}

Next we explain how the resonance structure and homoclinic points are
computed in more detail.

\subsection{Periodic points and hyperbolic splittings}


We
compute the resonant periodic orbits $\lambda_\tJ$  as periodic points of the Poincar\'e map $\Poinc$ using a Newton-like
method. A first guess for the periodic point is obtained from the two-body
problem equations.
Due to the time-reversal symmetry~\eqref{eq:symmetry}, 
in this case it is enough to use a 1-dimensional Newton method. This is how Figure~\ref{fig:trtbp_porbits_3to1} was produced.

\begin{remark}
The computation of periodic points has an accuracy (absolute error) of
$10^{-14}$.
\end{remark}

In~\cite{FejozGKR15} we check that the family of periodic orbits $\lambda_\tJ$
is indeed hyperbolic. Each periodic orbit corresponds to a fixed point $p\in
\RR^2$ for the \emph{iterated} Poincar\'e map $\fourmap\coloneqq\Poinc^4$, whose eigenvalues 
are not in the unit circle.
We also compute the hyperbolic (stable and unstable) directions $E^s(p)$ and
$E^u(p)$ associated to the fixed point.

\subsection{Invariant manifolds and homoclinic points}

Next step is to  extend the 1-dimensional local invariant manifolds in the following
way. 
Let $w^u(\xi)$ be a parametrization of the local unstable manifold $W^u$
(approximated by the linearization at the fixed point).
We take a fundamental domain $\xi\in[\xi_-, \xi_+]$ on the local
invariant manifold, and iterate it by $\fourmap$ a certain number of times $N$, until it straddles the symmetry axis $\{p_x=0\}$. See Figure~\ref{fig:resonance:SEC}.

\begin{figure}
	\begin{subfigure}{\textwidth}
		\centering		
		\resizebox{!}{0.7\height}{\input{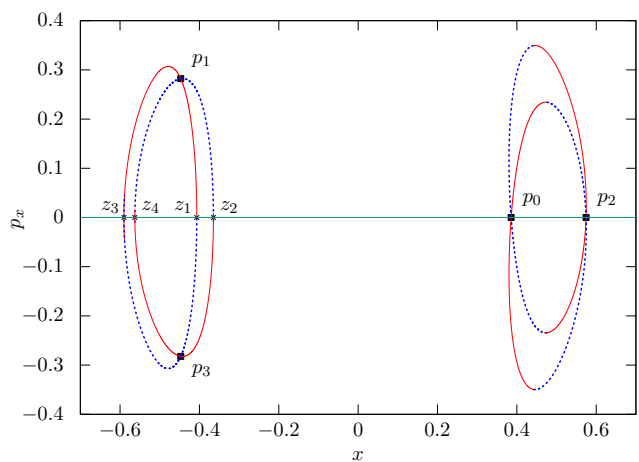}}
		\caption{Energy level $\tJ=-1.719$.}
		\label{fig:resonance:SEC}
	\end{subfigure}
	~
	\begin{subfigure}{\textwidth}
		\centering
		\resizebox{!}{0.7\height}{\input{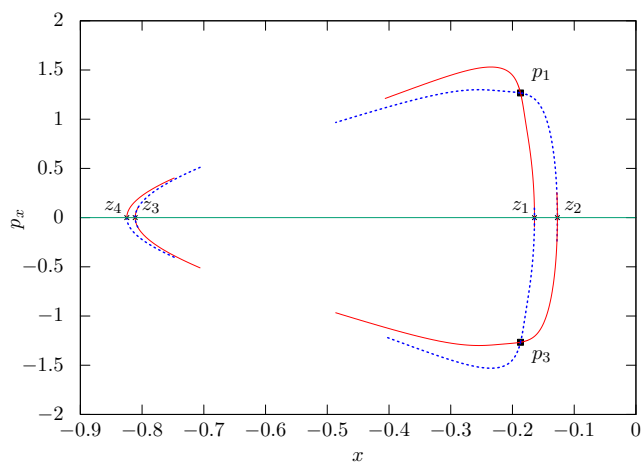}}
		\caption{Energy level $\tJ=-1.535$. Only the asymptotic manifolds of $p_1$ and $p_3$ are shown, since they are the only relevant ones for this paper.}
		\label{fig:resonance2:SEC}
	\end{subfigure}

	\caption{Hyperbolic structure on the section $\{y=0\}$ for two different energy levels. The periodic orbit corresponds to $p_0, p_1, p_2, p_4$, fixed points for $\fourmap$. The stable manifold is colored in blue, and the unstable in
	red. Points on the symmetry axis have both $y=0$ and $p_x=0 \implies \dot x=0$, so they correspond to symmetric trajectories for the flow. There are four symmetric homoclinic points, denoted $z_1, z_2, z_3, z_4$.}
	\label{fig:resonance}
\end{figure}


The intersection of the manifolds with the symmetry axis gives four symmetric
homoclinic points, that we have denoted $z_1, z_2, z_3, z_4\in \Sec$. See Figure~\ref{fig:resonance:SEC}.

\begin{remark}
	In~\cite{FejozGKR15} we already computed $z_1$ and $z_2$ on one branch of the manifolds.
	For this paper we need four homoclinic points, so we have computed the new homoclinic points $z_3$ and $z_4$ on the other branch.
\end{remark}

\begin{remark}
	One could also use the homoclinic points between $p_0$ and $p_2$. However these do not lie on the symmetry axis. Numerically, it is more convenient to use symmetric homoclinic points, since they are easier to compute, and we can verify the accuracy of our computations by checking how well the symmetry is preserved.
\end{remark}

An intersection corresponds to a root $\xi_*\in[\xi_-, \xi_+]$ of
the equation
\[\Pi_{p_x} \fourmap^N(w^u(\xi)) = 0, \] 
where $\Pi_{p_x}$ denotes projection onto the $p_x$ component.
To find this root, we use a 1-dimensional root bracketing algorithm (namely
Brent's method). 


As the energy is increased, the periodic orbit develops the ``loops'', and its asymptotic invariant manifolds become discontinuous, as explained in Remark~\ref{rem:loops}. See Figure~\ref{fig:resonance2:SEC}. Despite the discontinuities, the manifolds still intersect the symmetry axis, which allows us to compute the homoclinic points $z_i$.
Figure~\ref{fig:intersecs} represents the four families of homoclinic points as a function of $\tJ$.

To test the numerical accuracy of the homoclinic points, we repeated the
computation using the stable manifold (instead of the unstable one). Then we
looked at the difference between the homoclinic point obtained using the
unstable manifold and the homoclinic point obtained using the stable manifold.
This numerical error is smaller than $10^{-10}$ for all homoclinic points in
Figure~\ref{fig:intersecs}.
\begin{figure}
	\begin{center}
		\resizebox{!}{0.7\height}{\input{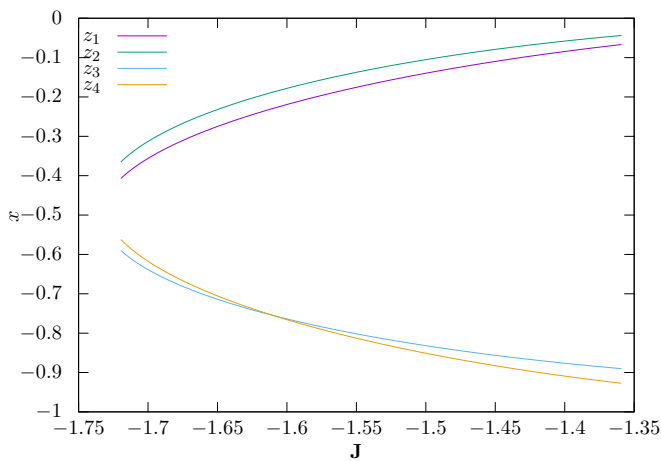}}
	\end{center}
	\caption{The four families of homoclinic points $z_i(\tJ)$ for $i=1,2,3,4$. We plot only their $x$ component, since $y=0$ is
		given by the surface of section, $p_x=0$ is given by the symmetry,
		and $p_y$ can be determined from the energy integral.}
	\label{fig:intersecs}
\end{figure}

\begin{remark}\label{rem:dgdt}

Of course, the appearance of tangencies between the manifolds and the surface of section depends on the system of coordinates being used. Let us briefly discuss the situation in Delaunay coordinates. The surface $\{g=0\}$ in \emph{rotating} Delaunay variables is a candidate surface of section
for our problem, since in the Kepler problem approximation we have $\dot g
=-1$. Thus we can assume that in the three-body problem $\dot g\leq 0$, but only for \textit{moderate values of the perturbation} by Jupiter (that is both $\mu$ small and not too close to collisions with either the Sun or Jupiter). 
Indeed, when the asteroid's trajectory becomes very eccentric ($\tJ \ge -1.485$, or equivalently  $\ecc\ge 0.77$), it passes close to Jupiter and then $dg/dt$ changes sign along the trajectory. See Figure~\ref{fig:dgdt}. Then $\{g=0\}$ is not a global Poincar\'e section anymore.
	In the companion paper \cite{stochastic23}, equation~\eqref{def:Reduced:ODE:circular}, we perform a time reparametrization to have $\frac{d}{ds}g=1$. This reparametrization is not valid if $dg/dt$ changes sign.
	In order to perform this reparametrization, we must restrict ourselves to moderate eccentricity values $\tJ < -1.485$ (or equivalently $\ecc < 0.77)$.
\end{remark}

\begin{figure}
	\begin{center}
		\resizebox{!}{0.7\height}{\input{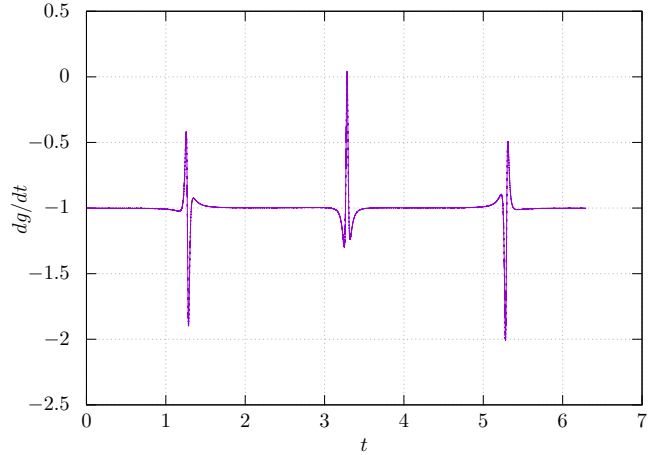}}
	\end{center}
	\caption{Monitoring $dg/dt$ (rotating Delaunay coordinates) along a homoclinic orbit with energy $\tJ=-1.485$. Notice that $dg/dt$ changes sign around $t=\pi$.}
	\label{fig:dgdt}
\end{figure}


\subsection{Numerical verification of transversality condition}\label{sec:transversality}

Let $p\in \Sigma$ be a fixed point for the (iterated) Poincar\'e map, and let
$W^s(p)$ and $W^u(p)$ be the associated stable and unstable $1$-dimensional invariant
manifolds on $\Sigma$.
Consider a homoclinic point $z\in W^s(p)\cap W^u(p)$. Let $v^u=(x^u, p_x^u)$
and $v^s=(x^s, p_x^s)$ be the tangent vectors to the unstable and stable
manifolds at $z$. We define the \emph{splitting angle} between $W^s(p)$ and $W^u(p)$ at
the homoclinic point $z$ as the directed angle from $v_u$ to $v_s$:
\[ \splitting=\arctan(p_x^u/x^u) - \arctan(p_x^s/x^s). \]

We have computed the homoclinic point in Cartesian coordinates. The tangent vectors $v_u,v_s$ are obtained through iteration of the
linear stable/unstable directions near the fixed point by the differential of
the Poincar\'e map.  
%

Recall that we are interested in four different homoclinic points.
Let $\splitting_i$ be the splitting angle at the corresponding homoclinic point
$z_i$.
Figure~\ref{fig:splitting} shows the splitting angles as a function of the
energy value.
We see that the splitting angle decreases in magnitude as the energy $\tJ$
decreases, or equivalently as $\ecc\to 0$. The manifolds become tangent (i.e. splitting angle becomes $0$) for some values of the energy, but generally they intersect transversally.
\begin{figure}
	\begin{center}
		\resizebox{!}{0.7\height}{\input{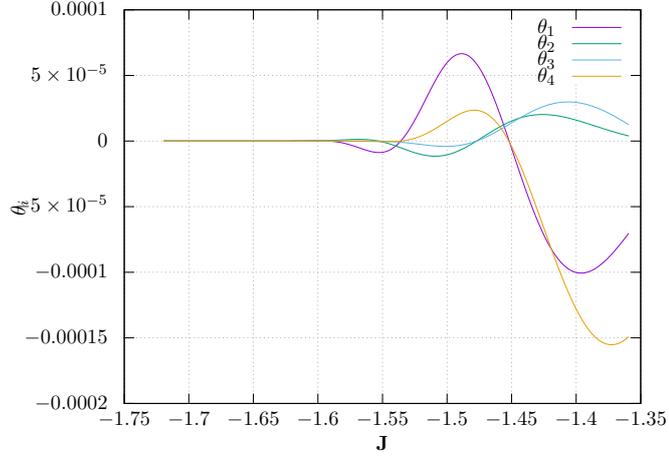}}
	\end{center}
	\caption{Splitting angles $\splitting_i$
		between the manifolds as a function of energy. Crossings with the horizontal
		axis correspond to tangencies of
		the manifolds. Note that two of the splittings ($\splitting_1$ and $\splitting_2$) were already computed in our previous paper~\cite{FejozGKR15}.}
	\label{fig:splitting}
\end{figure}
The particular values of the energy at which the manifolds become tangent are  listed in Table~\ref{tab:tangencies}.
\begin{table}
\begin{center}
\begin{tabular}{ll}
	\toprule
	splitting angle & \multicolumn{1}{c}{energy value of tangencies}\\
	\midrule
	$\splitting_1$ & 
	\textcolor{white}{-1.XXX} -1.535 \textcolor{white}{-1.XXX} -1.451\\
	$\splitting_2$ &
	-1.551 \textcolor{white}{-1.XXX}
	-1.475 \textcolor{white}{-1.XXX}\\
	$\splitting_3$ &
	-1.551 \textcolor{white}{-1.XXX}
	-1.475 \textcolor{white}{-1.XXX}\\
	$\splitting_4$ & 
	\textcolor{white}{-1.XXX} -1.535 \textcolor{white}{-1.XXX} -1.451\\
	\bottomrule
\end{tabular}
\caption{Energy levels at which we encounter homoclinic tangencies.}
\label{tab:tangencies}
\end{center}
\end{table}

To verify Item 2 of Ansatz \ref{ans:NHIMCircular:bis}, we are interested in energy intervals that contain \textit{at least two homoclinic channels without tangencies}. See Table~\ref{tab:intervals} for these intervals. Notice that these are relatively large intervals, where the eccentricity changes by as much as $0.1$, or $10\%$.
Due to Remark~\ref{rem:dgdt}, we restrict our analysis
to energies $\tJ< -1.485$. Thus, the final diffusion intervals will necessarily be smaller than $\mathcal{\tilde I}_1, \mathcal{\tilde I}_2, \mathcal{\tilde I}_3, \mathcal{\tilde I}_4$. In particular, we will work on the smaller diffusion intervals 
\[\mathcal{I}_i := \mathcal{\tilde I}_i \cap \{\tJ < -1.485\}. \]
\begin{table}
\addtolength{\tabcolsep}{3pt}
\begin{center}
\begin{tabular}{lcc}
	\toprule
	interval & 
	$\tJ$-values & $e$-values \\
	\midrule
	$\mathcal{\tilde I}_1$ & $[-1.551, -1.475]$ & $[0.676, 0.779]$ \\
	$\mathcal{\tilde I}_2$ & $[-1.535, -1.451]$ & $[0.700, 0.805]$ \\
	$\mathcal{\tilde I}_3$ & $[-1.475, -1.359]$ & $[0.779, 0.888]$ \\
	$\mathcal{\tilde I}_4$ & $[-1.451, -1.359]$ & $[0.805, 0.888]$ \\
	\bottomrule
\end{tabular}
\end{center}
\addtolength{\tabcolsep}{-3pt}
\caption{Intervals containing at least two homoclinic channels free of tangencies, given both in terms of the energy $\tJ$ and the equivalent approximate eccentricity $\ecc(t)$ of the instantaneous ellipse.
}
\label{tab:intervals}
\end{table}

\subsection{Bounds for $T_\tJ$, $L_\tJ$, and $L_i(t,\tJ)$.}
\begin{figure}
	\begin{center}
	    \resizebox{!}{0.7\height}{\input{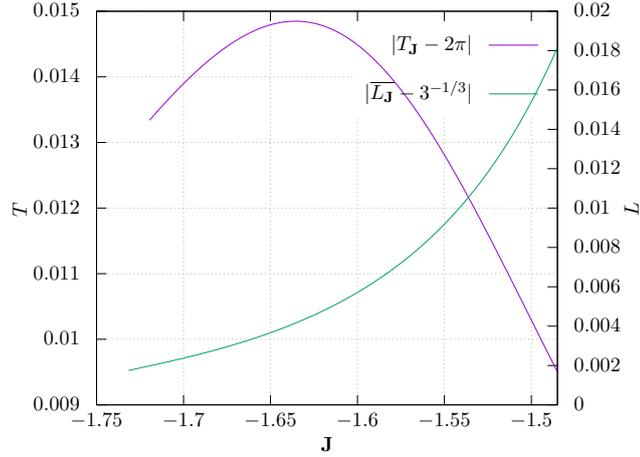}}
	\end{center}
	\caption{Resonant family of periodic orbits. We show the normalized period $T_\tJ - 2\pi$, and the
maximum deviation of the $L_\tJ$ component with respect to the resonant value $3^{-1/3}$.}
	\label{fig:porbits}
\end{figure}

Figure~\ref{fig:porbits} shows the normalized period $T_\tJ - 2\pi$ of the resonant periodic orbits as a function of $\tJ$. From this Figure, we obtain
\[ 9\mu < \abs{T_\tJ - 2\pi} < 15\mu, \]
which is the first bound given in Ansatz~\ref{ans:NHIMCircular:bis}.

Moreover, for each periodic orbit $\lambda_\tJ(t)=(L_\tJ, l_\tJ, G_\tJ, g_\tJ)(t)$, we monitor its $L$ component $L_\tJ(t)$ for $t\in[0,T_\tJ)$, and compute its deviation with respect to the resonant value $L_0=3^{-1/3}$,
\[ \overline{L_\tJ} := \sup_{0 \leq t < T_\tJ} \abs{L_\tJ - L_0}. \]
The result is plotted in Figure~\ref{fig:porbits}. From this Figure, we obtain
\[ \abs{L_\tJ - 3^{-1/3}} \leq 0.018 \approx 19\mu, \]
which is the second bound given in Ansatz~\ref{ans:NHIMCircular:bis}.

Ansatz \ref{ans:NHIMCircular:bis} also claims that the homoclinic orbits are confined in the interval $\abs{L-3^{-1/3}} < 42\mu$. To verify this, fix a homoclinic channel $i\in\{1,2,3,4\}$ and consider the set of homoclinic trajectories $\gamma_i(t,\tJ) = (L_i,l_i,g_i)(t,\tJ)$ for all $\tJ\in \mathcal{\tilde I}_i$. For each value of $\tJ$, i.e. for each homoclinic trajectory, we monitor its $L$ component $L_i(t,\tJ)$ 
as time $t$ evolves from $-M$ to $M$ with large enough $M$.

\begin{remark}
	We use the same value $M$ for the endpoints as in  Remark~\ref{rem:convergenceAlpha}. This guarantees the convergence of the numerical bound.
\end{remark}

To monitor the deviation of $L_i(t,\tJ)$ with respect to the resonant value $L_0=3^{-1/3}$, we compute the numerical bound
\[ \overline{L_i}(\tJ) := \sup_{-M < t < M} \abs{L_i(t,\tJ) - L_0}, \]
which  is plotted as a function of the energy $\tJ$ in Figure~\ref{fig:Lbound}. From this figure we obtain, for any $i$ and for all $J$,  the uniform bound 
\[\abs{L^i_\tJ(\sigma) - L_0} \leq 0.04 \approx 42\mu. \]

\begin{figure}
	\centering
	\resizebox{!}{0.7\height}{\input{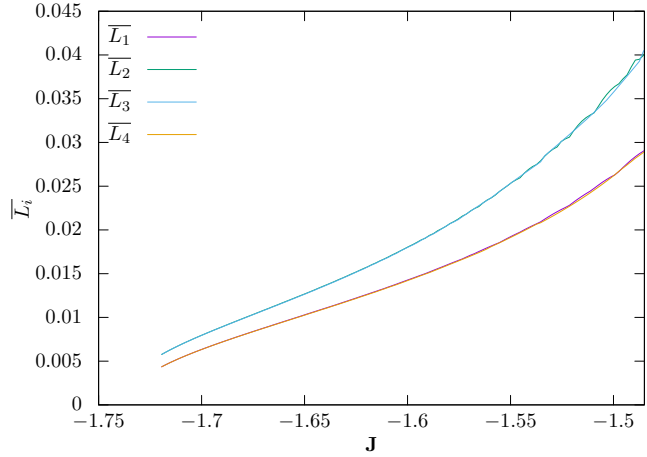}}
	\caption{Maximum deviation of $L_i(t,\tJ)$ with respect to the resonant value $L_0=3^{-1/3}$. Each curve corresponds to a different homoclinic channel $i\in\{1,2,3,4\}$.}
	\label{fig:Lbound}
\end{figure}

\section{Numerical verification of Ansatz \ref{ansatz:Melnikov:1}}\label{sec:nondegeneracy}

In this section we verify the non-degeneracy conditions stated in
Ansatz~\ref{ansatz:Melnikov:1} numerically. The first step is to compute numerically the functions $\alpha_i$ and $B_i$, $i=1,2,3,4$ (one for each homoclinic trajectory), given in Proposition~\ref{prop:Melnikov}. 

\paragraph{The functions $\alpha_i$.}
%
The functions $\alpha_i$ are computed using numerical integration.
In the numerical computation, we approximate the improper integral 
\[
\alpha^\pm_i(I)
=\mu\lim_{N\to\pm\infty}\left(\int_0^{2\pi N}\frac{\left(\pa_G
	\Delta H_\ccirc\right)\circ \ga_I^i(\sigma)}{-1+\mu \left(\pa_G
	\Delta H_\ccirc\right)\circ \ga_I^i(\sigma)}d\sigma+2\pi N\nu(I)\right)
\]
by a sequence of bounded integrals
\[
\alpha^\pm_i(I;N)
=\mu\left(\int_0^{2\pi N}\frac{\left(\pa_G
	\Delta H_\ccirc\right)\circ \ga_I^i(\sigma)}{-1+\mu \left(\pa_G
	\Delta H_\ccirc\right)\circ \ga_I^i(\sigma)}d\sigma+2\pi N\nu(I)\right) \quad \text{ for } N=1,2,\dotsc
\]
\begin{remark}
	The function
	$\partial_G \Delta H_\textrm{circ}$ was given
	explicitly in~\cite{FejozGKR15}, Appendix B.1.
\end{remark}
Furthermore, each bounded integral is decomposed
into the sum of $N$ integrals of equal size:
\[
\alpha^\pm_i(I;N)
=\mu\sum_{k=0}^{N-1}\left(\int_{2\pi k}^{2\pi (k+1)}\frac{\left(\pa_G
	\Delta H_\ccirc\right)\circ \ga_I^i(\sigma)}{-1+\mu \left(\pa_G
	\Delta H_\ccirc\right)\circ \ga_I^i(\sigma)}d\sigma+2\pi \nu(I)\right).\
\]
We use the QAGS adaptative integrator from the GNU Scientific Library~\cite{GSL} to compute each of these $N$
integrals 
\[ I_{i,k}=\int_{2\pi k}^{2\pi (k+1)}\frac{\left(\pa_G
	\Delta H_\ccirc\right)\circ \ga_I^i(\sigma)}{-1+\mu \left(\pa_G
	\Delta H_\ccirc\right)\circ \ga_I^i(\sigma)}d\sigma, \quad k=1,2,\dotsc,N
\]
within a relative error bound $\eps_{\textrm{rel}}=10^{-8}$, in such a way that
the following inequality holds
\[ \abs{\textrm{RESULT} - I_{i,k}} \leq \eps_{\textrm{rel}}\abs{I_{i,k}}, \]
where RESULT is the numerical value of the integral obtained by the algorithm.

\begin{remark}\label{rem:convergenceAlpha}
We always choose a value of $N=N(I)$ large enough so that 
$\abs{\alpha_i^\pm(I;N) - \alpha_i^\pm(I;N+1)} \leq 10^{-6}$. This convergence is
illustrated in Figure~\ref{fig:omegatest}.
\end{remark}

\begin{figure}[h]
\begin{center}
\resizebox{!}{0.7\height}{\begin{tikzpicture}[gnuplot]
\path (0.000,0.000) rectangle (12.500,8.750);
\gpcolor{color=gp lt color border}
\gpsetlinetype{gp lt border}
\gpsetdashtype{gp dt solid}
\gpsetlinewidth{1.00}
\draw[gp path] (1.688,0.985)--(1.868,0.985);
\draw[gp path] (11.947,0.985)--(11.767,0.985);
\node[gp node right] at (1.504,0.985) {$0$};
\draw[gp path] (1.688,2.476)--(1.868,2.476);
\draw[gp path] (11.947,2.476)--(11.767,2.476);
\node[gp node right] at (1.504,2.476) {$0.005$};
\draw[gp path] (1.688,3.967)--(1.868,3.967);
\draw[gp path] (11.947,3.967)--(11.767,3.967);
\node[gp node right] at (1.504,3.967) {$0.01$};
\draw[gp path] (1.688,5.459)--(1.868,5.459);
\draw[gp path] (11.947,5.459)--(11.767,5.459);
\node[gp node right] at (1.504,5.459) {$0.015$};
\draw[gp path] (1.688,6.950)--(1.868,6.950);
\draw[gp path] (11.947,6.950)--(11.767,6.950);
\node[gp node right] at (1.504,6.950) {$0.02$};
\draw[gp path] (1.688,8.441)--(1.868,8.441);
\draw[gp path] (11.947,8.441)--(11.767,8.441);
\node[gp node right] at (1.504,8.441) {$0.025$};
\draw[gp path] (1.688,0.985)--(1.688,1.165);
\draw[gp path] (1.688,8.441)--(1.688,8.261);
\node[gp node center] at (1.688,0.677) {$0$};
\draw[gp path] (3.154,0.985)--(3.154,1.165);
\draw[gp path] (3.154,8.441)--(3.154,8.261);
\node[gp node center] at (3.154,0.677) {$10$};
\draw[gp path] (4.619,0.985)--(4.619,1.165);
\draw[gp path] (4.619,8.441)--(4.619,8.261);
\node[gp node center] at (4.619,0.677) {$20$};
\draw[gp path] (6.085,0.985)--(6.085,1.165);
\draw[gp path] (6.085,8.441)--(6.085,8.261);
\node[gp node center] at (6.085,0.677) {$30$};
\draw[gp path] (7.550,0.985)--(7.550,1.165);
\draw[gp path] (7.550,8.441)--(7.550,8.261);
\node[gp node center] at (7.550,0.677) {$40$};
\draw[gp path] (9.016,0.985)--(9.016,1.165);
\draw[gp path] (9.016,8.441)--(9.016,8.261);
\node[gp node center] at (9.016,0.677) {$50$};
\draw[gp path] (10.481,0.985)--(10.481,1.165);
\draw[gp path] (10.481,8.441)--(10.481,8.261);
\node[gp node center] at (10.481,0.677) {$60$};
\draw[gp path] (11.947,0.985)--(11.947,1.165);
\draw[gp path] (11.947,8.441)--(11.947,8.261);
\node[gp node center] at (11.947,0.677) {$70$};
\draw[gp path] (1.688,8.441)--(1.688,0.985)--(11.947,0.985)--(11.947,8.441)--cycle;
\gpcolor{rgb color={0.580,0.000,0.827}}
\draw[gp path] (11.507,0.987)--(11.361,0.987)--(11.214,0.987)--(11.068,0.987)--(10.921,0.988)%
  --(10.775,0.988)--(10.628,0.988)--(10.481,0.989)--(10.335,0.989)--(10.188,0.990)--(10.042,0.990)%
  --(9.895,0.991)--(9.749,0.992)--(9.602,0.992)--(9.456,0.993)--(9.309,0.994)--(9.162,0.996)%
  --(9.016,0.997)--(8.869,0.998)--(8.723,1.000)--(8.576,1.002)--(8.430,1.004)--(8.283,1.006)%
  --(8.137,1.009)--(7.990,1.012)--(7.843,1.016)--(7.697,1.020)--(7.550,1.024)--(7.404,1.029)%
  --(7.257,1.035)--(7.111,1.041)--(6.964,1.049)--(6.818,1.057)--(6.671,1.067)--(6.524,1.078)%
  --(6.378,1.091)--(6.231,1.106)--(6.085,1.123)--(5.938,1.142)--(5.792,1.164)--(5.645,1.190)%
  --(5.498,1.219)--(5.352,1.253)--(5.205,1.292)--(5.059,1.337)--(4.912,1.387)--(4.766,1.445)%
  --(4.619,1.511)--(4.473,1.585)--(4.326,1.667)--(4.179,1.758)--(4.033,1.856)--(3.886,1.961)%
  --(3.740,2.071)--(3.593,2.183)--(3.447,2.296)--(3.300,2.409)--(3.154,2.525)--(3.007,2.658)%
  --(2.860,2.831)--(2.714,3.083)--(2.567,3.472)--(2.421,4.059)--(2.274,4.895)--(2.128,5.983)%
  --(1.981,7.233)--(1.835,8.431);
\gpsetpointsize{4.00}
\gp3point{gp mark 1}{}{(11.507,0.987)}
\gp3point{gp mark 1}{}{(11.361,0.987)}
\gp3point{gp mark 1}{}{(11.214,0.987)}
\gp3point{gp mark 1}{}{(11.068,0.987)}
\gp3point{gp mark 1}{}{(10.921,0.988)}
\gp3point{gp mark 1}{}{(10.775,0.988)}
\gp3point{gp mark 1}{}{(10.628,0.988)}
\gp3point{gp mark 1}{}{(10.481,0.989)}
\gp3point{gp mark 1}{}{(10.335,0.989)}
\gp3point{gp mark 1}{}{(10.188,0.990)}
\gp3point{gp mark 1}{}{(10.042,0.990)}
\gp3point{gp mark 1}{}{(9.895,0.991)}
\gp3point{gp mark 1}{}{(9.749,0.992)}
\gp3point{gp mark 1}{}{(9.602,0.992)}
\gp3point{gp mark 1}{}{(9.456,0.993)}
\gp3point{gp mark 1}{}{(9.309,0.994)}
\gp3point{gp mark 1}{}{(9.162,0.996)}
\gp3point{gp mark 1}{}{(9.016,0.997)}
\gp3point{gp mark 1}{}{(8.869,0.998)}
\gp3point{gp mark 1}{}{(8.723,1.000)}
\gp3point{gp mark 1}{}{(8.576,1.002)}
\gp3point{gp mark 1}{}{(8.430,1.004)}
\gp3point{gp mark 1}{}{(8.283,1.006)}
\gp3point{gp mark 1}{}{(8.137,1.009)}
\gp3point{gp mark 1}{}{(7.990,1.012)}
\gp3point{gp mark 1}{}{(7.843,1.016)}
\gp3point{gp mark 1}{}{(7.697,1.020)}
\gp3point{gp mark 1}{}{(7.550,1.024)}
\gp3point{gp mark 1}{}{(7.404,1.029)}
\gp3point{gp mark 1}{}{(7.257,1.035)}
\gp3point{gp mark 1}{}{(7.111,1.041)}
\gp3point{gp mark 1}{}{(6.964,1.049)}
\gp3point{gp mark 1}{}{(6.818,1.057)}
\gp3point{gp mark 1}{}{(6.671,1.067)}
\gp3point{gp mark 1}{}{(6.524,1.078)}
\gp3point{gp mark 1}{}{(6.378,1.091)}
\gp3point{gp mark 1}{}{(6.231,1.106)}
\gp3point{gp mark 1}{}{(6.085,1.123)}
\gp3point{gp mark 1}{}{(5.938,1.142)}
\gp3point{gp mark 1}{}{(5.792,1.164)}
\gp3point{gp mark 1}{}{(5.645,1.190)}
\gp3point{gp mark 1}{}{(5.498,1.219)}
\gp3point{gp mark 1}{}{(5.352,1.253)}
\gp3point{gp mark 1}{}{(5.205,1.292)}
\gp3point{gp mark 1}{}{(5.059,1.337)}
\gp3point{gp mark 1}{}{(4.912,1.387)}
\gp3point{gp mark 1}{}{(4.766,1.445)}
\gp3point{gp mark 1}{}{(4.619,1.511)}
\gp3point{gp mark 1}{}{(4.473,1.585)}
\gp3point{gp mark 1}{}{(4.326,1.667)}
\gp3point{gp mark 1}{}{(4.179,1.758)}
\gp3point{gp mark 1}{}{(4.033,1.856)}
\gp3point{gp mark 1}{}{(3.886,1.961)}
\gp3point{gp mark 1}{}{(3.740,2.071)}
\gp3point{gp mark 1}{}{(3.593,2.183)}
\gp3point{gp mark 1}{}{(3.447,2.296)}
\gp3point{gp mark 1}{}{(3.300,2.409)}
\gp3point{gp mark 1}{}{(3.154,2.525)}
\gp3point{gp mark 1}{}{(3.007,2.658)}
\gp3point{gp mark 1}{}{(2.860,2.831)}
\gp3point{gp mark 1}{}{(2.714,3.083)}
\gp3point{gp mark 1}{}{(2.567,3.472)}
\gp3point{gp mark 1}{}{(2.421,4.059)}
\gp3point{gp mark 1}{}{(2.274,4.895)}
\gp3point{gp mark 1}{}{(2.128,5.983)}
\gp3point{gp mark 1}{}{(1.981,7.233)}
\gp3point{gp mark 1}{}{(1.835,8.431)}
\gpcolor{color=gp lt color border}
\draw[gp path] (1.688,8.441)--(1.688,0.985)--(11.947,0.985)--(11.947,8.441)--cycle;
\node[gp node center,rotate=-270.0] at (0.292,4.713) {$\abs{\alpha_1^+(N) - \alpha_1^+(N-1)}$};
\node[gp node center] at (6.817,0.215) {$N$};
\gpdefrectangularnode{gp plot 1}{\pgfpoint{1.688cm}{0.985cm}}{\pgfpoint{11.947cm}{8.441cm}}
\end{tikzpicture}
}
\end{center}
    \caption{Convergence of the improper integral $\alpha_1^+$ as a function of
    $N$ for the energy value $\tJ=-1.719$. The approximation error is 
    $|\alpha_1^+(67) - \alpha_1^+(68)|<10^{-6}$, while the value of the
    integral is of the order of $10^{-1}$ (see Figure~\ref{fig:omegaj}).}
\label{fig:omegatest}
\end{figure}
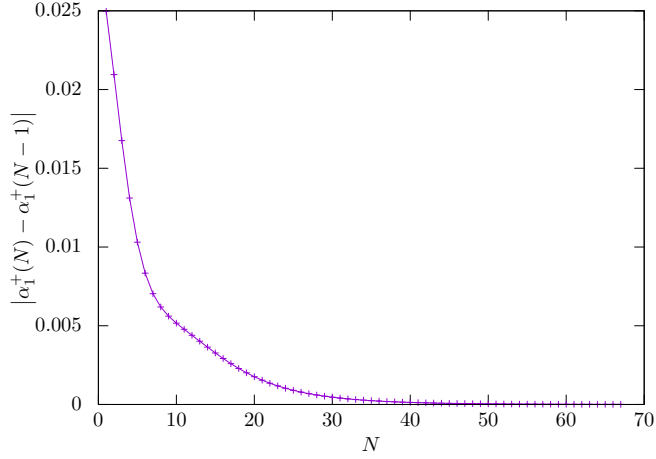

For numerical stability, on the unstable manifold we always integrate the
homoclinic orbit \emph{forward} in time, starting from a homoclinic point on
the local unstable manifold. Analogously, on the stable manifold we always
integrate the homoclinic orbit \emph{backward} in time, starting from a point
on the local stable manifold.

\begin{remark} By the reversibility \eqref{eq:symmetry},  $\alpha_i^-(I) = -\alpha_i^+(I)$,
    and thus we have $\alpha_i(I) = 2\alpha_i^+(I)$.
\end{remark}

The functions $\alpha_i$ are plotted in Figure~\ref{fig:omegaj} for $i=1,2,3,4$.
Two of them coincide at the point $I \approx -1.608$.
Note that no more than two phase shifts coincide for a given energy value.

\begin{figure}
	\centering
		\resizebox{!}{0.7\height}{\input{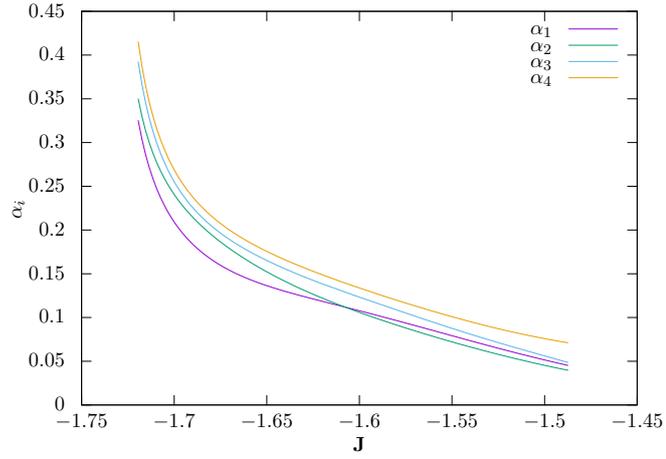}}
	\caption{Functions $\alpha_i(\tJ)$. 
	There is a crossing at energy value $\tJ\approx -1.608$.}
	\label{fig:omegaj}
\end{figure}

\paragraph{The functions $B_i$.} 
Next we find the functions $B_i^\inn$, $B_i^\out$, and $B_i$ introduced in
Proposition~\ref{prop:Melnikov}. Since these functions are complex-valued, we
compute their real and imaginary parts separately. They are plotted in Figures~\ref{fig:Bin}, \ref{fig:Bout} and \ref{fig:Bi} respectively.

\begin{figure}
	\centering
	\resizebox{!}{0.7\height}{\begin{tikzpicture}[gnuplot]
\path (0.000,0.000) rectangle (12.500,8.750);
\gpcolor{color=gp lt color border}
\gpsetlinetype{gp lt border}
\gpsetdashtype{gp dt solid}
\gpsetlinewidth{1.00}
\draw[gp path] (1.564,0.985)--(1.744,0.985);
\draw[gp path] (11.947,0.985)--(11.767,0.985);
\node[gp node right] at (1.380,0.985) {$-0.008$};
\draw[gp path] (1.564,1.917)--(1.744,1.917);
\draw[gp path] (11.947,1.917)--(11.767,1.917);
\node[gp node right] at (1.380,1.917) {$-0.006$};
\draw[gp path] (1.564,2.849)--(1.744,2.849);
\draw[gp path] (11.947,2.849)--(11.767,2.849);
\node[gp node right] at (1.380,2.849) {$-0.004$};
\draw[gp path] (1.564,3.781)--(1.744,3.781);
\draw[gp path] (11.947,3.781)--(11.767,3.781);
\node[gp node right] at (1.380,3.781) {$-0.002$};
\draw[gp path] (1.564,4.713)--(1.744,4.713);
\draw[gp path] (11.947,4.713)--(11.767,4.713);
\node[gp node right] at (1.380,4.713) {$0$};
\draw[gp path] (1.564,5.645)--(1.744,5.645);
\draw[gp path] (11.947,5.645)--(11.767,5.645);
\node[gp node right] at (1.380,5.645) {$0.002$};
\draw[gp path] (1.564,6.577)--(1.744,6.577);
\draw[gp path] (11.947,6.577)--(11.767,6.577);
\node[gp node right] at (1.380,6.577) {$0.004$};
\draw[gp path] (1.564,7.509)--(1.744,7.509);
\draw[gp path] (11.947,7.509)--(11.767,7.509);
\node[gp node right] at (1.380,7.509) {$0.006$};
\draw[gp path] (1.564,8.441)--(1.744,8.441);
\draw[gp path] (11.947,8.441)--(11.767,8.441);
\node[gp node right] at (1.380,8.441) {$0.008$};
\draw[gp path] (1.564,0.985)--(1.564,1.165);
\draw[gp path] (1.564,8.441)--(1.564,8.261);
\node[gp node center] at (1.564,0.677) {$-1.75$};
\draw[gp path] (3.295,0.985)--(3.295,1.165);
\draw[gp path] (3.295,8.441)--(3.295,8.261);
\node[gp node center] at (3.295,0.677) {$-1.7$};
\draw[gp path] (5.025,0.985)--(5.025,1.165);
\draw[gp path] (5.025,8.441)--(5.025,8.261);
\node[gp node center] at (5.025,0.677) {$-1.65$};
\draw[gp path] (6.756,0.985)--(6.756,1.165);
\draw[gp path] (6.756,8.441)--(6.756,8.261);
\node[gp node center] at (6.756,0.677) {$-1.6$};
\draw[gp path] (8.486,0.985)--(8.486,1.165);
\draw[gp path] (8.486,8.441)--(8.486,8.261);
\node[gp node center] at (8.486,0.677) {$-1.55$};
\draw[gp path] (10.217,0.985)--(10.217,1.165);
\draw[gp path] (10.217,8.441)--(10.217,8.261);
\node[gp node center] at (10.217,0.677) {$-1.5$};
\draw[gp path] (11.947,0.985)--(11.947,1.165);
\draw[gp path] (11.947,8.441)--(11.947,8.261);
\node[gp node center] at (11.947,0.677) {$-1.45$};
\draw[gp path] (1.564,8.441)--(1.564,0.985)--(11.947,0.985)--(11.947,8.441)--cycle;
\node[gp node right] at (3.404,7.953) {$\Re{B_1^\inn}$};
\gpcolor{rgb color={0.580,0.000,0.827}}
\draw[gp path] (3.588,7.953)--(4.504,7.953);
\draw[gp path] (2.623,5.262)--(2.692,4.724)--(2.762,4.784)--(2.831,4.803)--(2.900,4.812)%
  --(2.969,4.819)--(3.038,4.824)--(3.108,4.828)--(3.177,4.832)--(3.246,4.835)--(3.315,4.838)%
  --(3.384,4.842)--(3.454,4.845)--(3.523,4.849)--(3.592,4.855)--(3.661,4.862)--(3.731,4.873)%
  --(3.800,4.899)--(3.869,5.083)--(3.938,4.739)--(4.007,4.804)--(4.077,4.822)--(4.146,4.831)%
  --(4.215,4.837)--(4.284,4.841)--(4.354,4.845)--(4.423,4.848)--(4.492,4.851)--(4.561,4.853)%
  --(4.630,4.856)--(4.700,4.858)--(4.769,4.860)--(4.838,4.862)--(4.907,4.864)--(4.977,4.866)%
  --(5.046,4.868)--(5.115,4.869)--(5.184,4.871)--(5.253,4.872)--(5.323,4.874)--(5.392,4.875)%
  --(5.461,4.877)--(5.530,4.878)--(5.600,4.879)--(5.669,4.880)--(5.738,4.881)--(5.807,4.882)%
  --(5.876,4.882)--(5.946,4.883)--(6.015,4.883)--(6.084,4.884)--(6.153,4.884)--(6.223,4.884)%
  --(6.292,4.884)--(6.361,4.884)--(6.430,4.883)--(6.499,4.883)--(6.569,4.882)--(6.638,4.881)%
  --(6.707,4.879)--(6.776,4.876)--(6.845,4.871)--(6.915,4.860)--(6.984,4.815)--(7.053,5.028)%
  --(7.122,4.923)--(7.192,4.908)--(7.261,4.901)--(7.330,4.897)--(7.399,4.894)--(7.468,4.892)%
  --(7.538,4.889)--(7.607,4.887)--(7.676,4.885)--(7.745,4.883)--(7.815,4.880)--(7.884,4.875)%
  --(7.953,4.865)--(8.022,4.794)--(8.091,4.914)--(8.161,4.894)--(8.230,4.887)--(8.299,4.883)%
  --(8.368,4.880)--(8.438,4.877)--(8.507,4.874)--(8.576,4.871)--(8.645,4.866)--(8.714,4.858)%
  --(8.784,4.762)--(8.853,4.886)--(8.922,4.874)--(8.991,4.869)--(9.061,4.866)--(9.130,4.863)%
  --(9.199,4.860)--(9.268,4.856)--(9.337,4.852)--(9.407,4.841)--(9.476,4.894)--(9.545,4.860)%
  --(9.614,4.854)--(9.684,4.850)--(9.753,4.847)--(9.822,4.844)--(9.891,4.841)--(9.960,4.838)%
  --(10.030,4.830)--(10.099,4.882)--(10.168,4.842)--(10.237,4.837)--(10.306,4.833)--(10.376,4.830)%
  --(10.445,4.828)--(10.514,4.825)--(10.583,4.822);
\gpcolor{color=gp lt color border}
\node[gp node right] at (3.404,7.337) {$\Im{B_1^\inn}$};
\gpcolor{rgb color={0.000,0.620,0.451}}
\draw[gp path] (3.588,7.337)--(4.504,7.337);
\draw[gp path] (2.623,1.934)--(2.692,5.413)--(2.762,5.028)--(2.831,4.912)--(2.900,4.852)%
  --(2.969,4.812)--(3.038,4.783)--(3.108,4.758)--(3.177,4.736)--(3.246,4.714)--(3.315,4.692)%
  --(3.384,4.667)--(3.454,4.638)--(3.523,4.603)--(3.592,4.555)--(3.661,4.482)--(3.731,4.352)%
  --(3.800,4.034)--(3.869,1.708)--(3.938,6.116)--(4.007,5.298)--(4.077,5.084)--(4.146,4.983)%
  --(4.215,4.923)--(4.284,4.882)--(4.354,4.852)--(4.423,4.828)--(4.492,4.808)--(4.561,4.792)%
  --(4.630,4.777)--(4.700,4.764)--(4.769,4.753)--(4.838,4.742)--(4.907,4.733)--(4.977,4.724)%
  --(5.046,4.716)--(5.115,4.709)--(5.184,4.703)--(5.253,4.698)--(5.323,4.694)--(5.392,4.691)%
  --(5.461,4.689)--(5.530,4.688)--(5.600,4.688)--(5.669,4.690)--(5.738,4.693)--(5.807,4.697)%
  --(5.876,4.703)--(5.946,4.710)--(6.015,4.718)--(6.084,4.728)--(6.153,4.739)--(6.223,4.751)%
  --(6.292,4.766)--(6.361,4.782)--(6.430,4.802)--(6.499,4.825)--(6.569,4.853)--(6.638,4.889)%
  --(6.707,4.937)--(6.776,5.008)--(6.845,5.127)--(6.915,5.380)--(6.984,6.379)--(7.053,1.616)%
  --(7.122,3.957)--(7.192,4.309)--(7.261,4.457)--(7.330,4.543)--(7.399,4.603)--(7.468,4.649)%
  --(7.538,4.690)--(7.607,4.729)--(7.676,4.769)--(7.745,4.816)--(7.815,4.880)--(7.884,4.983)%
  --(7.953,5.224)--(8.022,7.145)--(8.091,3.851)--(8.161,4.380)--(8.230,4.534)--(8.299,4.615)%
  --(8.368,4.672)--(8.438,4.721)--(8.507,4.770)--(8.576,4.830)--(8.645,4.924)--(8.714,5.146)%
  --(8.784,8.226)--(8.853,4.153)--(8.922,4.486)--(8.991,4.598)--(9.061,4.664)--(9.130,4.716)%
  --(9.199,4.766)--(9.268,4.828)--(9.337,4.935)--(9.407,5.281)--(9.476,3.171)--(9.545,4.414)%
  --(9.614,4.576)--(9.684,4.650)--(9.753,4.702)--(9.822,4.748)--(9.891,4.803)--(9.960,4.890)%
  --(10.030,5.143)--(10.099,2.665)--(10.168,4.433)--(10.237,4.588)--(10.306,4.656)--(10.376,4.702)%
  --(10.445,4.743)--(10.514,4.790)--(10.583,4.864);
\gpcolor{color=gp lt color border}
\draw[gp path] (1.564,8.441)--(1.564,0.985)--(11.947,0.985)--(11.947,8.441)--cycle;
\node[gp node center] at (6.755,0.215) {$\tJ$};
\gpdefrectangularnode{gp plot 1}{\pgfpoint{1.564cm}{0.985cm}}{\pgfpoint{11.947cm}{8.441cm}}
\end{tikzpicture}
}
	\caption{Function $B_1^\inn(\tJ)$. Functions $B_i^\inn$ for $i=2,3,4$ are very similar; they are not shown for simplicity.}
	\label{fig:Bin}
\end{figure}
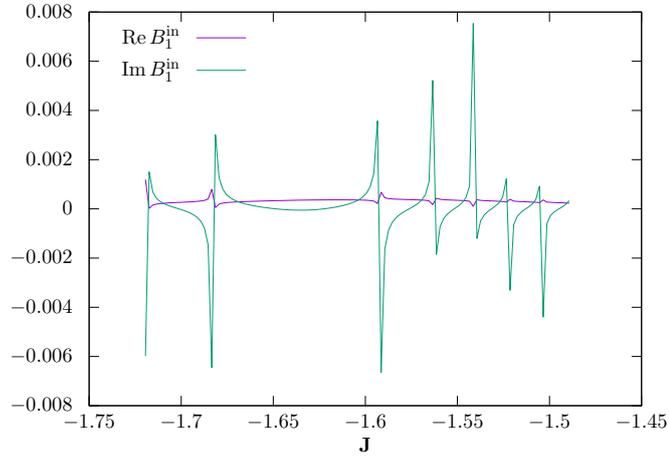

\begin{figure}
	\centering
	\resizebox{!}{0.7\height}{\input{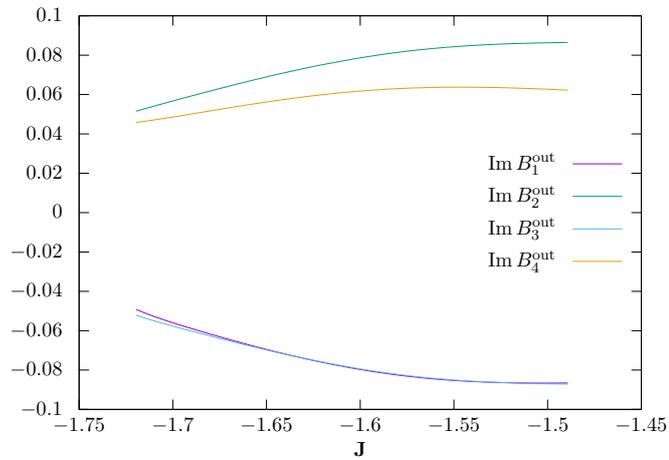}}
	\caption{Functions $B_i^\out(\tJ)$, $i=1,2,3,4$.  $\Re B_i^\out$ are all zero. Only imaginary part is plotted.}
	\label{fig:Bout}
\end{figure}

\begin{figure}
	\centering
	\resizebox{!}{0.7\height}{\input{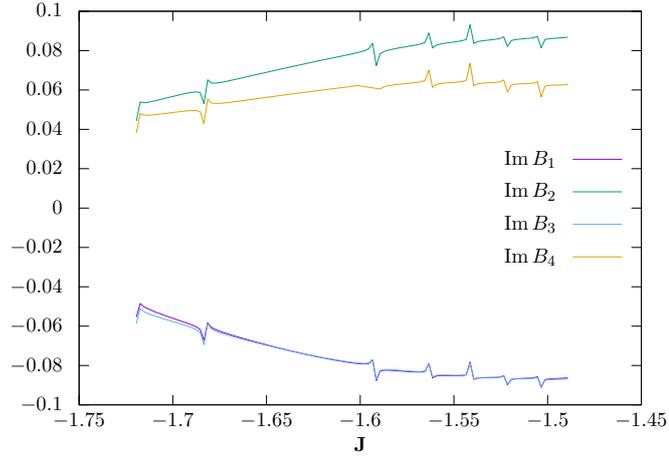}}
	\caption{Functions $B_i(\tJ)$, $i=1,2,3,4$. $\Re B_i$ are all. Only imaginary part is plotted.}
	\label{fig:Bi}
\end{figure}

\begin{remark}
	The function $\Delta H^{1,+}_\textrm{ell}$ involved in the numerator of these integrals was given explicitly in~\cite{FejozGKR15}, equation (102).
\end{remark}

In the numerical computation of the improper integrals $B_i^\out$ we
use the same techniques as for $\alpha_i^\pm$.
Again, these integrals are computed within a relative error bound $10^{-8}$.

\paragraph{The first order of the variance $\bms_0^2(I)$.}

Finally, we compute $\bms_0^2$, the first order of the variance, given in Ansatz~\ref{ansatz:Melnikov:1},
\begin{equation}\label{eq:variance}
\bms_0^2(I)=2\E_\omega\left|B_i(I)-\E_\omega B_i(I)\frac{1-e^{
		\im\beta^0_i}(I)}{1-\E_\omega\left(e^{
		\im\beta^0_i(I)}\right)}\right|^2,
\end{equation}
where
\begin{equation}\label{eq:beta}
\beta^0_i(I)=\nu(I)\ol g+\alpha_i(I).
\end{equation}

For definiteness, let us use the pair of homoclinic channels with symbols $i=2,3$. These are proper homoclinic channels on the interval $\mathcal{I}_1=[-1.551,-1.485]$, in the sense that they are free of tangencies in that interval; see Appendix~\ref{sec:transversality}.

Notice that functions $\alpha_i$ and $B_i$ in equations~\eqref{eq:variance} and ~\eqref{eq:beta} have already been computed. 
Further, recall that the operator $\E_\omega$ simply denotes the arithmetic mean with respect to $i=2$ and $i=3$. For example: $\E_{\omega} B_i(I) = \frac{1}{2}(B_2(I) + B_3(I))$.
Only the term $\theta(I) := \nu(I)\ol g$ is unknown. 
Thus we decide to compute the first order of the variance $\bms_0^2(I, \theta(I))$ for all possible values of $\theta\in[0,2\pi)$. See the result in Figure~\ref{fig:variance}. Clearly, $\bms_0^2$ does not vanish on the diffusion interval $\mathcal{I}_1$ (for any $\theta$ value).

\begin{figure}
	\centering
	\resizebox{!}{0.8\height}{\input{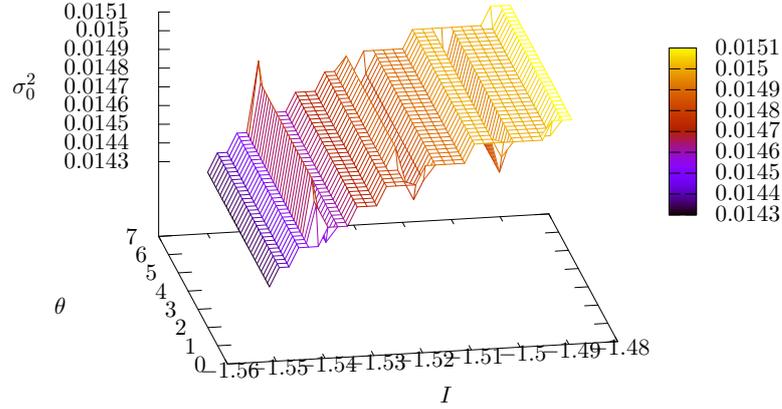}}
	\caption{First order of the variance $\sigma_0^2(I, \theta)$.}
	\label{fig:variance}
\end{figure}

\section{Numerical bounds for $\ecc(t)$ along the homoclinic channels}\label{sec:bounds}

For the proof of Theorem 1.2 (see Section 5.4 of the companion paper \cite{stochastic23}), we want to show that 
$
\EE_1^1(X,t) = E(\tJ,L_i(t,\tJ)) - E(\tJ,L_0)
$
is $\OO(\mu)$, where the function $E$ is defined as
\[
E(\tJ,L)=\sqrt{1-\frac{(\tJ+\frac{1}{2L^2})^2}{L^2}}.
\]
For each value $\tJ$, we monitor the $L$ component $L_i(t,\tJ)$ along the homoclinic, and compute the numerical bound
\[ \overline{\EE_1^1}(\tJ) := \sup_{-M < t < M} \abs{E(\tJ,L_i(t,\tJ)) - E(\tJ,L_0)}. \]
The function $\overline{\EE_1^1}(\tJ)$ is plotted in Figure~\ref{fig:E11_bound}. 
From this figure we obtain the uniform bound 
\[\abs{E(\tJ,L_i(t,\tJ)) - E(\tJ,L_0)} \leq 0.12 \approx 126\mu \]
for any $i$ and for all $\tJ$ in any of the intervals in Table \ref{tab:intervals}.

\begin{figure}
	\centering
	\resizebox{!}{0.7\height}{\input{figs/E11_bound}}
	\caption{The function $\overline{\EE_1^1}(J)$. Each curve corresponds to a different homoclinic channel $i\in\{1,2,3,4\}$.}
	\label{fig:E11_bound}
\end{figure}

We also need to show that 
$E_2(X) =  \sqrt{1-\frac{G^2}{L^2}} - E(\tJ,L)$
is $\OO(\mu)$.
For each value $\tJ$, we monitor the $L$ and $G$ components $L_i(t,\tJ), G_i(t,\tJ)$ along the homoclinic, and compute the numerical bound
\[ \overline{E_2}(\tJ) := \sup_{-M < t < M} \abs{\sqrt{1-\frac{G^2}{L^2}} - E(\tJ,L)}. \]
The function $\overline{E_2}(\tJ)$ is plotted in Figure~\ref{fig:E2_bound}. 
From this figure we obtain the uniform bound 
\[\abs{\sqrt{1-\frac{G^2}{L^2}} - E(\tJ,L)} \leq 0.12 \approx 126\mu \]
for any $i$ and for all $\tJ$  in any of the intervals in Table \ref{tab:intervals}.

\begin{figure}
	\centering
	\resizebox{!}{0.7\height}{\input{figs/E2_bound}}
	\caption{The function $\overline{E_2}(\tJ)$. Each curve corresponds to a different homoclinic channel $i\in\{1,2,3,4\}$.}
	\label{fig:E2_bound}
\end{figure}

\bibliographystyle{alpha}
\bibliography{references}
\end{document}